\documentclass[preprint,12pt]{article}
\usepackage{amssymb}
\usepackage{mathrsfs}
\usepackage{amsfonts}
\usepackage{graphicx}
\usepackage{colortbl,dcolumn}
\usepackage{amsmath}
\usepackage{caption}
\usepackage{psfrag}
\usepackage{booktabs}
\usepackage{cite}
\usepackage{appendix}
\usepackage{amsthm}
\usepackage{subfigure}
\usepackage{comment}
\numberwithin{equation}{section}
\usepackage[top=0.9in, bottom=0.9in, left=0.8in, right=0.8in]{geometry}

\usepackage{hyperref}
\hypersetup{hypertex=true,colorlinks=true,
linkcolor=blue,
anchorcolor=red,
citecolor=red}

\usepackage{soul}

\newcommand{\R}{\mathbb{R}}
\newcommand{\N}{\mathbb{N}}
\newcommand{\E}{\mathbb{E}}
\renewcommand{\P}{\mathbb{P}}
\newcommand{\tr}{\operatorname{trace}}

\newcommand{\dd}{\text{d}}

\newtheorem{theorem}{Theorem}[section]
\newtheorem{definition}[theorem]{Definition}
\newtheorem{lemma}[theorem]{Lemma}
\newtheorem{proposition}[theorem]{Proposition}

\newtheorem{remark}[theorem]{Remark}
\newtheorem{example}[theorem]{Example}
\newtheorem{assumption}[theorem]{Assumption}

\begin{document}
\title{
Projected Langevin Monte Carlo algorithms in non-convex and super-linear setting\footnotemark[1]
}

\author{
Chenxu Pang$^{\text{1}}$, \ Xiaojie Wang$^{1,}$\footnotemark[2] and Yue Wu$^{2}$
\\
\footnotesize $^{1}$ School of Mathematics and Statistics, HNP-LAMA, Central South University, Changsha, Hunan, P. R. China\\
\footnotesize $^{2}$ Department of Mathematics and Statistics, University of Strathclyde, Glasgow G1 1XH, UK \\
}

       \maketitle

       \footnotetext{\footnotemark[1] This work was supported by Natural Science Foundation of China 
       {\color{black} (12471394, 12071488, 12371417)}
                and {\color{black} China Scholarship Council (202306370290)}.    
        The authors 
want to thank Xiang Li for his helpful suggestions on the analysis in Section 6 and Bin Yang and Xiaoyan Zhang for their useful comments based on carefully reading 
the manuscript. This work was conducted as part of the first author’s joint PhD program at the National University of Singapore (NUS), supported from the China 
Scholarship Council (CSC), and the author sincerely appreciates the support provided by both institutions.
                }

            \footnotetext{
                \footnotemark[2]Corresponding author.
                }
    
        \footnotetext{E-mail address: c.x.pang@csu.edu.cn,\ x.j.wang7@csu.edu.cn,\ 
        yue.wu@strath.ac.uk
        }
        %
        %
       \begin{abstract}
       It is of significant interest in many applications to sample from a high-dimensional target distribution 
       $\pi$ with the density $\pi(\dd x) \propto e^{-U(x)} (\dd x) $, based on the temporal discretization of the Langevin stochastic differential equations (SDEs).
       In this paper, we propose an explicit projected Langevin Monte Carlo (PLMC) algorithm with  non-convex potential $U$ and super-linear gradient of $U$
       and investigate the non-asymptotic analysis of its sampling error in total variation distance.
       Equipped with time-independent regularity estimates for the associated Kolmogorov equation,
       we derive the non-asymptotic bounds on the total variation distance between the target distribution of the Langevin SDEs and the law induced by the PLMC scheme with order 
       $\mathcal{O}(d^{\max\{3\gamma/2 , 2\gamma-1 \}} h |\ln h|)$,
       where $d$ is the dimension of the target distribution
       and $\gamma \geq 1$ characterizes the growth of the gradient of $U$.
       In addition,
       if the gradient of $U$ is globally Lipschitz continuous,
       an improved convergence order of $\mathcal{O}(d^{3/2} h)$ for the classical Langevin Monte Carlo (LMC) scheme is derived with a refinement of the proof based on Malliavin calculus techniques.
       To achieve a given precision $\epsilon$,
       the smallest number of iterations of
       the PLMC algorithm is proved to be of order ${\mathcal{O}}\big(\tfrac{d^{\max\{3\gamma/2 , 2\gamma-1 \}}}{\epsilon} \cdot \ln (\tfrac{d}{\epsilon}) \cdot \ln (\tfrac{1}{\epsilon}) \big)$.
       In particular,
       the classical Langevin Monte Carlo (LMC) scheme with the non-convex potential $U$ and the globally Lipschitz gradient of $U$ can be guaranteed by order ${\mathcal{O}}\big(\tfrac{d^{3/2}}{\epsilon}  \cdot \ln (\tfrac{1}{\epsilon}) \big)$.
       Numerical experiments are provided to confirm the theoretical findings.


\end{abstract}

%
\section{Introduction}	\label{section: introduction}
Sampling from a high-dimensional ($d \gg 1$) target distribution $\pi$  plays a crucial role in various fields such as Bayesian inference, statistical physics, machine learning and computational biology and has been a subject of recent intensive research efforts.
For example, evaluating the expectation of some functional $\phi$ with respect to  $\pi$:
\begin{equation} \label{introduction:invariant-pi}
    \pi(\phi):=\int_{\mathbb{R}^d} \phi(x) \pi(\dd x),
\end{equation}
is of great interests in applications in the area of 
Bayesian statistics.
A typical approach of sampling is to set the target measure $\pi$ as the invariant measure of the stochastic differential equations (SDEs) and undertake an appropriate numerical scheme that discretizes such SDEs in time.
More precisely,
we consider a class of sampling methods based on the following overdamped Langevin stochastic differential equations (SDEs) of It\^o type:
\begin{equation} \label{eq:langevin-SODE}
    \dd X_t  =  
    - \nabla U ( X_t ) \, \dd t 
    + \sqrt{2} \,\dd W_t,
    \quad
   X_0 = x_0,
   \quad 
       t >0,
\end{equation}
{\color{black}
where 
$W_{\cdot} = \left(W_{1,\cdot}, \dots, W_{d,\cdot} 
\right)^{T}:[0, \infty) \times \Omega \rightarrow \mathbb{R}^{d}$ denotes the $\mathbb{R}^{d}$-valued standard Brownian motion} with respect to $\left\{\mathcal{F}_{t}\right\}_{t \in [0, T]}$
and
{\color{black}
the initial datum
}
$x_{0}: \Omega \rightarrow \mathbb{R}^{d}$ is assumed to be $\mathcal{F}_{0}$-measurable.
Under certain conditions, 
the dynamics of \eqref{eq:langevin-SODE} 
{\color{black}
admits a unique invariant distribution
}$\pi$ 
{\color{black}
with the density proportional to $x \mapsto e^{-U(x)}$.
}
To asymptotically sample from $\pi$,
{\color{black}
 which is the the invariant distribution of \eqref{eq:langevin-SODE},
}
one notable example is the unadjusted Langevin Monte Carlo (LMC for short) algorithm, which corresponds to the well-known Euler-Maruyama scheme of the  Langevin SDE \eqref{eq:langevin-SODE}, 
given by 
\begin{equation} \label{equation:numerical-scheme-euler-maruyama-paper3}
\widetilde{Y}_{n+1} = 
\widetilde{Y}_{n} 
-  \nabla U
\big(\widetilde{Y}_{n}\big) h 
+ \sqrt{2h}
\xi_{n+1},
\quad
\widetilde{Y}_0 = x_0,
\end{equation}
where $h \in (0,1)$ represents the uniform timestep and 
$\xi_{k} = ( \xi_{1,k}, \dots, \xi_{d,k})^{T}$, $k\in \mathbb{N}$, 
are i.i.d standard $d$-dimensional Gaussian vectors.

Non-asymptotic analysis focuses on the explicit dependency of the error with respect to the algorithm parameters, e.g., step size, rather than explaining the asymptotic behavior as the algorithm iterates to infinity or the step size tends to zero.
Non-asymptotic convergence analysis  for the LMC algorithm \eqref{equation:numerical-scheme-euler-maruyama-paper3}
is typically investigated with 

\noindent
{\color{black}
(i) \textit{global Lipschitz continuity condition}:
}
there exists a positive constant $\bar{L}>0$ such that
\begin{equation}\label{equation:lipschitz-u}
\|\nabla U(x) - \nabla U(y) \| 
\leq \bar{L} \|x-y \|, 
\quad \forall x, y \in \mathbb{R}^{d},
\, \text{ and }
\end{equation}
{\color{black}
(ii) \textit{strong convexity condition}: 
}
there exists a positive constant $\widetilde{L}>0$ such that
\begin{equation}\label{equation:strongly-convex-u}
\left\langle 
x-y,\nabla U(x)-\nabla U(y) 
\right\rangle 
\geq
\widetilde{L}\|x-y \|^{2}, 
\quad \forall x, y \in \mathbb{R}^{d}.
\end{equation}
In recent years, working with condition (i) \textcolor{blue}{and} (ii), non-asymptotic error analysis between the target distribution and the law of the LMC algorithm under various metrics, such as the Wasserstein distance and the total variation distance has been well established (see \cite{durmus2017nonasymptotic,dalalyan2017theoretical,durmus2019high}).

Except for a few rare cases, it is extremely hard for the 
{\color{black}
Langevin
}
SDE \eqref{eq:langevin-SODE} to satisfy either condition (i) or (ii). {A commonly used counterexample in quantum mechanics is the double-well potential presented in Example \ref{eqn:doublewell}.}
What if the drift $\nabla U$ grows superlinearly?
Conventional LMC algorithm loses its powers when attempting to sample from the target distribution inherited by \eqref{eq:langevin-SODE}. 
For example, as claimed by \cite{hutzenthaler2011strong,mattingly2002ergodicity}, for a large class of SDEs with super-linear growth coefficients, the Euler-Maruyama scheme \eqref{equation:numerical-scheme-euler-maruyama-paper3} leads to divergent numerical approximations in both finite and infinite time intervals.
Therefore, 
a convergent numerical algorithm for the Langevin SDEs \eqref{eq:langevin-SODE} with non-globally Lipschitz drift is necessary.
Recent years have witnessed a considerable growth in construction and analysis of convergent schemes for SDEs in the non-globally Lipschitz setting (see \cite{wang2013tamed,hutzenthaler2012strong,li2019explicit,sabanis2016euler,wang2020mean,beyn2016stochastic,beyn2017stochastic,szpruch2018}).
To deal with the super-linear drift of the Langevin SDEs \eqref{eq:langevin-SODE}, authors in \cite{brosse2019tamed} used a tamed Langevin algorithm to obtain the non-asymptotic error bounds on 2-Wasserstein distance and total variation distance. 

Beside the challenge of the super-linear drift, relaxing the strongly-convex potential, i.e. condition (ii), may lead to possible collapse of the classical method in the non-asymptotic error analysis \cite{majka2020nonasymptotic},
where
the contractivity condition can be replaced by a \textit{contractivity at infinity} condition (see Assumption \ref{assumption:contractivity-at-infinity-condition-paper3}).
Recently,
working on such a relaxed condition and demonstrating the non-asymptotic error analysis of the corresponding Langevin SDEs \eqref{eq:langevin-SODE} have indeed gained appreciable attention (see \cite{pages2023unadjusted,majka2020nonasymptotic,neufeld2022non} for example).
Nevertheless,
it is worth mentioning that a majority of existing works investigate this topic via the Wasserstein distance error bounds.
In \cite{majka2020nonasymptotic}, authors obtained the non-asymptotic upper bound on  1-Wasserstein distance and 2-Wasserstein distance of order $1/2$ and $1/4$, respectively, under globally Lipschitz condition.
If the Lipschitz condition is weakened by a polynomial growth condition, authors in \cite{neufeld2022non} improved the   1-Wasserstein error bound and the  2-Wasserstein error bound with respective convergence order of $1$ and $1/2$.


In this paper,
we investigate the non-asymptotic convergence of the target distribution admitted by \eqref{eq:langevin-SODE}
in total variation distance \eqref{equation:total-variation-between-probability distributions}-\eqref{equation:total-variation-between-rv-and-distribution-Cb}
in non-convex and non-globally Lipschitz setting (see Assumptions \ref{assumption:globally-polynomial-growth-condition-paper3}, \ref{assumption:contractivity-at-infinity-condition-paper3} and \ref{assumption:coercivity-condition-of-the-drift-paper3} below). 
The setting of total variation distance  
allows us to consider bounded and measurable function $\phi \in B_{b}(\mathbb{R}^{d})$ in
\eqref{introduction:invariant-pi}.
{\color{black}
Classical examples include 
indicator and step functions
}(see \eqref{equation:bounded-test-function-paper3}).
As far as we know, the investigation of approximation errors in weak approximations of \eqref{introduction:invariant-pi} without fulfilling conditions (i) and (ii) is still in its initial phases.
{\color{black}
The present article turns to handle the Langevin SDEs \eqref{eq:langevin-SODE} with super-linear growing nonlinearities. To this end, we propose a projected Langevin Monte Carlo (PLMC) algorithm 
with a uniform timestep $h \in (0,1)$ as follows:
\begin{equation} \label{equation:numerical-scheme-PLMC-algorithm-paper3-in-introduction}
    Y_{n+1} = 
    \mathscr{P}(Y_{n}) -  \nabla U \left(\mathscr{P}(Y_{n})\right) h 
    + \sqrt{2h}\xi_{n+1}, \quad
   Y_0 = x_0,
\end{equation}
where
$\xi_{k} = ( \xi_{1,k}, \dots, \xi_{d,k})^{T}$, $k\in \mathbb{N}$,
are i.i.d standard $d$-dimensional Gaussian vectors and
$\mathscr{P}: \mathbb{R}^{d} \rightarrow \mathbb{R}^{d}$ denotes the projection operator given by
\begin{equation}
\begin{aligned}
\mathscr{P}(x):= \begin{split}
\left\{
    \begin{array}{ll}
    \min\big\{
    1, 
    \vartheta
    ( 
    \tfrac{d}{h}
    )^{1/2\gamma} 
    \| x\|^{-1} 
     \big\}x,  \hspace{1em} \text{for}\ \gamma > 1,
    \\
   x,\ \hspace{10.7em} \text{for}\ \gamma = 1,
 \end{array}\right.
 \end{split}
\quad \forall x\in \mathbb{R}^{d}.
\end{aligned}
\end{equation}
Here $\gamma$ determined in Assumption \ref{assumption:globally-polynomial-growth-condition-paper3} below characterizes the growth of the gradient of $U$ and $\vartheta \geq 1$, independent of the stepsize $h$ and the dimension $d$, 
can be customized to prevent
the numerical approximations \eqref{equation:numerical-scheme-PLMC-algorithm-paper3-in-introduction} from being projected too often in the iteration for the case $\gamma > 1$.
In particular, 
our PLMC algorithm \eqref{equation:numerical-scheme-PLMC-algorithm-paper3-in-introduction} reduces to the classical LMC algorithm \eqref{equation:numerical-scheme-euler-maruyama-paper3} for the Langevin SDEs \eqref{eq:langevin-SODE}
with a Lipschitz continuous drift, i.e. $\gamma=1$.
}
%

{
\color{black}
The key idea of this paper is to turn non-asymptotic bounds on the total variation distance between the law of the numerical approximation and the target distribution to the weak error analysis in the setting of non-smooth test functions $\phi \in B_{b}(\mathbb{R}^{d})$.
We mention that, the seminal work \cite{bally1996law} obtained order one weak convergence in finite time of the Euler scheme for SDEs with globally Lipschitz and smooth coefficients, where the test functions are only assumed to be measurable and bounded.
Different from \cite{bally1996law}, the present weak error analysis is carried out in the infinite time horizon and the drift coefficient $f$ of SDEs possibly grows superlinearly.
These two aspects force us to encounter two major obstacles.
}
The first one is to get a couple of a priori estimates that are independent of time and stepsize, including the uniform moment bounds of the PLMC scheme \eqref{equation:numerical-scheme-PLMC-algorithm-paper3-in-introduction} and the time-independent regularity estimates of the Kolmogorov equation,
which is challenging due to the loss of condition (ii)
{\color{black}
and to the test functions not being smooth.}
Another one is the  discontinuity of the proposed PLMC algorithm \eqref{equation:numerical-scheme-PLMC-algorithm-paper3-in-introduction}, which results in further difficulties in handling the weak error via the Kolmogorov equation.
Novel techniques are used to circumvent these two difficulties.
Discrete arguments are adopted to obtain the uniform moment bounds of the PLMC algorithm \eqref{equation:numerical-scheme-PLMC-algorithm-paper3-in-introduction} (see the proof of Lemma \ref{lemma:Uniform-moment-bounds-of-the-PLMC-paper3}).
{\color{black}
Moreover,
the Bismut-Elworthy-Li formula (Lemma \ref{lemma:Bismut-Elworthy-Li-formula-paper3}), combined with the Markov property of the transition semigroup $u(\cdot,\cdot)$ as shown in \eqref{equation:def-of-u-paper3}, is used to derive the time-independent regularity estimates of the Kolmogorov equation (see Section \ref{section:Kolmogorov-equation-and-regularization-estimates}).}
To handle the discontinuity of the PLMC algorithm \eqref{equation:squared-PLMC-algorithm-paper3},
we introduce the continuous-time version \eqref{introduction:continuous-time-version-of-PLMC algorithm-paper3} at each step to fully exploit the Kolmogorov equation (see \eqref{equation:kolmogorov-equation-paper3}).

For the weak error analysis,
given any terminal time $T \in (0, \infty)$ such that $Nh=T$, $N \in \mathbb{N}$,
we separate the error 
$\big| \mathbb{E}\left[\phi(Y^{x_{0}}_{N}) \right] - \mathbb{E}\left[\phi(X^{x_{0}}_{T}) \right]\big|$, i.e., $\big| \mathbb{E}\left[u(T, x_{0}) \right] - \mathbb{E}\left[u(0, Y^{x_{0}}_{N}) \right] \big|$ based on the associated Kolmogorov equation (see \eqref{equation:kolmogorov-equation-paper3}) into two parts as $J_{1}$ and $J_{2}$ in  \eqref{eq:decomposition-of-the-weak-error-paper3}.
The  first part $J_{1}$
is of {\color{black}
order 
$\mathcal{O}(d^{2\gamma-1}h)$,
}
which can be considered as a  direct consequence of the time-independent regularity of $u(t,\cdot)$ (see Theorem \ref{theorem:regular-estimate-of-u-and-its-derivatives-paper3}) and the convergence of the projected operator $\mathscr{P}(\cdot)$ (see Lemma \ref{lemma:error-estimate-between-x-and-projected-x-paper3}).
By virtue of the continuous version of the PLMC algorithm \eqref{introduction:continuous-time-version-of-PLMC algorithm-paper3}, the Kolmogorov equation and the It\^o formula, the second error term $J_{2}$ 
can be proved to be 
{\color{black}
$\mathcal{O}(d^{\max\{3\gamma/2 , 2\gamma-1 \}}h |\ln h|)$ 
}
(see Theorem \ref{theorem: Time-independent weak error analysis-paper3} and its proof for more details).
{
\color{black}
For the the Lipschitz case, i.e., $\gamma=1$,
we can improve the estimate of $J_{2}$ with the help of a refinement of the proof based on Malliavin calculus techniques and thus obtain the improved convergence rate $\mathcal{O}(d^{3/2}h)$ (see Theorem \ref{theorem:optimal-convergence-rate-lipschitz-setting-paper3}).
Theorem \ref{theorem: Time-independent weak error analysis-paper3} and
Theorem \ref{theorem:optimal-convergence-rate-lipschitz-setting-paper3}
pave the way to analyzing the smallest number of iterations of the PLMC scheme \eqref{equation:numerical-scheme-PLMC-algorithm-paper3-in-introduction} required to approach the target distribution inherited by \eqref{eq:langevin-SODE} on the total variation distance with given precision, also known as the mixing time, 
which has been remaining an active field \cite{li2022sqrt,cheng2018sharp} in recent years.
More precisely,
it is shown that, to achieve a given precision $\epsilon$,
the smallest number of iterations of the 
PLMC scheme \eqref{equation:numerical-scheme-PLMC-algorithm-paper3-in-introduction} 
is of order ${\mathcal{O}}\big(\tfrac{d^{\max\{3\gamma/2 , 2\gamma-1 \}}}{\epsilon} \cdot \ln (\tfrac{d}{\epsilon}) \cdot \ln (\tfrac{1}{\epsilon}) \big)$ for $\gamma>1$,
and of order ${\mathcal{O}}\big(\tfrac{d^{3/2}}{\epsilon} \cdot \ln (\tfrac{1}{\epsilon}) \big)$ for $\gamma=1$.
(see Proposition \ref{theorem-mixing-time}).
}

We summarize our main contributions as follows:
\begin{itemize}
    \item A projected Langevin Monte Carlo algorithm, capable of dealing with super-linear systems and covering the classical Langevin Monte Carlo algorithm, is presented.
    \item Non-asymptotic bounds on the total variation distance between the law of the PLMC algorithm and the target distribution, inherited by \eqref{eq:langevin-SODE}, are established for non-convex potential.
    \item  {\color{black}
    The smallest number of iterations of the projected Langevin Monte Carlo scheme required to approach the target distribution, admitted by \eqref{eq:langevin-SODE}, in the total variation distance with given precision is shown.
    }
\end{itemize}

The rest of this article is organized as follows.
The next section formulates the primary setting and shows the main result of this paper.
In Section \ref{section:Preliminary-results-paper3}, 
we present some a priori estimates of both SDE and the PLMC algorithm.
Section \ref{section:Kolmogorov-equation-and-regularization-estimates} reveals the Kolmogorov equation and its regularity estimates.
In Section \ref{section:Time-independent-error-analysis}, the time-independent weak error analysis between SDE and the PLMC scheme is given. {\color{black} Section \ref{section:optimal-error-analysis} is devoted to the optimal weak convergence rate for the Lipschitz case $\gamma = 1$.}
Some numerical tests are shown to illustrate our theoretical findings in Section \ref{section:Numerical-experiment}. Finally, the Appendix contains the detailed proof of several auxiliary lemmas.

\section{Settings and main results}
Throughout this paper, we use $\N$ to denote the set of all positive integers
and let $ d,m \in \N$, $ T \in (0, \infty) $ be given. Let $\| \cdot \|$ and $ \langle \cdot, \cdot \rangle $ denote the Euclidean
norm and the inner product of vectors in $\R^d$, respectively. 
Adopting the same notation as the vector norm, we denote $\|A\| : =\sqrt{\tr(A^{T}A)}$ as the trace norm of a matrix $A \in \R^{d \times m}$.
We use $\max\{a,b\}$ and $\min\{a,b\}$ for the maximum and  minimum values of between $a$ and $b$ respectively.
Given a filtered probability space $ \left( \Omega, \mathcal{ F }, \{ \mathcal{ F }_t \}_{ t \in [0,T] }, \P 
\right) $, we use $\E$ to mean the expectation and $L^{r} (\Omega, \R^d ), r \geq 1 $, to denote 
the family of $\R^d$-valued random variables $\xi$ satisfying   $\E[ \|\xi \|^{r}]<\infty$.
Moreover, 
we introduce a new notation $X^{x}_{t}$ for $t\in [0,\infty)$ denoting the solution of SDE \eqref{eq:langevin-SODE} satisfying the initial condition $X^{x}_{0} = X_{0}=x$. 
In addition, denote by $C_{b}(\mathbb{R}^{d})$ 
(resp. $B_{b}(\mathbb{R}^{d})$ )
the Banach space of all uniformly continuous and bounded mappings
(resp. Borel bounded mappings) 
$\varphi: \mathbb{R}^{d} \rightarrow \mathbb{R}$ endowed with the norm $\|\varphi \|_{0} = \sup_{x\in \mathbb{R}^{d}} |\varphi(x) |$.

For the vector-valued function $\textbf{u}: \mathbb{R}^{d} \rightarrow \mathbb{R}^{{\ell}}$,
$\textbf{u} = (
u_{(1)}, \dots, u_{({\ell})}
)$, its first order partial derivative is considered as the Jacobian matrix as
\begin{equation}
D \mathbf{u}=
\left(\begin{array}{ccc}
\frac{\partial u_{(1)}}{\partial x_1} & \cdots & \frac{\partial u_{(1)}}{\partial x_d} \\
\vdots & \ddots & \vdots \\
\frac{\partial u_{(\ell)}}{\partial x_1} & \cdots & \frac{\partial u_{(\ell)}}{\partial x_d}
\end{array}\right)_{\ell \times d}.
\end{equation}
For any $v_{1} \in \mathbb{R}^{d}$, one knows $D (\mathbf{u} ) v_{1} \in \mathbb{R}^{\ell}$ and
one can define $D^{2} \mathbf{u}(v_{1},v_{2})$ as 
\begin{equation}
D^{2} \mathbf{u}(v_{1},v_{2})
:= D \big( D (\mathbf{u} )v_{1} \big) v_{2}, \quad \forall v_{1}, v_{2} \in \mathbb{R}^{d}.
\end{equation}
In the same manner, one can define
\begin{equation}
D^{3} \mathbf{u}(v_{1},v_{2},v_{3})
:= D\Big( D \big( D (\mathbf{u} )v_{1} \big) v_{2} \Big)v_{3}, \quad \forall v_{1}, v_{2}, v_{3} \in \mathbb{R}^{d},
\end{equation}
and for any integer $k \geq 3$ the $k$-th order 
{\color{black} derivatives
}
of the function  $\textbf{u}$ can be defined recursively.
Given the Banach spaces $\mathcal{X}$ and $\mathcal{Y}$, we denote by $L(\mathcal{X}, \mathcal{Y})$ the Banach space of bounded linear operators from  $\mathcal{X}$ into $\mathcal{Y}$.
Then the partial derivatives of the function  $\textbf{u}$ can also be regarded as the operators
\begin{equation}
D \textbf{u}(\cdot)(\cdot): \mathbb{R}^d \to L(\mathbb{R}^d, \mathbb{R}^\ell),    
\end{equation}
\begin{equation}
D^2 \textbf{u}(\cdot)(\cdot,\cdot): \mathbb{R}^d \to 
L(\mathbb{R}^d, 
{\color{black}
L(\mathbb{R}^d,\mathbb{R}^\ell))
}\cong  L(\mathbb{R}^d\otimes\mathbb{R}^d, \mathbb{R}^\ell)
\end{equation}
and
\begin{equation}
D^3 \textbf{u}(\cdot)(\cdot,\cdot,\cdot): \mathbb{R}^d \to L(\mathbb{R}^d, L(\mathbb{R}^d,
{\color{black}
L(\mathbb{R}^d,\mathbb{R}^\ell)))
}
\cong  L((\mathbb{R}^d)^{\otimes 3}, \mathbb{R}^\ell).
\end{equation}
We remark that the partial derivatives of the scalar valued function can be covered by the special case $\ell =1$.

For any $k\in \mathbb{N}$, let $C^{k}_{b}(\mathbb{R}^{d})$ be the subspace of $C_{b}(\mathbb{R}^{d})$ consisting of all functions with bounded partial derivatives $D^{i}\varphi(x)$, $1\leq i \leq k$.
In what follows, we use the letter $C$ to denote generic constants, independent of both the step size $h\in (0,1)$ and the dimension $d$.
Also, let $\textbf{1}_{H}$ be the indicator function of a set $H$.

Further, 
the total variation distance between two Borel probability distributions $\mu_{1}$ and $\mu_{2}$ is defined by
\begin{equation} \label{equation:total-variation-between-probability distributions} 
\|
\mu_{1}-\mu_{2} 
\|_{\text{TV}} 
:=
\sup_{
\varphi 
\in 
B_{b}(
\mathbb{R}^{d}
), 
\ \varphi \neq 0 
} 
\dfrac{
\left|
\int_{\mathbb{R}^d} 
\varphi(x) \mu_{1}(\dd x)  
- 
\int_{\mathbb{R}^d} 
\varphi(x) \mu_{2}(\dd x) 
\right| 
}{
\|\varphi \|_{0}
}.
\end{equation}
Consequently, if $X$ is a $\mathbb{R}^{d}$-valued random variable, the probability distribution of $X$ is denoted by $\Pi(X)$. Then, the total variation distance between $\Pi(X)$ and any Borel probability distribution $\mu$ is given as,
\begin{equation}
\|
\Pi(X)-\mu 
\|_{\text{TV}} 
:= 
\sup_{
\varphi 
\in 
B_{b}(
\mathbb{R}^{d}
), 
\ 
\|\varphi\|_{0} 
\leq 1 
} 
\left|
\mathbb{E}
\left[
\varphi(X) 
\right]  
- 
\int_{\mathbb{R}^d} 
\varphi(x) 
\mu(\dd x) 
\right|.
\end{equation}
{\color{black}
Recalling that for any $\varphi \in B_{b}(\mathbb{R}^{d})$, there exists a sequence $ \{ \varphi_{k} \}_{k \in \mathbb{N} }$ of bounded and continuous functions which converges boundedly and pointwise to $\varphi$, i.e., which satisfies 
{\color{black}
$ \sup_{k \in \mathbb{N}} \sup_{x \in \mathbb{R}^{d}} | \varphi_{k}(x)| < \infty $
}
and $\lim_{k \rightarrow \infty} \varphi_{k}(x) = \varphi(x)$, for all $x\in \mathbb{R}^{d}$  
(see \cite{brehier2024total} or \cite[Proposition 4.2]{ethier2009markov} for details).
} 
Therefore, the total variation distance between two Borel probability distributions $\mu_{1}$ and $\mu_{2}$ has the representation as follows,
\begin{equation}\label{equation:tv-distance-of-two-measures}
\begin{aligned}
\|
\mu_{1}-\mu_{2} 
\|_{\text{TV}} 
&=
\sup_{
\varphi 
\in 
C_{b}(\mathbb{R}^{d}),
\ \varphi \neq 0 
} 
\dfrac{
\left|
\int_{
\mathbb{R}^d
} 
\varphi(x) 
\mu_{1}(\dd x)  
- 
\int_{
\mathbb{R}^d
} 
\varphi(x) 
\mu_{2}(\dd x) 
\right| 
}{
\|\varphi \|_{0}
} \\
&=
\sup_{
\varphi 
\in C_{b}(\mathbb{R}^{d}), 
\ 
{ \color{black}\| \varphi \|_{0} \leq 1 }
}
\left|
\int_{\mathbb{R}^d} 
\varphi(x) 
\mu_{1}(\dd x)  
- 
\int_{\mathbb{R}^d} 
\varphi(x) 
\mu_{2}(\dd x) 
\right|.
\end{aligned}
\end{equation}
For the probability distribution of $X$ and any Borel probability distribution $\mu$, one also has
\begin{align} \label{equation:total-variation-between-rv-and-distribution-Cb}
\begin{split}
\|
\Pi(X)-\mu 
\|_{\text{TV}} 
&= 
\sup_{
\varphi 
\in 
C_{b}(\mathbb{R}^{d}), 
\ \varphi \neq 0 
}
\dfrac{
\left|
\mathbb{E}
\left[
\varphi(X) 
\right]   
- 
\int_{
\mathbb{R}^d
} 
\varphi(x) 
\mu(\dd x) 
\right| 
}{
\|\varphi \|_{0}
} \\
&=
\sup_{
\varphi 
\in C_{b}(\mathbb{R}^{d}), 
\ \|\varphi\|_{0} \leq 1 
}
\left|
\mathbb{E}
\left[
\varphi(X) 
\right]  
- \int_{\mathbb{R}^d} 
\varphi(x) \mu(\dd x) 
\right|.
\end{split}
\end{align}
Throughout the paper, we denote $f(x) := - \nabla U(x), \ x\in \mathbb{R}^{d}$, as the negative gradient of the potential $U$ for convenience.
Subsequently,
we set up a non-convex framework.
\begin{assumption} \label{assumption:globally-polynomial-growth-condition-paper3}
(Globally polynomial growth condition.)
Assume the drift coefficient $f: \mathbb{R}^{d} \rightarrow \mathbb{R}^{d}$ of SDE \eqref{eq:langevin-SODE} is twice continuously differentiable in $\mathbb{R}^{d}$, and there exists some constant $\gamma \in [1, \infty)$ such that,
\begin{equation}
\begin{aligned}
\left\|
D^{2}f(x) 
( v_{1}, v_{2}) 
\right\| 
&\leq 
C
{\color{black}
(
1+\|x\|
)^{\max\{0, \gamma-2 \} } 
}
\|v_{1} \|
\cdot 
\|v_{2} \|,
\quad 
\forall x, v_{1}, v_{2}\in \mathbb{R}^{d}. \\
\end{aligned}
\end{equation}
\end{assumption}
Note that Assumption \ref{assumption:globally-polynomial-growth-condition-paper3} immediately implies
\begin{equation} \label{equation:growth-of-derivative-of-f-paper3}
\begin{aligned}
\left\|
Df(x)  v_{1} 
- Df(\tilde{x}) v_{1}  
\right\| 
&\leq 
C
{\color{black}
(
1
+
\|x\|
+
\|\widetilde{x}\|
)^{\max\{0, \gamma-2 \}}
}
\|
x-\tilde{x}
\| 
\cdot 
\|v_{1} \|, 
\quad 
\forall x, \tilde{x}, v_{1}\in \mathbb{R}^{d}, \\
\left\|
Df(x)  v_{1}  
\right\| 
&\leq
C(
1+\|x\|
)^{\gamma-1}
\|v_{1} \| ,
\quad 
\forall x, v_{1}\in \mathbb{R}^{d},
\end{aligned}
\end{equation}
which in turn yields
\begin{equation} \label{equation:growth-of-f-paper3}
\begin{aligned}
\left\|
f(x)  - f(\tilde{x}) 
\right\| 
&\leq 
C
(
1
+
\|x\|
+
\|
\tilde{x}
\|
)^{\gamma-1}  
\|
x-\tilde{x}
\|,
\quad 
\forall x, \tilde{x}\in \mathbb{R}^{d}, \\
\left\|
f(x)   
\right\| 
&\leq 
C(
1
+
\|x\|
)^{\gamma} ,
\quad 
\forall x\in \mathbb{R}^{d}.
\end{aligned}
\end{equation}
In what follows, we formulate the \textit{contractivity at infinity} condition as shown in \cite{majka2020nonasymptotic}.
\begin{assumption} \label{assumption:contractivity-at-infinity-condition-paper3}
(Contractivity at infinity condition.)
For the drift coefficient $f: \mathbb{R}^{d} \rightarrow \mathbb{R}^{d}$ of SDE \eqref{eq:langevin-SODE},
there exist some positive constants $\widetilde{a}_{1}> \widetilde{a}_{2}>0$, $\mathcal{R} >0$ such that,
\begin{equation}
\left\langle 
x-y,f(x)-f(y) 
\right\rangle 
\leq 
\left(
\widetilde{a}_{1}
\mathbf{1}_{
\|x-y \| \leq \mathcal{R}
}
-\widetilde{a}_{2}
\right)
\| x-y\|^{2}, 
\quad 
\forall x, y \in \mathbb{R}^{d}.
\end{equation}
\end{assumption}
\begin{remark} \label{remark: one-side Lipschitz condition of the drfit-paper3}
It is noteworthy that Assumption \ref{assumption:contractivity-at-infinity-condition-paper3} implies a 
{\color{black}
one-sided
}
Lipschitz condition of the drfit $f: \mathbb{R}^{d}\rightarrow \mathbb{R}^{d}$ as follows: there exists some positive constant $L>0$ such that
\begin{equation} \label{eq:remark2.3-paper3}
\left\langle 
x-y,f(x)-f(y) 
\right\rangle 
\leq 
L\|x-y \|^{2}, 
\quad 
\forall x, y \in \mathbb{R}^{d}.
\end{equation}
\end{remark}
Furthermore,  we put 
{\color{black}
a more relaxed
dissipativity condition
}
on the drift $f$ as follows, compared to \cite[Assumption 3]{neufeld2022non} (see \cite[Remark 2.6]{neufeld2022non} for comparison).
\begin{assumption}\label{assumption:coercivity-condition-of-the-drift-paper3}
For the drift coefficient $f: \mathbb{R}^{d} \rightarrow \mathbb{R}^{d}$,
there exist some constants $a_{1}, a_{2}>0$
such that
{\color{black}
\begin{equation}
\left\langle 
x,f(x) 
\right\rangle 
\leq 
-a_{1}
\|
x 
\|^{2}
+ a_{2}, 
\quad 
\forall x\in \mathbb{R}^{d}.
\end{equation} }
\end{assumption}
Compared with the usual but strict {\color{black}strong convexity condition on $U$ (i.e., $\widetilde{a}_{1}=0$ in Assumption \ref{assumption:contractivity-at-infinity-condition-paper3}) and the global Lipschitz continuity condition on $\nabla U$ (i.e. $\gamma=1$ in Assumption \ref{assumption:globally-polynomial-growth-condition-paper3}),
}
the above settings enable us to accommodate for a much wider family of SDEs, especially for possible non-convex potentials. Here we present an example.


% ...

\begin{example}[Double-well potential]\label{eqn:doublewell}
    
Consider $f(x)=x(1-\|x\|^{2})$, which is the negative gradient of the double-well potential $U(x)=\|x \|^{4}/4 - \| x\|^{2}/2$. Such an $f$ violates \eqref{equation:lipschitz-u} and \eqref{equation:strongly-convex-u} but satisfies Assumptions \ref{assumption:globally-polynomial-growth-condition-paper3}-\ref{assumption:coercivity-condition-of-the-drift-paper3} with $\gamma=3$,  
{\color{black}
$a_{1}=1$, $a_{2}=1$,
}
{\color{black}
and $\widetilde{a}_{1} = 4\sqrt{2} + 19/2$, $\widetilde{a}_{2} = 1/2$, $\mathcal{R}=16+20\sqrt{2}$.
A brief proof is given in Appendix \ref{proof-of-example}.
}
\end{example}

To obtain  numerical approximations of the invariant measure allowed by such SDEs, 
we propose a family of projected Langevin Monte Carlo (PLMC) algorithms as follows:
\begin{equation} \label{equation:numerical-scheme-PLMC-algorithm-paper3}
\begin{split}
\left\{
    \begin{array}{ll}
    Y_{n+1} = 
    \mathscr{P}(Y_{n}) +  f\left(\mathscr{P}(Y_{n})\right) h 
    + \sqrt{2h}\xi_{n+1},
    \\
   Y_0 = x_0,
 \end{array}\right.
 \end{split}
\end{equation}
where
$\xi_{k} = ( \xi_{1,k}, \dots, \xi_{d,k})^{T}$, $k\in \mathbb{N}$,
are i.i.d standard $d$-dimensional Gaussian vectors.
{\color{black}
For a parameter $\vartheta \geq 1$, independent of the stepsize $h$ and the dimension $d$,
the projection operator $\mathscr{P}: \mathbb{R}^{d} \rightarrow \mathbb{R}^{d}$ is  defined by
\begin{equation} \label{equation:projection-operator-paper3}
\begin{aligned}
\mathscr{P}(x):=
\left\{
    \begin{array}{ll}
    \min\big\{
    1, 
    \vartheta
    ( 
    \tfrac{d}{h}
    )^{1/2\gamma} 
    \| x\|^{-1} 
     \big\}x,  \hspace{1em} \text{for}\ \gamma > 1,
    \\
   x,\ \hspace{10.7em} \text{for}\ \gamma = 1,
 \end{array}\right.
\quad \forall x\in \mathbb{R}^{d},
\end{aligned}
\end{equation}
with $\gamma$ being given in Assumption \ref{assumption:globally-polynomial-growth-condition-paper3}.
Here the parameter $\vartheta \geq 1$
can be customized to prevent
the numerical solutions of the PLMC scheme \eqref{equation:numerical-scheme-PLMC-algorithm-paper3} from being projected too often in the iteration for the case $\gamma > 1$. 
}
{\color{black}
Also, we let $\bold{0}/0 = \bold{0}$ for $\bold{0} := (0, \cdots, 0)^T \in \mathbb{R}^{d}$.
}
For the particular case $\gamma=1$, $\mathscr{P}=I$ is the identity operator.
The first main result of this paper is formulated as follows.
\begin{theorem}\label{theorem:main-result-paper3}
(Main result: non-asymptotic bounds in the total variation distance)
Assume Assumptions \ref{assumption:globally-polynomial-growth-condition-paper3}, \ref{assumption:contractivity-at-infinity-condition-paper3}, \ref{assumption:coercivity-condition-of-the-drift-paper3}. Let $\{X^{x_{0}}_{t}\}_{t \geq 0}$ and $\{Y^{x_{0}}_{n}\}_{n \geq 0}$
be the solutions of SDE \eqref{eq:langevin-SODE} and the PLMC algorithm \eqref{equation:numerical-scheme-PLMC-algorithm-paper3} with the initial state $X^{x_{0}}_{0}=Y^{x_{0}}_{0}=x_{0}$, respectively.
{\color{black}
Then the Langevin SDE $\{X^{x_{0}}_{t}\}_{t\geq 0}$ in \eqref{eq:langevin-SODE} converges exponentially to 
the invariant measure $\pi$ under the total variation distance.
}
Furthermore,
{\color{black}
let $h\in (0, \min\{1/2a_{1}, 2a_{1}/(a_{1}+2C^{2}_{f}), 1\} )$ be the uniform timestep with $a_{1}$, $C_{f}$ coming from Assumption \ref{assumption:coercivity-condition-of-the-drift-paper3}
and Lemma \ref{lemma: useful estimate of PLMC algorithm-paper3}, respectively.
}
{\color{black}
Given 
any $N \in \mathbb{N}$,
there exist some constants
$c_{\star} = C(\widetilde{a}_{1}, \widetilde{a_{2}}, \mathcal{R}) > 0$, $C_{\star} = C(\widetilde{a}_{1}, \widetilde{a_{2}}, \mathcal{R}) > 0$ with $\widetilde{a}_{1}$, $\widetilde{a}_{2}$, $\mathcal{R}$ determined in Assumption \ref{assumption:contractivity-at-infinity-condition-paper3} and $C = C(c_{\star}, C_{\star}, x_{0})$ such that,
}
\begin{equation}
\left\|
\Pi(Y^{x_{0}}_{N}) 
-  
\pi 
\right\|_{\text{TV}}
\leq 
{\color{black}
C_{\star} \|\phi \|_{0} e^{- c_{\star} Nh}
}
\left(
1
+
{
\color{black}
\mathbb{E}
\left[
\|x_{0}\|
\right]
}
\right) 
+ 
Cd^{
\max\{3\gamma/2 , 2\gamma-1 \} 
}
h
\left| 
\ln{h} 
\right|.
\end{equation}
\end{theorem}
{\color{black}
We mention that,
for the particular case of the Lipschitz drift $f$, i.e. $\gamma=1$,
the PLMC \eqref{equation:numerical-scheme-PLMC-algorithm-paper3} reduces to the classical LMC  \eqref{equation:numerical-scheme-euler-maruyama-paper3} 
and we can obtain an improved convergence result.
}
\begin{theorem}(Main result: the Lipschitz case)
\label{theorem:main-result-paper3-optimal}
Let Assumptions \ref{assumption:globally-polynomial-growth-condition-paper3}, \ref{assumption:contractivity-at-infinity-condition-paper3}, \ref{assumption:coercivity-condition-of-the-drift-paper3} hold with $\gamma=1$.
Let $\{X^{x_{0}}_{t}\}_{t\geq 0}$ and $\{\widetilde{Y}^{x_{0}}_{n} \}_{n \geq 0}$ be the solutions of SDE \eqref{eq:langevin-SODE} and the LMC algorithm \eqref{equation:numerical-scheme-euler-maruyama-paper3} with the same initial state $X^{x_{0}}_{0} = Y^{x_{0}}_{0} = x_{0}$, respectively. Also, let 
$h\in (0, \min\{1/2a_{1}, 2a_{1}/(a_{1}+2C^{2}_{f}), 1/4\} )$, 
where $a_{1}$ and $C_f$ are given in Assumption \ref{assumption:coercivity-condition-of-the-drift-paper3} and Lemma \ref{lemma: useful estimate of PLMC algorithm-paper3}, respectively,
be the uniform timestep. 
Then, given 
any $N \in \mathbb{N}$,
there exist some constants
$c_{\star} = C(\widetilde{a}_{1}, \widetilde{a_{2}}, \mathcal{R}) > 0$, $C_{\star} = C(\widetilde{a}_{1}, \widetilde{a_{2}}, \mathcal{R}) > 0$ with $\widetilde{a}_{1}$, $\widetilde{a}_{2}$, $\mathcal{R}$ determined in Assumption \ref{assumption:contractivity-at-infinity-condition-paper3} and $C = C(c_{\star}, C_{\star}, x_{0})$ such that,
{\color{black}
\begin{equation} \label{equation:main-result-lipschitz-case-paper3}
\left\|
\Pi(\widetilde{Y}^{x_{0}}_{N}) 
-  
\pi 
\right\|_{\text{TV}}
\leq 
C_{\star} \|\phi \|_{0} e^{- c_{\star} Nh}
\left(
1
+
\mathbb{E}
\left[
\|x_{0}\|
\right]
\right) 
+ Cd^{3/2}
h.
\end{equation}
}
\end{theorem}
%
%
\begin{remark}
    While completing this paper, we are aware of \cite{li2023unadjusted}, where 
    the convergence rates of the $L^p$-Wasserstein distance ($p\in [0, \infty)$)
    were obtained for the LMC \eqref{equation:numerical-scheme-euler-maruyama-paper3} with a sequence of decreasing step sizes under the contractivity at infinity condition on the H\"older continuous drift $f$ with linear growth.
    Instead, the main focus of this paper is the analysis of the total variation distance with super-linearly growing continuous drift $f$ and fixed step sizes.
\end{remark}
As a direct consequence of the above two theorems, we obtain the following results. 
{\color{black}
\begin{proposition} \label{theorem-mixing-time}
Let
Assumptions \ref{assumption:globally-polynomial-growth-condition-paper3}, \ref{assumption:contractivity-at-infinity-condition-paper3}, \ref{assumption:coercivity-condition-of-the-drift-paper3}
hold.
To achieve a given precision level $\epsilon > 0$ under total variation distance,
a required number of iterations of the PLMC algorithm \eqref{equation:numerical-scheme-PLMC-algorithm-paper3} is of order ${\mathcal{O}}\big(\tfrac{d^{\max\{3\gamma/2 , 2\gamma-1 \}}}{\epsilon} \cdot \ln (\tfrac{d}{\epsilon}) \cdot \ln (\tfrac{1}{\epsilon}) \big)$ for $\gamma>1$.
{\color{black}
For the particular case $\gamma=1$,
the mixing time of the LMC algorithm \eqref{equation:numerical-scheme-euler-maruyama-paper3}
required to approximate the target measure of the Langevin SDE \eqref{eq:langevin-SODE}
is of order
${\mathcal{O}}\big(\tfrac{d^{3/2}}{\epsilon}  \cdot \ln (\tfrac{1}{\epsilon}) \big)$.
}
\end{proposition}

 The proof of Proposition \ref{theorem-mixing-time} is straightforward and postponed to Appendix \ref{proof-of-mixing-time}.
}
To the best of our knowledge,
this is the first result considering the mixing time of the LMC algorithm under the total variation distance in a nonconvex setting.


\section{Preliminary results}
\label{section:Preliminary-results-paper3}
\subsection{A priori estimates of the Langevin SDE}
We begin with the lemma concerning the uniform moments estimate of $\{X_{t} \}_{t \geq 0}$, defined by \eqref{eq:langevin-SODE}.
\begin{lemma}\label{lemma:Uniform-moment-bounds-of-the-Langevin-SDE-paper3}
(Uniform moment bounds of the Langevin SDE) Let the solution of the Langevin SDE, denoted by $\{X_{t} \}_{t \geq 0}$  in \eqref{eq:langevin-SODE}, satisfy 
{\color{black}
Assumption
\ref{assumption:coercivity-condition-of-the-drift-paper3}. 
}
Then, 
{\color{black}
for any constant $c\in (0, 2a_{1})$ 
and any $p \in [1,\infty)$,
}
\begin{equation}
\mathbb{E} 
\left[
\left\|
X_{t}  
\right\|^{2p} 
\right] 
\leq 
e^{-cpt}
\mathbb{E}
\left[
\|x_{0}\|^{2p}
\right] 
+ 
{\color{black}
\tfrac{(
4p-2+2a_{2}
)^{p}}{p}
\left(
\tfrac{ p-1}{(2a_{1} - c)p}
\right)^{p-1}
}
d^{p},
\end{equation}
where 
$a_{1}, a_{2}>0$
are determined by Assumption
\ref{assumption:coercivity-condition-of-the-drift-paper3}.
\end{lemma}
The proof of Lemma \ref{lemma:Uniform-moment-bounds-of-the-Langevin-SDE-paper3} can be found in Appendix \ref{proof-of-lemma:Uniform-moment-bounds-of-the-Langevin-SDE-paper3}.
The following lemma states the existence and the uniqueness of the invariant measure induced by the Langevin SDEs \eqref{eq:langevin-SODE}.
\begin{lemma}
\label{lemma:Existence-and-uniqueness-of-the-invariant-measure-paper3}
(Existence and uniqueness of the invariant measure for the Langevin SDEs)
Let Assumptions \ref{assumption:globally-polynomial-growth-condition-paper3}, \ref{assumption:contractivity-at-infinity-condition-paper3} and  \ref{assumption:coercivity-condition-of-the-drift-paper3} hold.
{\color{black}
Then the Langevin SDEs \eqref{eq:langevin-SODE} $\{ X^{x_{0}}_{t}\}_{t \geq 0}$,
with the initial condition  $X_{0}=x_{0}$ admit a unique invariant measure, denoted by $\pi$.
}
In addition,
given $\phi \in C_{b}(\mathbb{R}^{d})$,
{\color{black}
for some constants 
$c_{\star} = C(\widetilde{a}_{1}, \widetilde{a_{2}}, \mathcal{R}) > 0$ 
and $C_{\star} = C(\widetilde{a}_{1}, \widetilde{a_{2}}, \mathcal{R}) > 0$
such that
}
\begin{equation} \label{equation: expoenetially decrease of SDE to the invariant measure}
\left|
\mathbb{E} 
\left[
\phi\left(X^{x_{0}}_{t}\right)
\right]
-
\int_{\mathbb{R}^d} 
\phi(x)
\pi(\mathrm{d} x)
\right| 
\leq 
{\color{black}
C_{\star} \|\phi\|_{0}
}
{\color{black}
e^{- c_{\star} t}
}
\left(
1+
{
\color{black}
\mathbb{E}\left[
\|x_{0}\|
\right]
}
\right), \quad \forall  t > 0.
\end{equation}
\end{lemma}
\begin{proof}
[Proof of Lemma \ref{lemma:Existence-and-uniqueness-of-the-invariant-measure-paper3}]
Equipped with Lemma \ref{lemma:Uniform-moment-bounds-of-the-Langevin-SDE-paper3}, 
{\color{black}
the existence of the invariant measure for SDE \eqref{eq:langevin-SODE} is obtained by the Krylov-Bogoliubov criterion \cite[Theorem 7.1]{da2006introduction}.
}

Moreover, the uniqueness of the invariant probability measure, defined as $\pi$, 
{\color{black}
can be derived from the Doob Theorem \cite[Theorem 2.1.3]{cerrai2001second}. 
}
To achieve this, the strong-Feller property and the irreducibility property need to be validated.
Indeed,
for the Langevin SDEs \eqref{eq:langevin-SODE}
it suffices to show that, given $\phi \in C_{b}(\mathbb{R}^{d})$, 
\begin{equation}
\label{equation:feller-in-lemma-paper3}
\big|
P_{t} \phi(x_{1}) - P_{t} \phi(x_{2})
\big|
\leq 
{\color{black}
C(t) \|\phi \|_{0} }
\cdot \|x_{1}-x_{2} \|,
\quad 
\forall t > 0,
\quad x_{1}, x_{2} \in \mathbb{R}^{d},
\end{equation}
where 
$P_{t}$ is the Markov semigroup of the corresponding SDEs \eqref{eq:langevin-SODE} (see \cite[Chapter 5]{da2006introduction}) denoted by
\begin{equation}
P_{t} \phi(x) = \mathbb{E}[\phi(X^{x}_{t})], \quad
 \phi \in C_{b}(\mathbb{R}^{d}).
\end{equation}
{\color{black}
Thanks to \eqref{equation:1st-derivative-of-U-paper3} and \eqref{eq:strong-feller-property-paper3} below, \eqref{equation:feller-in-lemma-paper3} is satisfied with
$C(t)=(
 e^{2Lt}-1 
)^{1/2}/2L^{1/2}t$,
which behaves as 
$\mathcal{O}(1/\sqrt{t})$ when $t \rightarrow 0$
and
grows exponentially
when $t \rightarrow \infty$.
}
The strong Feller property of $P_t$ is thus guaranteed.
Since the noise of
\eqref{eq:langevin-SODE} is additive and non-degenerate, 
{\color{black}
the irreducibility of SDEs \eqref{eq:langevin-SODE} is straightforward (see \cite[Proposition 2.3.3]{cerrai2001second} or \cite[Proposition 7.1.8]{da2006introduction} for example).
}
As a consequence, the Langevin SDE \eqref{eq:langevin-SODE} is ergodic.

{\color{black}
Moving on to the exponentially mixing property \eqref{equation: expoenetially decrease of SDE to the invariant measure}, it follows from Kantorovich–Rubinstein representation of the $L^{1}$-Wasserstein distance (see \cite[Proposition 2.1]{pages2023unadjusted}) and Theorem \ref{theorem:regular-estimate-of-u-and-its-derivatives-paper3} that, for any $t_{0} \in (0,t) \cap (0,1)$,
\begin{equation}
\begin{aligned}
\left|
\mathbb{E} 
\left[
\phi\left(X^{x_{0}}_{t}\right)
\right]
-
\int_{\mathbb{R}^d} 
\phi(x)
\pi(\mathrm{d} x)
\right|  
&=
\left|
\int_{\mathbb{R}^d} 
P_{t}\phi(x_{0})
-
P_{t}\phi(x)
\pi(\mathrm{d} x)
\right| \\
& =
\left|
\int_{\mathbb{R}^d} 
\mathbb{E}
\left[
P_{t_{0}}
\phi
\left(
X^{x_{0}}_{t-t_{0}}
\right)
-
P_{t_{0}}
\phi
\left(
X^{x}_{t-t_{0}}
\right)
\right]
\pi(\mathrm{d} x)
\right| \\
&\leq
C
\dfrac{\|\phi\|_{0}}{\sqrt{t_{0}}}
\int_{\mathbb{R}^d} 
\mathcal{W}_{1}
\big( \Pi(X^{x_{0}}_{t-t_{0}})
,
\Pi(X^{x}_{t-t_{0}})
\big)
\pi(\mathrm{d} x),
\end{aligned}
\end{equation}
where 
$\mathcal{W}_{1}(\mu, \upsilon) := \inf_{\Pi \in \mathcal{C}(\mu, \upsilon)} 
\int_{\mathbb{R}^{d} \times \mathbb{R}^{d}}
\|x-y \| \dd \Pi(x,y) 
$
denotes the $L^{1}$-Wasserstein distance between the probability distributions $\mu$ and $\upsilon$ with $\mathcal{C}(\mu, \upsilon)$ being the collection of measures on $\mathbb{R}^{d} \times \mathbb{R}^{d}$ having $\mu$ and $\upsilon$ as marginals.
Besides, by \cite[Theorem 1.4]{luo2016exponential}, there exist two constants $c_{\star} = C(\widetilde{a}_{1}, \widetilde{a_{2}}, \mathcal{R}) > 0$ and $C_{\star} = C(\widetilde{a}_{1}, \widetilde{a_{2}}, \mathcal{R}) > 0$ with $\widetilde{a}_{1}$, $\widetilde{a}_{2}$, $\mathcal{R}$ given in Assumption \ref{assumption:contractivity-at-infinity-condition-paper3} such that
\begin{equation}
\mathcal{W}_{1}
(X^{x_{0}}_{t-t_{0}}
,
X^{x}_{t-t_{0}}
)
\leq
C_{\star} e^{-c_{\star}(t-t_{0})}
\mathbb{E}
\left[
\|x_{0} - x\|
\right].
\end{equation}
Equipped with the property of the function $\phi \in C_{b}(\mathbb{R}^{d})$ and the estimate above, we obtain that
\begin{equation}
\begin{aligned}
\left|
\mathbb{E} 
\left[
\phi\left(X^{x_{0}}_{t}\right)
\right]
-
\int_{\mathbb{R}^d} 
\phi(x)
\pi(\mathrm{d} x)
\right| 
&\leq
\min\left\{
2\|\phi\|_{0}, \
C_{\star} 
\dfrac{e^{-c_{\star}(t-t_{0})}}{\sqrt{t_{0}}}
\int_{\mathbb{R}^d} 
\mathbb{E}
\left[
\|x_{0} - x\|
\right]
\pi
(\dd x)
\right\} \\
& \leq
C_{\star}
\|\phi\|_{0}
e^{-c_{\star}t}
\left(
1+ \mathbb{E}[\|x_{0}\|]
\right),
\end{aligned}
\end{equation}
as required.
}
%
\end{proof}
{\color{black}
We remark that the exponential convergence result \eqref{equation: expoenetially decrease of SDE to the invariant measure}
can be also obtained due to \cite[Theorem 2.5]{goldys200512}.
}
{\color{black}
Owing to \eqref{equation:tv-distance-of-two-measures} and Lemma \ref{lemma:Existence-and-uniqueness-of-the-invariant-measure-paper3},
}
it immediately implies $\Pi(X^{x_{0}}_{t})$ 
converges exponentially to an invariant measure in the total variation distance as below,
{\color{black}
\begin{equation}\label{equation:convergence-sode-paper3}
\big\|
\Pi\left(X^{x_{0}}_{t}\right) - \pi 
\big\|_{\text{TV}} 
\leq C_{\star}
\|\phi\|_{0}
{\color{black}
e^{-c_{\star} t}
}
\left(
1+\mathbb{E}
\left[
\|x_{0}\|
\right]
\right), 
\quad \forall t > 0.
\end{equation}
}
To sample from the invariant measure $\pi$ of the Langevin SDE \eqref{eq:langevin-SODE}, in general we require the existence, not the uniqueness, of the invariant measure induced by the PLMC algorithm (see \eqref{equation:error-decomposition-tv-paper3} or \cite{dalalyan2017theoretical,majka2020nonasymptotic,neufeld2022non,pages2023unadjusted,durmus2019high} for example).
The main task of this article is to reveal how  fast the law $\Pi(Y_{n})$ converges to the invariant measure.


\subsection{A priori estimates of the PLMC algorithm}
In the following we give some useful properties of the PLMC algorithm \eqref{equation:numerical-scheme-PLMC-algorithm-paper3}.
\begin{lemma} \label{lemma: useful estimate of PLMC algorithm-paper3}
Let Assumption \ref{assumption:globally-polynomial-growth-condition-paper3} hold. Then, for any $ x\in \mathbb{R}^{d}$,
the following estimates hold true
\begin{equation}
\begin{aligned}
\|\mathscr{P}(x) \| 
&\leq 
\|x\|,
\\
\|\mathscr{P}(x) \| 
&\leq 
 \min
 \left\{ 
 \|x \| , 
 \vartheta
 d^{1/2\gamma} h^{-1/2\gamma}
 \right\},\  \text{for} \ \gamma>1, \\
\| 
f(\mathscr{P}(x)) 
\| 
&\leq 
C_{f}
\left[
\mathbf{1}_{\gamma >1}
\vartheta ^{\gamma}
d^{1/2} h^{-1/2}
+
\mathbf{1}_{\gamma =1}
(1+\| x\|)
\right],
\end{aligned}
\end{equation}
where $C_{f}$ is a constant depending only on the drift $f$.
\end{lemma}
\begin{proof}[Proof of Lemma \ref{lemma: useful estimate of PLMC algorithm-paper3}]
{\color{black}
The proof is straightforward and omitted here. Similar assertions can be found in \cite[Lemma 4.2]{pang2024linear}.
}
\end{proof}
The lemma below provides the uniform moment bounds for the PLMC algorithm \eqref{equation:numerical-scheme-PLMC-algorithm-paper3}.
\begin{lemma}\label{lemma:Uniform-moment-bounds-of-the-PLMC-paper3}(Uniform moment bounds of the PLMC algorithm)
Let Assumptions \ref{assumption:globally-polynomial-growth-condition-paper3}, \ref{assumption:contractivity-at-infinity-condition-paper3}, \ref{assumption:coercivity-condition-of-the-drift-paper3} be fulfilled and let 
{\color{black}
$h\in (0, \min\{1/2a_{1}, 2a_{1}/(a_{1}+2C^{2}_{f}), 1\} )$
}be the uniform timestep.
Let the numerical approximations $\{Y_{n} \}_{n \geq 0}$ {\color{black}be} produced by the PLMC algorithm \eqref{equation:numerical-scheme-PLMC-algorithm-paper3}.
{\color{black}
Then, for 
any $p\in [1,\infty) \cap \mathbb{N}$,
}
{\color{black}
\begin{equation}
\mathbb{E} 
\left[ 
\left\| Y_{n}  \right\|^{2p} 
\right] 
\leq e^{-\frac{a_{1}}{2} t_{n}}
\mathbb{E}\left[
(1 + \|{x}_{0} \|^{2})^{p}
\right] + Cd^{p},
\end{equation}
}
where $t_{n}:=nh$, $n \geq 0$,
{\color{black}
and $C=C(a_{1}, a_{2},C_{f}, \vartheta, \gamma, p)$.
}
\end{lemma}
\begin{proof}[Proof of Lemma \ref{lemma:Uniform-moment-bounds-of-the-PLMC-paper3}]


{\color{black}
By \eqref{equation:numerical-scheme-PLMC-algorithm-paper3} and Assumption \ref{assumption:coercivity-condition-of-the-drift-paper3}, it is straightforward to show that, 
}
for $n \in\{0,1,2, \ldots, N-1\}$, $N\in \mathbb{N}$,
\begin{equation} \label{equation:squared-PLMC-algorithm-paper3}
\begin{aligned}
 \|Y_{n+1} \|^{2} 
 &= \big\|
 \mathscr{P}(Y_{n}) 
 +  f\left(\mathscr{P}(Y_{n})\right) h 
 + \sqrt{2h} \xi_{n+1} 
 \big\|^{2} \\
 & = \|\mathscr{P}(Y_{n}) \|^{2} 
 + 2h \left\langle 
 \mathscr{P}(Y_{n}),  f\left(\mathscr{P}(Y_{n})\right)
 \right\rangle 
 + 2\sqrt{2h} 
 \left\langle \mathscr{P}(Y_{n}),  
 \xi_{n+1} \right\rangle 
 + 2\sqrt{2}h^{3/2}  
 \left\langle f\left(\mathscr{P}(Y_{n})\right),  
 \xi_{n+1} \right\rangle \\
 & \quad + h^{2} \| f\left(\mathscr{P}(Y_{n})\right)\|^{2} + 2h\|\xi_{n+1} \|^{2} \\
&  {\color{black}
\leq
(1-2a_{1}h)
\|\mathscr{P}(Y_{n}) \|^{2}
+
2\sqrt{2h} 
 \left\langle \mathscr{P}(Y_{n}),  
 \xi_{n+1} \right\rangle 
 + 2\sqrt{2}h^{3/2}  
 \left\langle f\left(\mathscr{P}(Y_{n})\right),  
 \xi_{n+1} \right\rangle 
}\\
& \quad
{\color{black}
+ h^{2} \| f\left(\mathscr{P}(Y_{n})\right)\|^{2}
 + 2h\|\xi_{n+1} \|^{2} + 2a_{2}h .}
\end{aligned}
\end{equation}
Using the Young inequality and Lemma \ref{lemma: useful estimate of PLMC algorithm-paper3} yields
\begin{equation}\label{equation:young-inequality-for-crossing-term-paper3}
\begin{aligned}
2\sqrt{2}h^{3/2}  
\left\langle 
f\left(\mathscr{P}(Y_{n})\right),  
\xi_{n+1} 
\right\rangle 
&\leq  
\tilde{\epsilon}
h^{2} 
\|
f\left(
\mathscr{P}(Y_{n})
\right)
\|^{2} 
+ \tfrac{2h}{\tilde{\epsilon}}
\|\xi_{n+1} \|^{2} \\
& 
{\color{black}
\leq 
\tilde{\epsilon}
h^{2}
C^{2}_{f}
\left[
\textbf{1}_{\gamma>1}
\vartheta^{2\gamma}
dh^{-1}
+
\textbf{1}_{\gamma=1}
\left(
1 + 
\left\|
\mathscr{P}(Y_{n})
\right\|
\right)^{2}
\right]
+
\tfrac{2h}{\tilde{\epsilon}}
\|\xi_{n+1} \|^{2}.
}
\end{aligned}
\end{equation}
{\color{black}
Before proceeding further,
we claim that the following estimate holds true for $\gamma \geq 1$,
\begin{equation}  \label{equation: basic expansion of PLMC algorithm-paper3}
\begin{aligned}
1+\|Y_{n+1} \|^{2} 
\leq 
(1-a_{1}h)
\left(
1+\|\mathscr{P}(Y_{n}) \|^{2}
\right)
(1+\Xi_{n+1})  
+ Cdh,
\end{aligned}
\end{equation}
where
\begin{equation}
\Xi_{n+1}:= 
\underbrace{
\dfrac{
2\sqrt{2h}
\left\langle \mathscr{P}(Y_{n}),  \xi_{n+1} 
\right\rangle
}{
(1-a_{1}h)
\left(
1+\|\mathscr{P}(Y_{n}) \|^{2} 
\right) 
}
}_{=:I_{1}} 
+ \underbrace{
\dfrac{
\color{black}
C_{\gamma}h\|\xi_{n+1} \|^{2}
}{
(1-a_{1}h)\left(1+\|\mathscr{P}(Y_{n}) \|^{2} \right) 
}
}_{=:I_{2}},
\end{equation}
{
\color{black}
with
$C_{\gamma}:= \textbf{1}_{\gamma>1}4 + \textbf{1}_{\gamma=1}(4C^{2}_{f}/a_{1} + 2)$
and $C = C(a_{1}, a_{2},C_{f}, \vartheta, \gamma) > 0$.
}
The proof of \eqref{equation: basic expansion of PLMC algorithm-paper3} is divided into two cases depending on different ranges of $\gamma$.
For the case that $\gamma>1$,
we choose $\tilde{\epsilon}=1$ in \eqref{equation:young-inequality-for-crossing-term-paper3} and
obtain that,
for some constant $C = C(a_{1}, a_{2},C_{f}, \vartheta)>0$,
\begin{equation} \label{equation:replace-f-with-F}
\begin{aligned}
 \|Y_{n+1} \|^{2} 
 \leq (1-2a_{1}h)
 \|\mathscr{P}(Y_{n}) \|^{2} 
+ 2\sqrt{2h} 
\left\langle 
\mathscr{P}(Y_{n}),  
\xi_{n+1} 
\right\rangle 
+ 4h\|\xi_{n+1}\|^{2} 
+ Cdh.
\end{aligned}
\end{equation}
Bearing $1-2a_{1}h \leq 1-a_{1}h$ in mind, the treatment for the case that $\gamma > 1$ is thus finished.
For the case that $\gamma=1$,
we recall
\eqref{equation:young-inequality-for-crossing-term-paper3} with $\tilde{\epsilon} = a_{1}/2C^{2}_{f}$ and use the fundamental inequality $(1+\|x \|)^{2} \leq 2(1 + \| x\|^{2})$, $\forall x\in \mathbb{R}^{d}$,
to obtain
\begin{equation}\label{equation:young-inequality-for-crossing-term-paper3-lip}
\begin{aligned}
2\sqrt{2}h^{3/2}  
\left\langle 
f\left(\mathscr{P}(Y_{n})\right),  
\xi_{n+1} 
\right\rangle 
& \leq
\tfrac{a_{1}}{2}
h^{2}
\left(
1 + 
\left\|
\mathscr{P}(Y_{n})
\right\|
\right)^{2}
+
\tfrac{
4C^{2}_{f}
}{
a_{1}
}
h\|\xi_{n+1} \|^{2}
\\
&\leq 
a_{1}
h^{2}
\|
\mathscr{P}(Y_{n})
\|^{2}
+
a_{1}
h^{2}
+
\tfrac{
4C^{2}_{f}
}{
a_{1}
}
h\|\xi_{n+1} \|^{2}.
\end{aligned}
\end{equation}
Moreover, using Lemma \ref{lemma: useful estimate of PLMC algorithm-paper3} and the fundamental inequality again
shows
\begin{equation}
\| f\left(\mathscr{P}(Y_{n})\right)\|^{2}
\leq
2C^{2}_{f} 
\|\mathscr{P}(Y_{n})\|^{2}
+2C^{2}_{f}.
\end{equation}
Combining these estimates with \eqref{equation:squared-PLMC-algorithm-paper3}
and recalling $h\in (0, \min\{1/2a_{1}, 2a_{1}/(a_{1}+2C^{2}_{f}) ,1\} )$
yield, for some constant $C=C(a_{1}, a_{2}, C_{f})$,
\begin{small}
\begin{equation}
\begin{aligned}
\|Y_{n+1} \|^{2} 
 &\leq 
\left[
1-
\left(
2a_{1}-a_{1}h-2C^{2}_{f}h 
\right)h
\right]
 \|\mathscr{P}(Y_{n}) \|^{2} 
+ 2\sqrt{2h} 
\left\langle 
\mathscr{P}(Y_{n}),  
\xi_{n+1} 
\right\rangle 
+ 
\left(
\tfrac{
4C^{2}_{f}
}{
a_{1}
} +2 
\right)
h\|\xi_{n+1}\|^{2} 
+ Ch 
\\
& \leq
(1-a_{1}h)
 \|\mathscr{P}(Y_{n}) \|^{2} 
+ 2\sqrt{2h} 
\left\langle 
\mathscr{P}(Y_{n}),  
\xi_{n+1} 
\right\rangle 
+ 
\left(
\tfrac{
4C^{2}_{f}
}{
a_{1}
} +2 
\right)
h\|\xi_{n+1}\|^{2} 
+ Ch.
\end{aligned}
\end{equation}
\end{small}
}
{\color{black}
Therefore, the estimate of \eqref{equation: basic expansion of PLMC algorithm-paper3} for $\gamma=1$ is finished.
}

Following the binomial expansion theorem and taking the 
{\color{black}
conditional expectation 
}with respect to $\mathcal{F}_{t_{n}}$ on both sides of \eqref{equation: basic expansion of PLMC algorithm-paper3} show that,
\begin{align}\label{equation: conditional expectation of the expansion of PLMC algorithm}
    \begin{split}
 &\mathbb{E}
\left[
(1+\|Y_{n+1} \|^{2})^{p} 
\big| 
\mathcal{F}_{t_{n}} 
\right] \\
&\leq   
\underbrace{
(1-a_{1}h)^{p}
(1 + \|\mathscr{P}(Y_{n}) \|^{2})^{p}  
\mathbb{E} 
\left[
(1 + \Xi_{n+1}) ^{p} \big| \mathcal{F}_{t_{n}} 
\right] 
}_{=:\mathbb{I}_{1}}\\
&\quad + 
\underbrace{
\color{black}
\sum_{\ell=1}^{p} 
\mathcal{C}_{p}^{\ell} 
(Cdh)^{\ell}
(1-a_{1}h)^{p-\ell}  
(
1 + \|\mathscr{P}(Y_{n}) \|^{2}
)^{p-\ell}
\mathbb{E} 
\left[
(1 + \Xi_{n+1}) ^{p-\ell} 
\big| \mathcal{F}_{t_{n}} 
\right] 
}_{=:\mathbb{I}_{2}},       
    \end{split}
\end{align}
{\color{black}
where $\mathcal{C}^{\ell}_{p}:=p !/(\ell ! (p-\ell)!)$.
}
Further analysis is based on estimates of $\mathbb{I}_{1}$ and $\mathbb{I}_{2}$.

\noindent \textbf{For the estimate of $\mathbb{I}_{1}$:}

The key component 
of the estimate $\mathbb{I}_{1}$
is 
\begin{equation}
(1 + \|\mathscr{P}(Y_{n}) \|^{2})^{p}
\mathbb{E} 
\left[
(1 + \Xi_{n+1}) ^{p} \big| \mathcal{F}_{t_{n}} 
\right].
\end{equation}
We use the binomial expansion theorem again to deduce
{\color{black}
\begin{align}
    \begin{split}
       & \mathbb{E} 
\left[
(1 + \Xi_{n+1}) ^{p} \big| \mathcal{F}_{t_{n}} 
\right] 
= \sum_{i=0}^{p}
\mathcal{C}_{p}^{i} 
\mathbb{E}
\left[
\Xi_{n+1}^{i} \big| \mathcal{F}_{t_{n}}
\right]. \\
    \end{split}
\end{align}
}
{\color{black}
Then, the estimate for $\mathbb{I}_{1}$ is decomposed
further into the following three steps.
}
\begin{description}
\item[Step I: estimate of]
$\mathbb{E}\left[\Xi_{n+1} \big| \mathcal{F}_{t_{n}} \right]$

Based on the property of the Gaussian random variable  and the fact that $\xi_{n+1}$ is independent of $\mathcal{F}_{t_{n}}$, we deduce
\begin{equation} \label{equation:property-of-Brownian-motion-paper3}
\begin{aligned}
\mathbb{E} 
\left[ 
\xi_{j,n+1} \big| \mathcal{F}_{t_{n}}  
\right]  =0, 
\quad 
\mathbb{E} 
\left[
| \xi_{j,n+1}|^{2} \big| \mathcal{F}_{t_{n}}  
\right]  =1, 
\quad j \in \{1, \ldots, d\},
\end{aligned}
\end{equation}
resulting in
{\color{black}
\begin{equation} 
\begin{aligned}
\mathbb{E}
\left[
\Xi_{n+1} \big| \mathcal{F}_{t_{n}} 
\right] 
=
\mathbb{E}
\left[
I_{2} | \mathcal{F}_{t_{n}}  
\right]
= \dfrac{C_{\gamma}dh}
{ (1-a_{1}h)\left(1+\|\mathscr{P}(Y_{n}) \|^{2} \right) }.
\end{aligned}
\end{equation}
}

\item[Step II: estimate of]
$\mathbb{E}\left[\Xi^{2}_{n+1} \big| \mathcal{F}_{t_{n}} \right]$

Recalling some power properties of the Gaussian random variable, we derive that, $\forall \ell \in \mathbb{N}$,
\begin{equation} \label{equation:power-property-of-Brownian-motion-paper3}
\begin{aligned}
\mathbb{E} 
\left[
\left(
\xi_{j,n+1}
\right)^{2\ell-1}  
\big| \mathcal{F}_{t_{n}} 
\right] =0,\quad
\mathbb{E} 
\left[
\left(
\xi_{j,n+1}
\right)^{2\ell}  
\big| \mathcal{F}_{t_{n}} 
\right] 
=  (2 \ell -1)!! ,  
\quad \forall \ n \in \mathbb{N},\ j \in \{1, \ldots, d\},
\end{aligned}
\end{equation}
where 
{\color{black}
$(2 \ell -1)!! := \Pi_{i=1}^{\ell}(2i-1)$. 
}
{\color{black}
Since $\mathscr{P}(Y_{n})$ is $\mathcal{F}_{t_{n}}$-measurable and
\begin{equation}
\mathbb{E} 
\left[
\xi_{j,n+1}
\|\xi_{n+1} \|^{2} 
\big| \mathcal{F}_{t_{n}} 
\right] =0,
\quad \forall \ n \in \mathbb{N},\ j \in \{1, \ldots, d\},
\end{equation}
one can obtain
\begin{equation} 
\begin{aligned}
\mathbb{E}
\left[
I_{1}I_{2} \big| \mathcal{F}_{t_{n}}  
\right] 
= 0.
\end{aligned}
\end{equation} 
}
By \eqref{equation:property-of-Brownian-motion-paper3}, we immediately deduce that
{\color{black}
\begin{equation} \label{equation:estimate-of-I1 square}
\begin{aligned}
\mathbb{E}
\left[
(I_{1})^{2} \big| \mathcal{F}_{t_{n}}  
\right] 
&=
\dfrac{
8
h 
\mathbb{E}
\left[
\big|
\sum^{d}_{j=1}
[
\mathscr{P}(Y_{n}) 
]_j
\
\xi_{j, n+1}
\big|^{2}
\Big| \mathcal{F}_{t_{n}} 
\right]
}{
(1-a_{1}h)^{2} (1 + \|\mathscr{P}(Y_{n}) \|^{2})^{2} 
} \\
&=
\dfrac{
8
h 
\sum^{d}_{j=1}
\big|
[
\mathscr{P}(Y_{n}) 
]_j
\big|^{2}
}{
(1-a_{1}h)^{2} (1 + \|\mathscr{P}(Y_{n}) \|^{2})^{2} 
} \\
&\leq 
\dfrac{8h  }{(1-a_{1}h)^{2} (1 + \|\mathscr{P}(Y_{n}) \|^{2}) } ,
\end{aligned}
\end{equation}
where $\mathscr{P}(x) = 
([\mathscr{P}(x)]_1, \cdots, [\mathscr{P}(x)]_d )^{T}$.
}
In addition, for any $\ell \in [2, \infty) \cap \mathbb{N}$, it holds that
\begin{equation} \label{equation:estimate-of-I2-paper3}
\begin{aligned}
\mathbb{E}
\left[ 
|I_{2}|^{\ell} \big| \mathcal{F}_{t_{n}}  
\right] 
&
{\color{black} \leq}
\dfrac{ \color{black} C_{\gamma}^{\ell} \times (2\ell-1)!! \times
d^{\ell}h^{\ell} 
}{
(1-a_{1}h)^{\ell} 
(
1 + \|\mathscr{P}(Y_{n}) \|^{2}
)^{\ell} 
}.
\end{aligned}
\end{equation}
In summary, we have
{\color{black}
\begin{equation} 
\begin{aligned}
\mathbb{E}
\left[
\Xi_{n+1}^{2} \big| \mathcal{F}_{t_{n}} 
\right] 
&= \mathbb{E}
\left[
I_{1}^2 | \mathcal{F}_{t_{n}}  
\right]
+
\mathbb{E}
\left[
I_{2}^2 | \mathcal{F}_{t_{n}}  
\right] \\
&\leq 
\dfrac{8h   }
{(1-a_{1}h)^{2}
(
1 + \|\mathscr{P}(Y_{n}) \|^{2}
) }
+
\dfrac{ 3 C^{2}_{\gamma} d^{2}h^{2}  }
{
(1-a_{1}h)^{2} 
(
1 + \|\mathscr{P}(Y_{n}) \|^{2}
)^{2} 
}
.
\end{aligned}
\end{equation}
}


\item[Step III: estimate of] $\mathbb{E}\left[\Xi^{\ell}_{n+1} \big| \mathcal{F}_{t_{n}} \right]$, $\ell \in [3,p] \cap \mathbb{N}$

It follows from \eqref{equation:power-property-of-Brownian-motion-paper3} and the Cauchy-Schwarz inequality that, for $\ell \in [3,p) \cap \mathbb{N}$,
\begin{equation} \label{equation:estimate-of-I1}
\begin{aligned}
\mathbb{E}
\left[ 
|I_{1}|^{\ell}
\big| 
\mathcal{F}_{t_{n}}  
\right] 
&\leq 
\dfrac{
{\color{black}
C
h^{\ell/2}
\|\mathscr{P}(Y_{n}) \|^{\ell}
\mathbb{E} \left[
\| \xi_{n+1}\|^{\ell}
\big| 
\mathcal{F}_{t_{n}} 
\right]
}
}{
(1-a_{1}h)^{\ell} 
(
1 + \|\mathscr{P}(Y_{n}) \|^{2}
)^{\ell } 
}  
\\
&\leq
\dfrac{
{
\color{black}
C d^{\ell/2}h^{\ell/2}  
}
}{
(1-a_{1}h)^{\ell} 
(
1 + \|\mathscr{P}(Y_{n}) \|^{2}
)^{\ell/2 } 
} .
\end{aligned}
\end{equation}
{\color{black}
Bearing \eqref{equation:estimate-of-I2-paper3} and the fundamental inequality $(a + b)^{\ell} \leq 2^{\ell-1}(a^{\ell} + b^{\ell})$, for $a, b \in \mathbb{R}$, $\ell \in [1,p] \cap \mathbb{N}$, in mind, 
}
we deduce 
\begin{equation} 
\begin{aligned}
\mathbb{E}
\left[
\Xi_{n+1}^{\ell} \big| \mathcal{F}_{t_{n}} 
\right] 
&= \mathbb{E}
\left[ 
(I_{1}+I_{2})^{\ell} \big| \mathcal{F}_{t_{n}}  
\right] \\
&\leq 2^{\ell-1} 
\Big(
\mathbb{E}
\left[
|I_{1}|^{\ell} \big| \mathcal{F}_{t_{n}} 
\right] 
+ \mathbb{E}\left[
|I_{2}|^{\ell} \big| \mathcal{F}_{t_{n}} 
\right] 
\Big) \\
&\leq 
\dfrac{
Cd^{\ell/2}h^{\ell/2}  
}{(1-a_{1}h)^{\ell} 
(1 + \|\mathscr{P}(Y_{n}) \|^{2})^{\ell/2} }
+
\dfrac{Cd^{\ell}h^{\ell}  }{(1-a_{1}h)^{\ell} (1 + \|\mathscr{P}(Y_{n}) \|^{2})^{\ell} }  .
\end{aligned}
\end{equation}
\end{description}
Combining \textbf{Step I}$\sim $\textbf{Step III} and the fact $(1-a_{1}h)^{\ell} \geq (1-a_{1}h)^{p}$, $\ell \in [1,p] \cap \mathbb{N}$, yields, 
{\color{black}
for some constants $C=C(p)$,
}
\begin{small}
\begin{equation} \label{equation:prior-estimate-of-mathbb-I1-paper3}
\begin{aligned}
\mathbb{I}_{1}
&
{\color{black}
\leq (1-a_{1}h)^{p}(
1 + \|\mathscr{P}(Y_{n}) \|^{2}
)^{p} 
\Bigg(
1
+
\dfrac{
p
C_{\gamma}
dh
}
{
(1-a_{1}h)\left(1+\|\mathscr{P}(Y_{n}) \|^{2} \right) 
}
+ 
\dfrac{ 4p(p-1)h  }
{(1-a_{1}h)^{2} (1 + \|\mathscr{P}(Y_{n}) \|^{2}) }
}
\\
&\quad
{\color{black}
+
\dfrac{ C^{2}_{\gamma}
p(p-1)d^{2}h^{2}  }
{2(1-a_{1}h)^{2} (1 + \|\mathscr{P}(Y_{n}) \|^{2})^{2} }
+
\sum_{\ell=3}^{p}
\dfrac{
Cd^{\ell/2}h^{\ell/2} 
}{(1-a_{1}h)^{\ell} 
(1 + \|\mathscr{P}(Y_{n}) \|^{2})^{\ell/2 } }
+
\dfrac{Cd^{\ell}h^{\ell}  }{(1-a_{1}h)^{\ell} (1 + \|\mathscr{P}(Y_{n}) \|^{2})^{\ell} } 
\Bigg)
} \\
& \leq
(1-a_{1}h)^{p}
(
1 
+ 
\|\mathscr{P}(Y_{n}) \|^{2}
)^{p} 
+ Cdh 
(
1 
+ 
\|\mathscr{P}(Y_{n}) \|^{2}
)^{p-1}
+ Cd^{2}h^{2} 
(
1 
+ 
\|\mathscr{P}(Y_{n}) \|^{2}
)^{p-2} \\
&\quad 
+ \sum_{\ell=3}^{p}
Cd^{\ell/2}
h^{\ell/2} 
(
1 
+ 
\|\mathscr{P}(Y_{n}) \|^{2}
)^{p-\ell/2}
+ 
Cd^{\ell}h^{\ell} 
(
1 
+ 
\|\mathscr{P}(Y_{n}) \|^{2}
)^{p-\ell}
.
\end{aligned}
\end{equation}
\end{small}
{\color{black}
For any $k \in [1,p]$ and $\epsilon_{1} \in (0, a_{1}/4p)$, 
employing the Young inequality yields
}
\begin{equation}
\begin{aligned}
Cd^{k} (1 + \|\mathscr{P}(Y_{n}) \|^{2})^{p-k}
\leq 
\epsilon_{1} (1 + \|\mathscr{P}(Y_{n}) \|^{2})^{p} 
+ C_{\epsilon_{1}}d^{p},
\end{aligned}
\end{equation}
{\color{black}
where $C_{\epsilon_{1}} = C(a_{1}, C_{f}, p, \gamma) > 0$.
}
Therefore, the estimate \eqref{equation:prior-estimate-of-mathbb-I1-paper3} can be rewritten as
\begin{equation} 
\begin{aligned}
\mathbb{I}_{1}
&\leq 
(1-a_{1}h)^{p}(
1 + \|\mathscr{P}(Y_{n}) \|^{2}
)^{p} 
+ p\epsilon_{1} h (
1 + \|\mathscr{P}(Y_{n}) \|^{2}
)^{p} 
+
C_{\epsilon_{1}}d^{p} h.
\end{aligned}
\end{equation}

\noindent \textbf{For the estimate of $\mathbb{I}_{2}$:}

To handle $\mathbb{I}_{2}$, one follows similar arguments of treating $\mathbb{I}_{1}$, 
{\color{black}
but does not require as fine an estimate as $\mathbb{I}_1$.}
{\color{black} For
{\color{black}
the step size $h\in (0,1)$ and}
$\ell \in [1,p] \in \mathbb{N}$, 
there exist some constants $C = C(p) > 0$ such that,
\begin{equation}\label{equation:rough-estimate-II2-revision-paper3}
\begin{aligned}
&(
1-a_{1}h
)^{p-\ell}
(
1 + \|\mathscr{P}(Y_{n}) \|^{2}
)^{p-\ell}
\mathbb{E} 
\left[
(1 + \Xi_{n+1}) ^{p-\ell} 
\big| \mathcal{F}_{t_{n}} 
\right]  \\
&\leq
(1-a_{1}h)^{p-\ell}
(
1 
+ 
\|\mathscr{P}(Y_{n}) \|^{2}
)^{p-\ell} 
+ Cdh 
(
1 
+ 
\|\mathscr{P}(Y_{n}) \|^{2}
)^{p-\ell-1}
+ Cd^{2}h^{2} 
(
1 
+ 
\|\mathscr{P}(Y_{n}) \|^{2}
)^{p-\ell-2} \\
&\quad 
+ \sum_{i=3}^{p-\ell}
Cd^{i/2}
h^{i} 
(
1 
+ 
\|\mathscr{P}(Y_{n}) \|^{2}
)^{p-\ell - i/2}
+ 
Cd^{i}h^{i} 
(
1 
+ 
\|\mathscr{P}(Y_{n}) \|^{2}
)^{p-\ell-i} \\
& \leq
C
\left[
(
1
+ 
\|\mathscr{P}(Y_{n}) \|^{2}
)^{p-\ell}
+
d^{p-\ell}
\right],
\end{aligned}
\end{equation}
where we have used the Young inequalities as
\begin{equation}
\begin{aligned}
\color{black}
a^{i/2} b ^{p-\ell - i/2}
\leq
\tfrac{ia^{p-\ell}}{2(p-\ell)}
+
\tfrac{(p-\ell-i/2)b^{p-\ell}}{p-\ell}, \quad
a^{i} b ^{p-\ell - i} \leq 
\tfrac{ia^{p-\ell}}{p-\ell}
+
\tfrac{(p-\ell-i)b^{p-\ell}}{p-\ell}, \quad
a, b \geq 1.
\end{aligned}
\end{equation}
Putting the estimate \eqref{equation:rough-estimate-II2-revision-paper3} into $\mathbb{I}_{2}$ leads to,
for some constant $C = C(a_{1}, a_{2},C_{f}, \vartheta, \gamma, p) > 0$,
\begin{equation} 
\begin{aligned}
\mathbb{I}_{2}
\leq
Ch
\sum_{\ell=1}^{p}
d^{\ell}
(1 + \|\mathscr{P}(Y_{n}) \|^{2})^{p-\ell} 
+ Chd^{p}.
\end{aligned}
\end{equation}
}
{\color{black}
In addition,
using the Young inequality yields, 
}
for $\ell \in [1,p] \cap \mathbb{N}$ and  a sufficiently small constant 
{\color{black}
$\epsilon_{2} \in (0, a_{1}/4p)$,
}
\begin{equation}
Ch
d^{\ell}
(1 + \|\mathscr{P}(Y_{n}) \|^{2})^{p-\ell}
\leq
\epsilon_{2}
h
(
1 + \|\mathscr{P}(Y_{n}) \|^{2}
)^{p}
+ 
C_{\epsilon_{2}}
d^{p}
h.
\end{equation}
Consequently, we derive the final estimate of $\mathbb{I}_{2}$ as below,
\begin{equation} 
\begin{aligned}
\mathbb{I}_{2}
&\leq  p\epsilon_{2}
h (1 + \|\mathscr{P}(Y_{n}) \|^{2})^{p} 
+ C_{\epsilon_{2}}d^{p}h,
\end{aligned}
\end{equation}
{\color{black}
where $C_{\epsilon_{2}} = C(a_{1}, a_{2},C_{f}, \vartheta, \gamma, p) >0 $.
}

\noindent \textbf{Combining the estimates of $\mathbb{I}_{1}$ and $\mathbb{I}_{2}$:}

Plugging the estimates of $\mathbb{I}_{1}$ and $\mathbb{I}_{2}$ and taking expectations on the both sides yield
{\color{black}
\begin{equation} 
\begin{aligned}
\mathbb{E}
\left[ 
(1+\|Y_{n+1} \|^{2})^{p}  
\right] 
&\leq 
\big[
1-a_{1}h
+p
(
\epsilon_{1}
+\epsilon_{2}
)h
\big]
\mathbb{E}
\left[
(1 + \|\mathscr{P}(Y_{n}) \|^{2})^{p}
\right]
+ C_{\epsilon_{1},\epsilon_{2}}d^{p}h \\
&
\leq
\left(
1-
\tfrac{a_{1}}{2}
h
\right)
\mathbb{E}
\left[
(1 + \|\mathscr{P}(Y_{n}) \|^{2})^{p}
\right]
+ C_{\epsilon_{1},\epsilon_{2}}d^{p}h,
\end{aligned}
\end{equation}
}
{\color{black}
where $\epsilon_{1} + \epsilon_{2} \in (0, a_{1}/2p)$.
Recalling
$h\in (0, \min\{1/2a_{1}, 2a_{1}/(a_{1}+2C^{2}_{f}), 1\} )$ and using
Lemma \ref{lemma: useful estimate of PLMC algorithm-paper3} we show,
for $C_{\epsilon_{1}, \epsilon_{2}} = C(a_{1}, a_{2},C_{f}, \vartheta, \gamma, p) >0 $,
}
{\color{black}
\begin{equation} 
\begin{aligned}
\mathbb{E}
\left[ 
(1+\|Y_{n+1} \|^{2})^{p}  
\right] 
&\leq 
\left(
1-
\tfrac{a_{1}}{2}
h
\right)
\mathbb{E}
\left[
(1 + \|{Y}_{n} \|^{2})^{p}
\right]
+ C_{\epsilon_{1},\epsilon_{2}}d^{p}h \\
& = 
\left(
1-
\tfrac{a_{1}}{2}
h
\right)^{n+1}
\mathbb{E}\left[(1 + \|{x}_{0} \|^{2})^{p}\right] 
+ \sum_{i=0}^{n}
\left(
1-
\dfrac{a_{1}}{2}
h
\right)^{i}
C_{\epsilon_{1},\epsilon_{2}}d^{p}h \\
& \leq e^{-\tfrac{a_{1}}{2} t_{n+1}}
\mathbb{E}\left[
(1 + \|{x}_{0} \|^{2})^{p}
\right] 
+ 
\tfrac{2C_{\epsilon_{1},\epsilon_{2}}}{a_{1}}
d^{p},
\end{aligned}
\end{equation}
}
where we have used the fact that for any $x>0$, $1-x \leq e^{-x}$.
The proof is completed.
\end{proof}

\section{Kolmogorov equation and regularization estimates}
\label{section:Kolmogorov-equation-and-regularization-estimates}
To carry out the weak error analysis, 
for $\phi \in C_{b}(\mathbb{R}^{d})$ we introduce the function $u(\cdot,\cdot): [0, \infty) \times \mathbb{R}^{d} \rightarrow \mathbb{R}$ defined by
\begin{equation} \label{equation:def-of-u-paper3}
    u(t, x) :=  \mathbb{E}\left[\phi\left(X^{x}_{t} \right) \right],
\end{equation}
and 
{\color{black}
consider the  Kolmogorov equation associated with SDE \eqref{eq:langevin-SODE}:
\begin{equation}
\label{equation:kolmogorov-equation-paper3}
\begin{split}
\left\{
\begin{array}{ll}
\partial_{t} u(t,x) 
= Du(t,x)f(x) 
+  \sum_{j=1}^{d}
D^{2}u(t,x) \big(e_{j} , e_{j} \big), 
\quad t>0,\ x \in \mathbb{R}^{d},
    \\
  u(0, x) = \phi(x), \quad x \in \mathbb{R}^{d},
 \end{array}\right.
 \end{split}
\end{equation}
}
where $\{ e_{j}\}_{j\in \{1, \cdots,d \} }$ is denoted as the orthonormal basis of $\mathbb{R}^{d}$.
{\color{black}
Moreover, we recall the definition of the classical solution to \eqref{equation:kolmogorov-equation-paper3} (see \cite[Definition 1.6.1]{cerrai2001second} for details) as follows.
}
{\color{black}
\begin{definition}
(Classical solutions)
A function $u:[0,+\infty) \times \mathbb{R}^d \rightarrow \mathbb{R}$ is called to be a classical solution of the Kolmogorov equation \eqref{equation:kolmogorov-equation-paper3} if
\begin{enumerate}
\item[(i)] $u(\cdot, \cdot)$ is continuous on $[0,+\infty) \times \mathbb{R}^d$,

\item[(ii)] for any $t>0, u(t, \cdot) \in C_b^2\left(\mathbb{R}^d\right)$,

\item[(iii)]
for any $x \in \mathbb{R}^d$, $u(\cdot, x):(0,+\infty) \rightarrow \mathbb{R}$ is continuously differentiable with continuous derivative $\partial_{t}u(t,x)$,

\item[(iv)]
$\partial_{t}u(t,x)$, $Du(t,x)$ and $D^{2}u(t,x)$ are continuous from $(0,+\infty) \times \mathbb{R}^d$ into $\mathbb{R}$, $\mathbb{R}^{d}$ and $\mathbb{R}^{d \times d}$, respectively.

\item[(v)] 
$u(\cdot, \cdot)$ verifies the problem \eqref{equation:kolmogorov-equation-paper3}.
\end{enumerate}
\end{definition}
}
%
{\color{black}
Owing to \cite[Theorem 1.6.2]{cerrai2001second},
$u(t,x)$ is the unique
 classical
solution of the Kolmogorov equation \eqref{equation:kolmogorov-equation-paper3} under previous assumptions.
In the following, we attempt to derive 
more delicate estimates of $u$ and its derivatives,
which are essential for the subsequent error analysis,
are absent in \cite{cerrai2001second}.
Before proceeding further,
let us
revisit the mean-square differentiability of random functions,
quoted from \cite{10.1093/imanum/drad083,pang2024linear}.
}
\begin{definition}(Mean-square differentiable)
Let $\Psi: \Omega \times \mathbb{R}^{d} \rightarrow \mathbb{R}$ and $\psi_{i}: \Omega \times \mathbb{R}^{d} \rightarrow \mathbb{R}$ be random functions satisfying
\begin{equation}
\lim _{\tau_{1} \rightarrow 0} 
\mathbb{E}
\left[
\left|
\frac{1}{\tau_{1}}
\left[
\Psi\left(x+\tau_{1} e_i\right)-\Psi(x)
\right]
-\psi_i(x)
\right|^2
\right]=0, 
\quad \forall i \in\{1,2, \cdots, d\},
\end{equation}
where $e_{i}$ is the  unit vector in $\mathbb{R}^{d}$ with the $i$-th element being  $1$.  Then $\Psi$ is called to be mean-square differentiable,  with $\psi = (\psi_{1}, \dots, \psi_{d})$  being the derivative (in the mean-square differentiable sense) of $\Psi$ at $x$.
Also, we denote $\mathcal{D}_{(i)} \Psi = \psi_{i}$  and  $\mathcal{D} \Psi(x) =\psi $.
\end{definition}
The above definition can be generalized to vector-valued functions in a component-wise manner.
As a result,
for every $t \in [0, \infty)$, we take the function $X^{(\cdot)}_t: \mathbb{R}^d\to \mathbb{R}^d$, and write its derivative as $\mathcal{D}X^{x}_{t}\in L(\mathbb{R}^d,\mathbb{R}^d)$, 
$\mathcal{D}^{2}X^{x}_{t}\in L(\mathbb{R}^d, L(\mathbb{R}^d,\mathbb{R}^d))$.
The following lemma states some a priori estimates of the mean-square derivative of solutions $\{X^{x}_{t}\}_{t\geq 0}$.

\begin{lemma} \label{lemma:differentiability-of-solutions-paper3}
Let Assumptions \ref{assumption:globally-polynomial-growth-condition-paper3},  \ref{assumption:contractivity-at-infinity-condition-paper3} and  \ref{assumption:coercivity-condition-of-the-drift-paper3} hold. Then, for all $ t \in [0,\infty)$ and $x, v_{1}, v_{2} \in \mathbb{R}^{d}$,
\begin{align}
\left\| \mathcal{D}X^{x}_{t} v_{1}\right\| 
& \leq e^{Lt}  \left\|  v_{1}\right\|, 
\\
%
{\color{black}
\|
\mathcal{D}^{2}X^{x}_{t} (v_{1}, v_{2})
\|_{L^{2} (\Omega, \mathbb{R}^{d})}
\leq 
C e^{(3L+1/2)t}
}
&
{\color{black}
\left[
\textbf{1}_{\gamma \in [1,2]}
+
\textbf{1}_{\gamma \in (2, \infty)} 
(d^{\gamma/2-1}
+
\|x\|^{\gamma-2})
\right]
\left\|
v_1 
\right\|
\cdot
\left\|
v_2
\right\|
}, 
\end{align}
where
$L>0$ comes from \eqref{eq:remark2.3-paper3}.

\end{lemma}
The proof of Lemma \ref{lemma:differentiability-of-solutions-paper3} is 
{\color{black}
postponed
}
to Appendix \ref{proof:lemma:differentiability-of-solutions-paper3}.
We remark that
{\color{black}
in existing results
}
with the strongly-convex condition \eqref{equation:strongly-convex-u},  one can prove the mean-square derivatives of solutions $\{X^{x}_{t} \}_{t\geq 0}$ 
{\color{black}
decrease
}
exponentially in time $t$ and use the chain rule to obtain 
{\color{black}
the regularity properties}  of $u(t,\cdot)$ and its derivatives for some smooth enough test functions (see \cite{pang2024linear}).
However, such arguments do not work in our setting. On the one hand, as shown by Lemma \ref{lemma:differentiability-of-solutions-paper3}, the  derivatives of  $\{X^{x}_{t} \}_{t\geq 0}$ increase exponentially in time. 
{\color{black}
On the other hand, 
}
the chain rule breaks down for non-smooth test functions $\phi \in C_{b}(\mathbb{R}^{d})$.
Here we resort to
the Bismut-Elworthy-Li formula (see \cite{cerrai2001second}). 
\begin{lemma} \label{lemma:Bismut-Elworthy-Li-formula-paper3}
(Bismut-Elworthy-Li formula \cite{cerrai2001second,elworthy1994formulae})
Let $\{X^{x}_{t}\}_{t\geq 0}$  be the solutions of the Langevin SDE \eqref{eq:langevin-SODE} with the initial state $X^{x}_{0}=x$. If $\{X^{x}_{t}\}_{t\geq 0}$ is mean-square differentiable,
then for any $\varphi \in C_{b}(\mathbb{R}^{d})$, $t > 0$ and $x, v \in \mathbb{R}^{d}$, 
\begin{equation}
\left\langle
D\big(
P_{t}\varphi(x)
\big),
v
\right\rangle
=
\dfrac{1}{\sqrt{2}t} 
\mathbb{E} 
\left[
\int_{0}^{t} \big\langle 
\mathcal{D}X^{x}_{s} v, \mathrm{d} W_{s} 
\big\rangle 
\varphi\left(X^{x}_{t} \right)  
\right],
\end{equation}
where $P_{t}(\cdot)$ defines a transition semigroup of the corresponding SDEs \eqref{eq:langevin-SODE} as $P_{t}\varphi(x) = \mathbb{E}[\varphi(X^{x}_{t})]$.
\end{lemma}
Recalling $P_{t}\phi(x) = u(t,x)$ by \eqref{equation:def-of-u-paper3},
the Bismut-Elworthy-Li formula paves the way  to presenting an expression for the derivatives of $u(t,\cdot)$ where it only  requires $\phi \in C_{b}(\mathbb{R}^{d})$.
Combining Lemma \ref{lemma:differentiability-of-solutions-paper3} with the Bismut-Elworthy-Li formula, 
we are now in the position to obtain some regularization estimates with regard to the derivatives of $u(t,\cdot)$, $t>0$ .
\begin{theorem} \label{theorem:regular-estimate-of-u-and-its-derivatives-paper3}
If Assumptions \ref{assumption:globally-polynomial-growth-condition-paper3},  \ref{assumption:contractivity-at-infinity-condition-paper3} and  \ref{assumption:coercivity-condition-of-the-drift-paper3} hold, then the function $u(t,\cdot) \in C^{2}_{b}(\mathbb{R}^{d})$.
Moreover,
{\color{black}
for $ t \in (0,1]$,
}
\begin{equation}\label{equation:Du-in-theorem-paper3}
\begin{aligned}
\left|
Du(t,x) v_{1} 
\right|
\leq 
{\color{black}
C
\dfrac{\|\phi\|_{0}}{\sqrt{t}} 
}
\|v_{1} \|, 
\quad \forall x, v_{1} \in \mathbb{R}^{d}, 
\end{aligned}
\end{equation}
and
{\color{black}
\begin{equation}\label{equation:D2u-in-theorem-paper3}
\begin{aligned}
\left|
D^{2}u(t,x)(v_{1}, v_{2}) 
\right|
\leq
C
\dfrac{\|\phi\|_{0}}{t} 
\left[
\textbf{1}_{\gamma \in [1,2]}
+
\textbf{1}_{\gamma \in (2, \infty)} 
(d^{\gamma/2-1}+\|x\|^{\gamma-2})
\right] 
\|v_{1}\| \cdot \|v_{2}\|,
\quad
\forall
x, v_{1}, v_{2} \in \mathbb{R}^{d}.
\end{aligned}
\end{equation}
}
%
{\color{black}
For $t\in (1,\infty)$, 
}
{\color{black}
\begin{equation} \label{equation:Du-for-t>1-paper3}
\begin{aligned}
\big| 
Du(t,x) v_{1} 
\big|
\leq 
C_{\star}
\|\phi\|_{0}
e^{-c_{\star}(t-1)} 
\left(
d^{1/2} + \|x\|
\right)
\|v_{1}\|, 
\quad \forall x, v_{1} \in \mathbb{R}^{d},
\end{aligned}
\end{equation}
}
and
{\color{black}
\begin{equation}\label{equation:D2u-for-t>1-paper3}
\begin{aligned}
%
&\left|
D^{2}u(t,x)(v_{1}, v_{2}) 
\right|
\leq 
C_{\star}
\|\phi\|_{0}
e^{-c_{\star}(t-1)} 
(d^{1/2} + \|x\|)
\left[
\textbf{1}_{\gamma \in [1,2]}
+
\textbf{1}_{\gamma \in (2, \infty)} 
(d^{\gamma/2-1}+\|x\|^{\gamma-2})
\right] 
\|v_{1}\| \cdot \|v_{2}\|, \\
& \hspace{32em}
\forall x, v_{1}, v_{2} \in \mathbb{R}^{d}.  \\
\end{aligned}
\end{equation}
}
\end{theorem}

\begin{proof}[Proof of Theorem \ref{theorem:regular-estimate-of-u-and-its-derivatives-paper3}]
    The Bismut-Elworthy-Li formula  (see Lemma \ref{lemma:Bismut-Elworthy-Li-formula-paper3} or \cite{cerrai2001second,elworthy1994formulae}) states that for some function $\Phi: \mathbb{R}^{d}\rightarrow \mathbb{R}$ belonging to $C_{b}(\mathbb{R}^{d})$, there exists $\mathcal{K}(\Phi, x)>0$, which depends on $\Phi$ and $x$, such that,
    \begin{equation} \label{equation:definition-of-Phi-paper3}
    |\Phi(x)| \leq \mathcal{K}(\Phi, x),\quad \forall x\in \mathbb{R}^{d},
    \end{equation}
 then we can calculate the first and the second derivative of 
\begin{equation}
{\color{black}
V(t,x)
}
:= 
\mathbb{E}[\Phi(X^{x}_{t})] 
\end{equation}
with respect to $x$.
{\color{black}
Recalling \eqref{equation:notation-of-the-derivatives-paper3},
we use the same notations $\eta^{v_{1}}(t,x)$ and $\xi^{v_{1}, v_{2}}(t,x)$ to denote the first and the second derivatives of $X^{x}_{t}$ with respect to $x$ for $t>0$, respectively. 
}
Indeed, we have
\begin{equation} \label{equation:Bismut-Elworthy-Li-formula-of-the-first-derivative-of-U}
\begin{aligned}
D
{\color{black}
V(t,x)
}
v_{1}
&=\dfrac{1}{\sqrt{2}t} 
\mathbb{E} 
\left[
\int_{0}^{t} 
\left\langle
\eta^{v_{1}}(s,x), \dd W_{s} 
\right\rangle 
\Phi\left(
X^{x}_{t} 
\right) 
\right]. \\
\end{aligned}
\end{equation}
{\color{black}
By \eqref{equation:definition-of-Phi-paper3}, Lemma \ref{lemma:differentiability-of-solutions-paper3}, the H\"older inequality and the It\^o isometry, we obtain that, for any $ v_{1} \in \mathbb{R}^{d}$,
\begin{equation}
\label{equation:1st-derivative-of-U-paper3}
\begin{aligned}
\big|
D
V(t,x)
v_{1} 
\big|
&  \leq 
\dfrac{1}{\sqrt{2}t} 
\left\|
\int_{0}^{t} 
\left\langle
\eta^{v_{1}}(s,x), \dd W_{s} 
\right\rangle 
\Phi\left(
X^{x}_{t} 
\right) 
\right\|_{L^{1}(\Omega, \mathbb{R})}  \\
& \leq 
\dfrac{
\| 
\Phi\left(
X^{x}_{t} 
\right)
\|_{L^{2}(\Omega, \mathbb{R})}
}{\sqrt{2}t} 
\Big\|
\int_{0}^{t } 
\left\langle 
\eta^{v_{1}}(s,x), \dd W_{s} 
\right\rangle 
\Big\|_{L^{2}(\Omega,\mathbb{R})} \\
&
 \leq
\dfrac{
\| \mathcal{K}(\Phi, X^{x}_{t})\|_{L^{2}(\Omega, \mathbb{R})}
}{\sqrt{2}t} 
\left(
\int_{0}^{t }
\left\|
\eta^{v_{1}}(s,x) 
\right\|^{2}_{L^{2}(\Omega,\mathbb{R}^{d})} 
\dd s
\right)^{1/2}\\
& \leq
\dfrac{
\| \mathcal{K}(\Phi, X^{x}_{t})\|_{L^{2}(\Omega, \mathbb{R})}
}{\sqrt{2}t} 
\left(
\int_{0}^{t } 
e^{2Ls} \dd s
\right)^{1/2} 
\| v_{1}\| \\
& 
= 
\dfrac{
(
 e^{2Lt}-1 
)^{1/2} 
}{2L^{1/2}t} 
\| 
\mathcal{K}(\Phi, X^{x}_{t})
\|_{L^{2}(\Omega, \mathbb{R})}
\| v_{1}\|.
\end{aligned}
\end{equation}
}
{\color{black}
As a consequence of \eqref{equation:1st-derivative-of-U-paper3},
one can derive the strong Feller property of $u(t, \cdot)$ corresponding to the Langevin SDEs \eqref{eq:langevin-SODE}.
Indeed, one can choose 
$\Phi = \phi \in C_{b}(\mathbb{R}^{d})$ in \eqref{equation:1st-derivative-of-U-paper3} and thus $\mathcal{K}(\Phi,x) = \| \phi\|_{0}$. Therefore one deduces
}
\begin{equation} \label{eq:strong-feller-property-paper3}
    \left|
    u(t,x_{1}) - u(t,x_{2})
    \right|
    \leq 
    {\color{black}
    C(t) \|\phi\|_{0} }
    \cdot
    \|x_{1}-x_{2} \|,
    \quad
    \forall x_{1}, x_{2} \in \mathbb{R}^{d},
    \quad t>0,
\end{equation}
{\color{black}
with 
$C(t)=(
 e^{2Lt}-1 
)^{1/2}/2L^{1/2}t$.
Such a constant $C(t)=\mathcal{O}(1/\sqrt{t})$ 
when $t \rightarrow 0$ and increases exponentially with respect to $t$ when $t \rightarrow \infty$.
}
{\color{black}
This implies that given any function $\phi \in C_{b}(\mathbb{R}^{d})$, $u(t, \cdot) = P_{t}\phi(\cdot)$ is Lipschitz continuous, leading to the strong Feller property of $P_{t}(\cdot)$, for any $ t > 0$.
}

Going back to the proof,
the Markov property of SDE \eqref{eq:langevin-SODE} 
{\color{black}
yields
}
\begin{equation}
{\color{black}
V(t,x)
}
=
\mathbb{E}
\left[
{\color{black}
V
}
\left(
t/2,X^{x}_{t/2 } 
\right) 
\right]
{\color{black}
=
P_{t/2}
\big(
V
\left(
t/2, \cdot 
\right)
\big)
\left(
X^{x}_{t/2 }
\right)
}
,
\end{equation}
which leads to the following expression with respect to $D{\color{black}
V
}(t,\cdot)$,
\begin{equation} \label{equation:first-derivative-of-U-version-2-paper3}
D{\color{black}
V
}(t,x) v_{1}
=
\dfrac{2}{\sqrt{2}t} 
\mathbb{E} 
\left[
\int_{0}^{\frac{t}{2} } 
\left\langle 
\eta^{v_{1}}(s,x), 
\dd W_{s} 
\right\rangle  
{\color{black}
V
}
\left(
t/2,X^{x}_{t/2 } 
\right)
\right].
\end{equation}
{\color{black}
By \eqref{equation:1st-derivative-of-U-paper3} and Lemma \ref{lemma:differentiability-of-solutions-paper3},
one notes that 
both functions 
${\color{black}
V
}(t, \cdot)$ and $X^{(\cdot)}_{t}$ are continuously differentiable,
implying that  the function
$
V(
t/2,X^{x}_{t/2 } 
)$, $t>0$, is also continuously differentiable. 
Therefore, we employ the chain rule to obtain first derivative of
${\color{black}
V
}(
t/2,X^{x}_{t/2 } 
)$ as
\begin{equation}
D
\left(
{\color{black}
V
}(
t/2,X^{x}_{t/2 } 
)
\right) 
{
\color{black}v_{2}
}
=
D
{\color{black}
V
}(
t/2,X^{x}_{t/2 } 
)
\eta^{v_{2}}(t/2,x).
\end{equation}
}
Accordingly, one can compute the second derivative of ${\color{black}
V
}(t,\cdot)$ as follows:
\begin{equation}
\begin{aligned}
D^{2}
{\color{black}
V
}(t,x) (v_{1}, v_{2})
&=  
\dfrac{2}{\sqrt{2}t} 
\mathbb{E} 
\left[
\int_{0}^{\frac{t}{2} } 
\left\langle 
\xi^{v_{1}, v_{2}}(s,x), \dd W_{s} 
\right\rangle  
{\color{black}
V
}\left(
t/2,X^{x}_{t/2 } 
\right) 
\right] \\
& \quad 
+ \dfrac{2}{\sqrt{2}t} 
\mathbb{E} 
\left[
\int_{0}^{\frac{t}{2} } 
\left\langle 
\eta^{v_{1}}(s,x), \dd W_{s} 
\right\rangle  
D
{\color{black}
V
}\left(
t/2,X^{x}_{t/2 } 
\right)
\eta^{v_{2}}(t/2,x) 
\right].
\end{aligned}
\end{equation}
{\color{black}
Due to \eqref{equation:1st-derivative-of-U-paper3}, Lemma \ref{lemma:differentiability-of-solutions-paper3}, 
the H\"older inequality and the It\^o isometry, one has
}
{\color{black}
\begin{equation} 
\label{equation:2nd-derivative-of-U-paper3}
\begin{aligned}
&\left| 
D^{2}V(t,x) 
(v_{1}, v_{2})
\right|  \\
&\leq 
\dfrac{2}{\sqrt{2}t} 
\Big\|
V\left(
t/2, X^{x}_{t/2 } 
\right) 
\Big\|_{L^{2}(\Omega,\mathbb{R})}
\Big\|
\int_{0}^{\frac{t}{2} } 
\left\langle 
\xi^{
v_{1}, v_{2}
}(s,x), 
\dd W_{s} 
\right\rangle 
\Big\|_{L^{2}(\Omega,\mathbb{R})} 
\\
& \quad + 
\dfrac{2}{\sqrt{2}t}  
\left\|
DV\left(
t/2, X^{x}_{t/2 } 
\right)
\eta^{
v_{2}
}(t/2,x) 
\right\|_{L^{2}(\Omega,\mathbb{R})}
\Big\|
\int_{0}^{\frac{t}{2} } 
\left\langle 
\eta^{v_{1}}(s,x), 
\dd W_{s} 
\right\rangle 
\Big\|_{L^{2}(\Omega,\mathbb{R})}  \\
&
\leq
C
\dfrac{
\| 
\mathcal{K}(\Phi, X^{x}_{t/2 }) 
\|_{L^{2}(\Omega, \mathbb{R})}
}{t}
\left(
\int_{0}^{\frac{t}{2} } 
\left\| 
\xi^{
v_{1}, v_{2}
}(s,x)
\right\|_{L^{2}(\Omega,\mathbb{R}^{d})}^{2}
\dd s
\right)^{1/2} \\
& \quad +
C
\dfrac{
\mathbb{E}
\bigg[
\Big(
\mathbb{E}
\Big[
\big| 
\mathcal{K}\big(
\Phi,X^{y}_{t/2 }
\big) 
\big|^2
\Big] 
\Big|_{y = X^{x}_{t/2 }}
\Big)^{1/2}
\bigg]
}{t^{2}}
\big(
 e^{Lt}-1 
\big)^{1/2}
\left\| 
\eta^{v_{2}}(t/2,x)
\right\|
\left(
\int_{0}^{\frac{t}{2} } 
\left\| 
\eta^{v_{1}}(s,x)
\right\|^{2}
\dd s
\right)^{1/2}
\\
& \leq
C\dfrac{
\| 
\mathcal{K}(\Phi, X^{x}_{t/2 }) 
\|_{L^{2}(\Omega, \mathbb{R})}
}{t}  
\big(
 e^{(3L+1/2)t}-1 
\big)^{1/2} 
\left[
\textbf{1}_{\gamma \in [1,2]}
+
\textbf{1}_{\gamma \in (2, \infty)} 
(d^{\gamma/2 - 1}+\|x\|^{\gamma-2})
\right] 
\|v_{1}\| \cdot \|v_{2}\|
\\
& \quad +
C\dfrac{
\mathbb{E}
\bigg[
\Big(
\mathbb{E}
\Big[
\big| 
\mathcal{K}\big(
\Phi,X^{y}_{t/2 }
\big) 
\big|^2
\Big] 
\Big|_{y = X^{x}_{t/2 }}
\Big)^{1/2}
\bigg]
}{t^{2}} 
\big(
 e^{Lt}-1 
\big)
e^{Lt/2}
\|v_{1}\| \cdot \|v_{2}\|. \\
\end{aligned}
\end{equation}
}
The proof will be separated into two cases in accordance with the range of time $t>0$.

\noindent
\textbf{Case I: $t \in (0,1]$}

Here we choose the function $\Phi = \phi$ with $\mathcal{K}(\Phi,x) = \| \phi\|_{0}$.
Then 
it follows immediately from \eqref{equation:1st-derivative-of-U-paper3}, \eqref{equation:2nd-derivative-of-U-paper3} and the H\"older inequality that
\begin{equation}\label{equation:Du-in-proof-for-t<1-paper3}
\begin{aligned}
\left|
Du(t,x) v_{1} 
\right|
\leq 
{\color{black}
C\dfrac{\|\phi\|_{0}}{\sqrt{t}}
}
\|v_{1} \|, 
\quad \forall x, v_{1} \in \mathbb{R}^{d}, 
\end{aligned}
\end{equation}
and
{\color{black}
\begin{equation}
\begin{aligned}
\left|
D^{2}u(t,x)(v_{1}, v_{2}) 
\right|
\leq
C
\dfrac{\|\phi\|_{0}}{t} 
\left[
\textbf{1}_{\gamma \in [1,2]}
+
\textbf{1}_{\gamma \in (2, \infty)} 
(
d^{\gamma/2-1}+\|x\|^{\gamma-2}
)
\right] 
\|v_{1}\| \cdot \|v_{2}\|,
\quad
x, v_{1}, v_{2} \in \mathbb{R}^{d},
\end{aligned}
\end{equation}
}
{\color{black}
for some constant $C$ depending on $L$,
where we have also used
the fact that $e^{ct}-1 \leq 2ct$, $\forall c > 0$, and $1/\sqrt{t} \leq 1/t$, $t \in (0,1]$.
}

\noindent
\textbf{Case II: $t \in (1,\infty)$}

Recalling \eqref{equation:def-of-u-paper3}, when $t>1$, the Markov property immediately implies that 
\begin{equation}
u(t,x)=\mathbb{E}[u(t-1, X^{x}_{1})],
\end{equation}
and from \eqref{equation: expoenetially decrease of SDE to the invariant measure} we arrive at
\begin{equation}
\left|
u(t-1,x)
-
\int_{
\mathbb{R}^d
} 
\phi(x) \pi(\mathrm{d} x)
\right| 
\leq
{\color{black}
C_{\star}
\|\phi\|_{0}
e^{-c_{\star} (t-1)}
\left(
1+
\|x\|
\right),
\quad
\forall
x \in \mathbb{R}^{d}.
}
\end{equation}
{\color{black}
Inspired by \cite{brehier2014approximation},
}
here we choose
\begin{equation}
{\color{black}
\Phi_{t}(x)
}:=
u(t-1,x)- 
\int_{\mathbb{R}^d} \phi(x) \pi(\mathrm{d}  x)
\end{equation}
with 
{\color{black}
\begin{equation}\label{eq:constant-for-t>1-paper3}
\mathcal{K}(\Phi_{t},x) =
C_{\star} 
\|\phi\|_{0}
e^{-c_{\star}(t-1)} 
(1+\|x\|).
\end{equation}
}
Then, by the Markov property again, one has
\begin{equation}
\begin{aligned}
{\color{black}
\mathbb{E}
\left[
\Phi_{t}(X^{x}_{1}) 
\right]
}
&=
\mathbb{E}
\left[
u(t-1,X^{x}_{1}) 
\right] 
- 
\int_{\mathbb{R}^d} \phi(x) 
\pi(\mathrm{d}  x) \\
& = 
u(t,x)
-
\int_{
\mathbb{R}^d
} 
\phi(x) 
\pi(\mathrm{d}  x) \color{black},
\end{aligned}
\end{equation}
leading to 
\begin{equation}
u(t,x)= 
{\color{black}
\mathbb{E}
[
\Phi_{t}(X^{x}_{1})
]
}
+ 
\int_{\mathbb{R}^d} 
\phi(x) \pi(\dd x)   
=
{\color{black}
V(1,x) 
}
+
\int_{\mathbb{R}^d} 
\phi(x) \pi(\dd x).
\end{equation}
{\color{black}
Equipped with the estimates \eqref{equation:1st-derivative-of-U-paper3}, \eqref{equation:2nd-derivative-of-U-paper3} at $t=1$, \eqref{eq:constant-for-t>1-paper3}  and Lemma \ref{lemma:Uniform-moment-bounds-of-the-Langevin-SDE-paper3},
}
{\color{black}
we 
attain
}
{\color{black}
\begin{equation}  \label{equation:D2u-in-proof-for-t>1-paper3}
\begin{aligned}
\big| 
Du(t,x) v_{1} 
\big|
\leq 
C_{\star}
\|\phi\|_{0}
e^{-c_{\star}(t-1)} 
\left(
d^{1/2} + \|x\|
\right)
\|v_{1}\|, 
\quad \forall x, v_{1} \in \mathbb{R}^{d},
\end{aligned}
\end{equation}
}
and
{\color{black}
\begin{equation}\label{equation:D2u-in-proof-for-t>1-paper3}
\begin{aligned}
&\left|
D^{2}u(t,x)(v_{1}, v_{2}) 
\right|
\leq 
C_{\star}
\|\phi\|_{0}
e^{-c_{\star}(t-1)} 
(d^{1/2} + \|x\|)
\left[
\textbf{1}_{\gamma \in [1,2]}
+
\textbf{1}_{\gamma \in (2, \infty)} 
(d^{\gamma/2-1}+\|x\|^{\gamma-2})
\right] 
\|v_{1}\| \cdot \|v_{2}\|, \\
& \hspace{32em}
\forall x, v_{1}, v_{2} \in \mathbb{R}^{d}.  \\
\end{aligned}
\end{equation}
}
%
Thus, the proof is completed.


\end{proof}


\section{Proof of Theorem \ref{theorem:main-result-paper3}: time-independent weak error analysis}
\label{section:Time-independent-error-analysis}
Thanks to \eqref{equation:convergence-sode-paper3} and by setting $T=Nh$,
the total variation distance between the law of the PLMC algorithm and the target distribution  induced by \eqref{eq:langevin-SODE} boils down to the weak error analysis of the PLMC scheme \eqref{equation:numerical-scheme-PLMC-algorithm-paper3} for the Langevin SDEs \eqref{eq:langevin-SODE}:
\begin{equation}\label{equation:error-decomposition-tv-paper3}
\begin{aligned}
\left\|
\Pi(Y^{x_{0}}_{N}) -   \pi
\right\|_{\text{TV}} 
&\leq \big\|
\Pi(Y^{x_{0}}_{N}) -  \Pi(X^{x_{0}}_{T})  
\big\|_{\text{TV}}
+ \left\|
\Pi(X^{x_{0}}_{T}) -  \pi
\right\|_{\text{TV}}  \\
& =  \sup_{
\phi \in C_{b}(\mathbb{R}^{d}), 
\ \|\phi\|_{0} \leq 1 
} 
\big|
\mathbb{E}\left[ 
\phi(X^{x_{0}}_{T}) 
\right]  
- \mathbb{E}\left[
\phi(Y^{x_{0}}_{N}) 
\right] 
\big| 
+ \left\|
\Pi(X^{x_{0}}_{T}) -   \pi
\right\|_{\text{TV}}  \\ 
&\leq 
 \sup_{
\phi \in C_{b}(\mathbb{R}^{d}),
\ \|\phi\|_{0} \leq 1 
} 
\big|
\mathbb{E}\left[
\phi(X^{x_{0}}_{T}) 
\right]  
- 
\mathbb{E}\left[
\phi(Y^{x_{0}}_{N}) 
\right] 
\big|
+
{\color{black}
C_{\star} \|\phi\|_{0}
e^{- c_{\star} Nh}
\left(
1+
\mathbb{E}\left[
\|x_{0}\|
\right]
\right).
}
\end{aligned}
\end{equation}
{\color{black}
To carry out the weak error analysis 
based on the Kolmogorov equations, 
it is convenient to introduce continuous-time extensions} of the PLMC algorithm \eqref{equation:numerical-scheme-PLMC-algorithm-paper3} as follows:
for $s\in [t_{n}, t_{n+1}]$, $n \in \{0,1,\dots,N-1 \}$, 
\begin{equation} \label{introduction:continuous-time-version-of-PLMC algorithm-paper3}
\left\{
\begin{array}{l}
\mathbb{Y}^{n}(s)=\mathbb{Y}^{n}(t_{n}) 
+  f\left(\mathbb{Y}^{n}(t_{n})\right) (s-t_{n}) + \sqrt{2} (W_{s}-W_{t_{n}}) \\
\mathbb{Y}^{n}(t_{n}) = \mathscr{P}(Y_{n}).
\end{array}\right.
\end{equation}
{\color{black}
Before proceeding further,
it is important to mention that, in the case $\gamma > 1$ the continuous-time extensions $\mathbb{Y}^{n}(s)$, $s\in [t_{n}, t_{n+1}]$, $n \in \{0,1,\dots,N-1 \}$ are not continuous on the grid points, due to the projection introduced in the PLMC algorithm \eqref{equation:numerical-scheme-PLMC-algorithm-paper3}.
For the Lipschitz case ($\gamma =1$), the continuous-time extensions are continuous on the grid points and thus continuous within the whole interval $[0, T]$.
}
Next we present some 
{\color{black}
regularity estimates 
}
of the process $\{\mathbb{Y}^{n}(s) \}_{s\in [t_{n}, t_{n+1}]}$.
\begin{lemma}\label{lemma:holder-continuity-of-PLMC algorithm-paper3}
Let Assumptions \ref{assumption:globally-polynomial-growth-condition-paper3},  \ref{assumption:contractivity-at-infinity-condition-paper3} and  \ref{assumption:coercivity-condition-of-the-drift-paper3} hold. 
Let $\{\mathbb{Y}^{n}(s) \}_{s\in [t_{n}, t_{n+1}]}$ be defined as \eqref{introduction:continuous-time-version-of-PLMC algorithm-paper3}, $n \in \{0,1,\dots,N-1 \}$, $N\in \mathbb{N}$. Then, for any $ p\in [1,\infty)$,
the following estimates hold true,
\begin{equation} \label{equation: moment bounds of the continuous time version of PLMC algorithm}
\begin{aligned} 
\mathbb{E} 
\left[
\big\|
\mathbb{Y}^{n}(s)
\big\|^{2p}
\right] 
\leq 
{\color{black}
C 
\left(
d^{p}
+ \mathbb{E} 
\left[
\|X_{0}\|^{2p}
\right] 
\right).
}
\end{aligned} 
\end{equation}
and
\begin{equation} 
\begin{aligned}
\mathbb{E} 
\left[
\big\| 
\mathbb{Y}^{n}(s) 
-
\mathbb{Y}^{n}(t_{n}) 
\big\|^{2p}
\right] 
&\leq 
{\color{black}
Ch^{p} 
\left(
d^{p\gamma} 
+
\mathbb{E} 
\left[
\|X_0 \|^{2p\gamma} 
\right]
\right),
}\\
\mathbb{E} 
\left[
\left\| 
f\big(\mathbb{Y}^{n}(s) \big) 
-f\big(\mathbb{Y}^{n}(t_{n}) \big)
\right\|^{2p}
\right] 
&\leq 
{\color{black}
C
h^{p} 
\left(
d^{2p\gamma - p}
+ \mathbb{E} 
\left[
\|X_{0}\|^{4p\gamma -2p}
\right] 
\right).
}
\end{aligned}
\end{equation}
\end{lemma}
{\color{black}
The proof of Lemma \ref{lemma:holder-continuity-of-PLMC algorithm-paper3} is straightforward due to Assumption \ref{assumption:globally-polynomial-growth-condition-paper3}, Lemma \ref{lemma:Uniform-moment-bounds-of-the-PLMC-paper3} and the H\"older inequality. Similar estimates can be found in \cite[Lemma 5.6]{pang2024linear} without the dependence with respect to the dimension $d$.
As a result, we omit the proof here.
}
Moreover, we would like to present the error estimate between 
{\color{black}
$x \in \R^d$
and its projection 
$\mathscr{P}(x)$, 
}
defined by {\color{black}\eqref{equation:projection-operator-paper3}}.
\begin{lemma} \label{lemma:error-estimate-between-x-and-projected-x-paper3}
{\color{black}
Let $\gamma$ be given in Assumption \ref{assumption:globally-polynomial-growth-condition-paper3} and
}
{\color{black}
let $\mathscr{P}(\cdot)$ be defined by \eqref{equation:projection-operator-paper3}. 
}
Then we have
{\color{black}
\begin{equation}
\|x-\mathscr{P}(x)\| \leq 2 \vartheta^{-4 \gamma} d^{-2} h^2\|x\|^{4 \gamma+1}
\end{equation}
}
{\color{black}
for $\gamma > 1$
}
and 
{\color{black}
$x-\mathscr{P}(x) = 0, \forall x \in \mathbb{R}^{d}$
for $\gamma=1$.
}
\end{lemma}
{\color{black}
We point out that 
Lemma \ref{lemma:error-estimate-between-x-and-projected-x-paper3} can be proved similarly to \cite[Lemma 5.7]{pang2024linear}.
}
For completeness,
the proof of Lemma \ref{lemma:error-estimate-between-x-and-projected-x-paper3} is also shown in Appendix \ref{Proof of Lemma projected error}. 
Up to this stage, we have established sufficient machinery to obtain the uniform weak error estimate of the Langevin SDE \eqref{eq:langevin-SODE} and the PLMC scheme \eqref{equation:numerical-scheme-PLMC-algorithm-paper3} as below.
\begin{theorem} 
\label{theorem: Time-independent weak error analysis-paper3}
Let Assumptions \ref{assumption:globally-polynomial-growth-condition-paper3}, \ref{assumption:contractivity-at-infinity-condition-paper3} and \ref{assumption:coercivity-condition-of-the-drift-paper3} hold. Let $\{X^{x_{0}}_{t}\}_{t\geq 0}$ and $\{Y^{x_{0}}_{n}\}_{n \geq 0}$ be the solutions of SDE \eqref{eq:langevin-SODE} and the PLMC algorithm \eqref{equation:numerical-scheme-PLMC-algorithm-paper3} with the same initial state $X^{x_{0}}_{0} = Y^{x_{0}}_{0} = x_{0}$, respectively. Also, let 
{\color{black}
$h\in (0, \min\{1/2a_{1}, 2a_{1}/(a_{1}+2C^{2}_{f}), 1\} )$}, 
{\color{black}
where $a_{1}$ and $C_f$ are given in Assumption \ref{assumption:coercivity-condition-of-the-drift-paper3} and Lemma \ref{lemma: useful estimate of PLMC algorithm-paper3}, respectively,
}
be the uniform timestep. Then for 
{\color{black} any
test function
}
$\phi \in C_{b}(\mathbb{R}^{d})$ and any terminal time $T$ such that $T=Nh$,
\begin{equation}\label{equation:weak-error-for-super-linear-drift-paper3}
\big| 
\mathbb{E}\left[
\phi(Y^{x_{0}}_{N}) 
\right] 
- \mathbb{E}\left[
\phi(X^{x_{0}}_{T}) 
\right]
\big| 
\leq 
Cd^{\max\{3\gamma/2 , 2\gamma-1 \} } h \left| \ln{h} \right|.
\end{equation}
\end{theorem}
{\color{black}
In particular, when $\gamma=1$, i.e. when
the drift $f$ is globally Lipschitz continuous, 
}
the PLMC algorithm \eqref{equation:numerical-scheme-PLMC-algorithm-paper3} reduces to the classical LMC  \eqref{equation:numerical-scheme-euler-maruyama-paper3} 
and one obtains
\begin{equation} \label{equation:weak-error-for-Lipshitz-drift-paper3-not-optimal-paper3}
\Big| 
\mathbb{E}
\big[
\phi\big(
\widetilde{Y}^{x_{0}}_{N}
\big) 
\big] 
- 
\mathbb{E}\left[
\phi\big(
X^{x_{0}}_{T}
\big) 
\right]
\Big| 
\leq 
Cd^{3/2} h \left| \ln{h} \right|.
\end{equation}
{
\color{black}
As indicated by \eqref{equation:weak-error-for-Lipshitz-drift-paper3-not-optimal-paper3}, we derive $d^{3/2}$ dimension dependence for the LMC in the case of the Lipschitz continuous drift $f$ under the contractivity at infinity condition, which is consistent with that for the LMC algorithm using decreasing stepsizes in \cite{li2023unadjusted}.
For the case of superlinearly growing drift $f$,
the dimension dependence of the PLMC algorithm  \eqref{equation:numerical-scheme-PLMC-algorithm-paper3}
is proved to be
$d^{\max\{3\gamma/2 , 2\gamma-1 \} }$, which is new in the literature.
More precisely,
the dimension dependence is $d^{3\gamma/2}$ for the case $1\leq \gamma \leq 2$
and $d^{2\gamma -1}$  for $\gamma > 2$. Whether the obtained dimension dependence can be further improved 
turns out to be an interesting topic \cite{mou2022improved} and would be a direction for our future work.}
\begin{proof} [Proof of Theorem \ref{theorem: Time-independent weak error analysis-paper3}]
By the telescoping argument, \eqref{introduction:continuous-time-version-of-PLMC algorithm-paper3} and \eqref{equation:def-of-u-paper3}, the weak error can be decomposed as follows:
\begin{equation} \label{eq:decomposition-of-the-weak-error-paper3}
\begin{aligned}
\big| 
\mathbb{E}\left[
\phi(Y^{x_{0}}_{N}) 
\right] 
- 
\mathbb{E}\left[
\phi(X^{x_{0}}_{T}) 
\right]
\big|
&=
\big|
u(0,Y^{x_{0}}_{N}) - u(T,x_{0})
\big|\\
&=  
\left|
\sum_{n=0}^{N-1}  
\mathbb{E}\big[
u(T-t_{n+1}, Y_{n+1}) 
\big] 
- 
\mathbb{E}\big[
u(T-t_{n}, Y_{n}) 
\big] 
\right| \\
&{\color{black}
\leq
\left|
\sum_{n=0}^{N-1}  
\mathbb{E}\big[
u \big(
T-t_{n}, 
\mathscr{P}\left(Y_n\right)
\big) 
\big] 
- 
\mathbb{E}\big[
u\big(
T-t_{n}, {Y}_{n}
\big) 
\big]
\right| 
}  \\
&\quad
{\color{black}
+
\left| 
\sum_{n=0}^{N-1} 
\mathbb{E}\big[
u \big(
T-t_{n+1}, 
Y_{n+1}
\big) 
\big] 
- \mathbb{E}\big[
u\big(
T-t_{n}, 
\mathscr{P}\left(Y_n\right)
\big) 
\big] 
\right| 
}\\
&{\color{black}
=
}
\underbrace{
\left|
\sum_{n=0}^{N-1}  
\mathbb{E}\big[
u \big(
T-t_{n}, 
\mathbb{Y}^{n}(t_{n}) 
\big) 
\big] 
- 
\mathbb{E}\big[
u\big(
T-t_{n}, {Y}_{n}
\big) 
\big]
\right|
}_{\color{black}
=: J_{1}}\\
&\quad + 
\underbrace{
\left| 
\sum_{n=0}^{N-1} 
\mathbb{E}\big[
u \big(
T-t_{n+1}, 
\mathbb{Y}^{n}(t_{n+1})
\big) 
\big] 
- \mathbb{E}\big[
u\big(
T-t_{n}, 
\mathbb{Y}^{n}(t_{n}) 
\big) 
\big] 
\right|
}_{\color{black}=: J_{2}},\\
\end{aligned}
\end{equation}
where the fact was used that $\mathbb{Y}^{n}(t_{n+1})=Y_{n+1}$
{\color{black}
and $\mathbb{Y}^n\left(t_n\right)=\mathscr{P}\left(Y_n\right)$
}
due to \eqref{introduction:continuous-time-version-of-PLMC algorithm-paper3}. 
In the following, 
{\color{black}
we set $N_{1} \in \mathbb{N}$ such that 
$N_1 h\leq 1 <(N_1+1)h$.
}
First, we deal with the case $\gamma >1$.

\noindent
\textbf{Step I: $\gamma >1$}

The estimate of $J_{1}$ can be derived by Lemma \ref{lemma:Uniform-moment-bounds-of-the-PLMC-paper3}, 
{\color{black}
\eqref{equation:Du-in-theorem-paper3}, \eqref{equation:Du-for-t>1-paper3} in 
}
Theorem \ref{theorem:regular-estimate-of-u-and-its-derivatives-paper3}, Lemma \ref{lemma:error-estimate-between-x-and-projected-x-paper3} and the Taylor expansion 
as follows, 
setting $\upsilon_{1}(\bar{r}):=Y_{n} + \bar{r}\big( \mathbb{Y}^{n}(t_{n})
- Y_{n}\big)$, $\bar{r} \in [0,1]$,
\begin{equation} \label{equation:estimate-of-J1-paper3}
\begin{aligned}
J_{1} 
&\leq 
\sum_{n=0}^{N-N_{1}-1} 
\left|  
\mathbb{E}
\left[
\int_{0}^{1}
Du \big(
T-t_{n}, 
\upsilon_{1}(\bar{r}) 
\big) 
\big( \mathbb{Y}^{n}(t_{n})
- Y_{n}
\big) \dd \bar{r}
\right] 
\right|
 \\
& \quad + 
\sum_{n=N-N_{1}}^{N-1} 
\left|  
\mathbb{E}
\left[
\int_{0}^{1}
Du \big(T-t_{n}, 
\upsilon_{1}(\bar{r})
 \big) 
 \big(
 \mathbb{Y}^{n}(t_{n})
- Y_{n}
\big) 
\dd \bar{r}
\right] 
\right| \\
& \leq 
{\color{black}
C_{\star} \|\phi\|_{0}
\sum_{n=0}^{N-N_{1}-1}
e^{-c_{\star}(T-t_{n}-1)} 
\left\|
\left(
d^{1/2} + \| \upsilon_{1}(\bar{r})\|
\right)
\left\|
\mathbb{Y}^{n}(t_{n}) - {Y}_{n} 
\right\|
\right\|_{L^{1}(\Omega,\mathbb{R})}
}
\\
&\quad +
C 
\sum_{n=N-N_{1}}^{N-1} 
\left(
\dfrac{1}{\sqrt{T-t_{n}}} 
\right) 
\left\|
\mathbb{Y}^{n}(t_{n}) - {Y}_{n} 
\right\|_{L^{1}(\Omega,\mathbb{R}^{d})} \\
& 
{\color{black}
\leq
2 \vartheta^{-4 \gamma} 
C_{\star} \|\phi\|_{0}
d^{-2} h^2
\sum_{n=0}^{N-N_{1}-1}
e^{-c_{\star}(T-t_{n}-1)} 
\left\|
\left(
d^{1/2} + \| \upsilon_{1}(\bar{r})\|
\right)
\left\|
 {Y}_{n} 
\right\|^{4\gamma + 1}
\right\|_{L^{1}(\Omega,\mathbb{R})}
}
\\
& \quad
{
\color{black}
+
2 \vartheta^{-4 \gamma}  C
d^{-2} h^2
\sum_{n=N-N_{1}}^{N-1}
\left(
\dfrac{1}{\sqrt{T-t_{n}}} 
\right) 
\left\|
 {Y}_{n} 
\right\|_{L^{4\gamma+1}(\Omega,\mathbb{R}^{d})}
}
\\
& \leq 
C d^{-2}
{\color{black}
\left( 
d^{1/2}
\left\|
x_{0} 
\right\|^{4\gamma+1}_{L^{4\gamma+1 }(\Omega, \mathbb{R}^{d})} 
+
\left\|
x_{0} 
\right\|^{4\gamma+2}_{L^{4\gamma+2 }(\Omega, \mathbb{R}^{d})}
+ d^{2\gamma+1}
\right)h
}\\
& \leq 
Cd^{2\gamma - 1} h,
\end{aligned}
\end{equation}
{\color{black}
where
it is 
straightforward}
to obtain that 
\begin{equation}
\label{equation:uniform-estimate-for-sum-paper3}
 \sum_{n=0}^{N-N_{1}-1}
 {\color{black}
 e^{
 -c_{\star}(T-t_{n}-1)
 }
 }h  
 \quad
 \text{and} \quad
 \sum_{n=N-N_{1}}^{N-1} 
 \left(
 \dfrac{1}{ \sqrt{T-t_{n}}} 
 \right)h
\end{equation}
 are uniformly bounded with respect to $T$.

{\color{black}
Regarding $J_2$, 
we first use
the It\^o formula and the Kolmogorov equation \eqref{equation:kolmogorov-equation-paper3} to obtain
\begin{equation}
\label{equation:Ito-formula-and-kolmogorov-equation-paper3}
\begin{aligned}
&u \big(
T-t_{n+1}, 
\mathbb{Y}^{n}(t_{n+1})
\big)  
- 
u\big(
T-t_{n}, 
\mathbb{Y}^{n}(t_{n}) 
\big)  \\
&=
-
\int^{t_{n+1}}_{t_{n}}
\partial_{s}
u\big(
T-s, \mathbb{Y}^{n}(s)
\big)
\dd s
+
\int^{t_{n+1}}_{t_{n}}
D
u\big(
T-s, \mathbb{Y}^{n}(s)
\big)
f\big(
\mathbb{Y}^{n}(t_{n}) 
\big) \dd s \\
&\quad
+
\sqrt{2}
\int^{t_{n+1}}_{t_{n}}
D
u\big(
T-s, \mathbb{Y}^{n}(s)
\big)
\dd W_s 
+
\sum_{j=1}^{d}
\int^{t_{n+1}}_{t_{n}}
D^{2}
u\big(
T-s, \mathbb{Y}^{n}(s)
\big)
\big(e_{j} , e_{j} \big)
\dd s \\
&=
\int^{t_{n+1}}_{t_{n}}
D
u\big(
T-s, \mathbb{Y}^{n}(s)
\big)
\left(
f\big(
\mathbb{Y}^{n}(t_{n}) 
\big) 
-
f\big(
\mathbb{Y}^{n}(s) 
\big) 
\right)
\dd s
+
\sqrt{2}
\int^{t_{n+1}}_{t_{n}}
D
u\big(
T-s, \mathbb{Y}^{n}(s)
\big)
\dd W_s.
\end{aligned}
\end{equation}
}
{\color{black}
It follows from \eqref{equation:Ito-formula-and-kolmogorov-equation-paper3} and 
{\color{black}
the conditional expectation argument
}
that
}
\begin{small}
\begin{equation}
\label{equation:decomposition-J2-paper3}
\begin{aligned}
J_{2} 
&=
\left|
\sum_{n=0}^{N-1}  
\mathbb{E}\left[
\int_{t_{n}}^{t_{n+1}} 
Du\big(
T-s, 
\mathbb{Y}^{n}(s)
\big)
\Big(
f\big(
\mathbb{Y}^{n}(t_{n}) 
\big)
- 
f\big(
\mathbb{Y}^{n}(s) 
\big)
\Big)  \mathrm{d} s 
\right] 
\right| \\
& {
\color{black} \leq
}
\underbrace{
\left| 
\sum_{n=0}^{N-1}  
\mathbb{E}\left[
\int_{t_{n}}^{t_{n+1}} 
Du\big(
T-s, 
\mathbb{Y}^{n}(t_{n})
\big)
\Big(
f\big(
\mathbb{Y}^{n}(t_{n}) 
\big) 
-  
f\big(
\mathbb{Y}^{n}(s) 
\big)
\Big) \mathrm{d} s
\right] 
\right| 
}_{=:J_{2,1}}\\
&\quad + \color{black}
\underbrace{
\left|
\sum_{n=0}^{N-1}  
\mathbb{E}\left[
\int_{t_{n}}^{t_{n+1}} 
\Big(
Du\big(
T-s, \mathbb{Y}^{n}(s)
\big) 
- Du\big(
T-s, \mathbb{Y}^{n}(t_{n})
\big) 
\Big)
\Big(
f\big(\mathbb{Y}^{n}(t_{n}) \big)
-  f\big(\mathbb{Y}^{n}(s) \big)
\Big) 
\mathrm{d} s
\right] 
\right|
}_{=:J_{2,2}}.
\\
\end{aligned}
\end{equation}
\end{small}
Recalling the Taylor expansion, we note that,
\begin{equation}\label{equation:taylor-expansion-of-f-paper3}
\begin{aligned}
f\big(
\mathbb{Y}^{n}(s) 
\big) 
&= 
f\big(
\mathbb{Y}^{n}(t_{n}) 
\big) 
+ Df\big(
\mathbb{Y}^{n}(t_{n}) 
\big) 
\big(
\mathbb{Y}^{n}(s) 
- \mathbb{Y}^{n}(t_{n}) 
\big)
+ \mathcal{R}_{f}\big(
\mathbb{Y}^{n}(s), \mathbb{Y}^{n}(t_{n}) 
\big)\\
& = 
f\big(
\mathbb{Y}^{n}(t_{n}) 
\big) 
+ Df\big(
\mathbb{Y}^{n}(t_{n}) 
\big) 
\left[
f\left(
\mathbb{Y}^{n}(t_{n})
\right) (s-t_{n}) 
+ \sqrt{2} (W_{s}-W_{t_{n}}) 
\right] \\
& \quad + 
\mathcal{R}_{f}\big(
\mathbb{Y}^{n}(s), 
\mathbb{Y}^{n}(t_{n}) 
\big),
\end{aligned}
\end{equation}
where
\begin{equation}
\begin{aligned}
&\mathcal{R}_{f} 
\big(
\mathbb{Y}^{n}(s), 
\mathbb{Y}^{n}(t_{n})
\big) \\
&:= \int_{0}^{1}  
\Big(
Df\left(
\mathbb{Y}^{n}(t_{n})
+ r \big(
\mathbb{Y}^{n}(s) 
- \mathbb{Y}^{n}(t_{n}) 
\big)
\right) 
-Df\big(
\mathbb{Y}^{n}(t_{n}) 
\big)
\Big) 
\big(
\mathbb{Y}^{n}(s) - \mathbb{Y}^{n}(t_{n})
\big) 
\ \dd r .
\end{aligned}
\end{equation}
Moreover, in light of Assumption \ref{assumption:globally-polynomial-growth-condition-paper3}, Lemma \ref{lemma:Uniform-moment-bounds-of-the-PLMC-paper3} and Lemma \ref{lemma:holder-continuity-of-PLMC algorithm-paper3}, one can further 
{\color{black}
apply}
the H\"older inequality to imply
\begin{small}
\begin{equation} \label{equation:estimate-of-the-remaining-term-R-paper3}
\begin{aligned}
&\left\|
\mathcal{R}_{f} 
\big(
\mathbb{Y}^{n}(s), 
\mathbb{Y}^{n}(t_{n}) 
\big)  
\right\|_{L^{2}(\Omega, \mathbb{R}^{d})} \\
& \leq  
C \int_0^1 
{\color{black}
\left\|
\big(
1+\left\|
r \mathbb{Y}^{n}(s)+(1-r) \mathbb{Y}^{n}(t_{n})
\right\|
+
\left\|\mathbb{Y}^{n}(t_{n})\right\|
\big)^{\max\{0, \gamma-2\}} 
\left\|
\mathbb{Y}^{n}(s) - \mathbb{Y}^{n}(t_{n})\right\|^2 
\right\|_{L^2(\Omega, \mathbb{R})} 
}
\mathrm{d} r \\
&
\leq
{\color{black}
C
\Bigg(
\textbf{1}_{\gamma \in [1,2]}
\left\|
\mathbb{Y}^{n}(s) - \mathbb{Y}^{n}(t_{n}) 
\right\|^{2}_{L^4(\Omega, \mathbb{R})}
}\\
&
\quad \quad
{
\color{black}
+
\textbf{1}_{\gamma \in (2, \infty)}
\int_0^1 
\left\|
\big(
1+\left\|
r \mathbb{Y}^{n}(s)+(1-r) \mathbb{Y}^{n}(t_{n})
\right\|
+
\left\|\mathbb{Y}^{n}(t_{n})\right\|
\big)^{\gamma-2} 
\left\|
\mathbb{Y}^{n}(s) - \mathbb{Y}^{n}(t_{n})\right\|^2 
\right\|_{L^2(\Omega, \mathbb{R})} 
\Bigg)
}\\
& \leq
{\color{black}
C 
\left[
\textbf{1}_{\gamma \in [1,2]}
\left(
d^{\gamma} 
+ 
\|x_{0}\|^{2\gamma}_{L^{4\gamma} (\Omega, \mathbb{R}^{d})} 
\right) 
+
\textbf{1}_{\gamma \in (2, \infty)}
\left(
d^{3\gamma/2 - 1} 
+ 
\|x_{0}\|^{3\gamma - 2}_{L^{6\gamma-4} (\Omega, \mathbb{R}^{d})} 
\right) 
\right] h
}
\\
& \leq Cd^{\max \{\gamma, 3\gamma/2-1\}}h.
\end{aligned}
\end{equation}
\end{small}
Regarding $J_{2,1}$, 
{\color{black}
using \eqref{equation:taylor-expansion-of-f-paper3}
and a conditional expectation argument gives
}
\begin{equation}
\begin{aligned}
J_{2,1}
= 
\bigg| 
-\sum_{n=0}^{N-1}  
&\mathbb{E}
\Big[
\int_{t_{n}}^{t_{n+1}} 
Du\big(T-s, \mathbb{Y}^{n}(t_{n})\big)
\Big(
Df\big(
\mathbb{Y}^{n}(t_{n}) 
\big) 
f\left(\mathbb{Y}^{n}(t_{n})\right) (s-t_{n})  \\
&\ +  
\mathcal{R}_{f}\big(\mathbb{Y}^{n}(s), \mathbb{Y}^{n}(t_{n}) \big)
\Big) \dd s 
\Big] 
\bigg|,
\end{aligned}
\end{equation}
where we derive from
{\color{black}
\eqref{equation:growth-of-derivative-of-f-paper3}, \eqref{equation:growth-of-f-paper3} and Lemma \ref{lemma:Uniform-moment-bounds-of-the-PLMC-paper3} 
}
that
\begin{equation}\label{equation:estimate-of-another-remaining-term-paper3}
\begin{aligned}
\left\|
Df\big(
\mathbb{Y}^{n}(t_{n}) 
\big) 
f\left(\mathbb{Y}^{n}(t_{n})\right) (s-t_{n})
\right\|_{L^{2}(\Omega,\mathbb{R}^{d})}
\leq
C
\left(
1+
\sup_{0\leq r \leq N} \|Y_{r} \|^{2\gamma-1}_{L^{ 4\gamma-2} (\Omega, \mathbb{R}^{d})} 
\right)h 
\leq Cd^{\gamma-1/2}h.
\end{aligned}
\end{equation}
{\color{black}
Using Lemmas \ref{lemma:Uniform-moment-bounds-of-the-Langevin-SDE-paper3}, \ref{lemma:Uniform-moment-bounds-of-the-PLMC-paper3}, 
{\color{black}
\eqref{equation:Du-in-theorem-paper3}, \eqref{equation:Du-for-t>1-paper3},
} 
\eqref{equation:estimate-of-the-remaining-term-R-paper3}, \eqref{equation:estimate-of-another-remaining-term-paper3} and the H\"older inequality yields
}
{\color{black}
\begin{equation} \label{equation: estimate of J2,1-paper3}
\begin{aligned}
J_{2,1} 
& \leq 
C_{\star} \|\phi \|_{0}
\sum_{n=0}^{N-N_{1}-1} 
\int_{t_{n}}^{t_{n+1}}
e^{-c_{\star}(T-s-1)}
\Big(
d^{1/2} 
+ 
\big\|
\mathbb{Y}^{n}(t_{n})
\big\|_{L^{2}(\Omega, \mathbb{R}^{d})}
\Big)
\\
& \quad \quad 
\times
\left(
\left\|
Df\big(
\mathbb{Y}^{n}(t_{n}) 
\big) 
f\left(\mathbb{Y}^{n}(t_{n})\right) (s-t_{n})
\right\|_{L^{2}(\Omega,\mathbb{R}^{d})}
+
\left\|
\mathcal{R}_{f} 
\big(
\mathbb{Y}^{n}(s), 
\mathbb{Y}^{n}(t_{n}) 
\big)  
\right\|_{L^{2}(\Omega, \mathbb{R}^{d})} 
\right) \dd s \\
&\quad
+
C
\sum_{n=N-N_{1}}^{N-1} 
\int_{t_{n}}^{t_{n+1}}
\dfrac{ \|\phi\|_{0}}{\sqrt{T-s}}
\Big(
\left\|
Df\big(
\mathbb{Y}^{n}(t_{n}) 
\big) 
f\left(\mathbb{Y}^{n}(t_{n})\right) (s-t_{n})
\right\|_{L^{2}(\Omega,\mathbb{R}^{d})} \\
& \hspace{21em}
+
\left\|
\mathcal{R}_{f} 
\big(
\mathbb{Y}^{n}(s), 
\mathbb{Y}^{n}(t_{n}) 
\big)  
\right\|_{L^{2}(\Omega, \mathbb{R}^{d})} 
\Big) \dd s \\
&
\leq Cd^{\max \{\gamma+1/2, (3\gamma-1)/2\}}h.
\end{aligned}
\end{equation}
}
However, the estimate of $J_{2,2}$ is rather technical. To overcome the possible singularities, we need to 
{\color{black}
split}
the time interval 
{\color{black}
$[0,T)$
}
into 
{\color{black}
$[0, T-N_1h-h)$, $[T-N_1h-h, T-N_1h)$, $[T-N_1h, T-h)$ and $[T-h, T)$.
}
Applying the Taylor expansion to $Du(t,\cdot)$ leads to, for  $\upsilon_{2}^{n}(\bar{r}) :=\mathbb{Y}^{n}(t_{n}) + \bar{r}
(\mathbb{Y}^{n}(s)-\mathbb{Y}^{n}(t_{n}))$, $\bar{r}\in [0,1]$, 
\begin{small}
\begin{equation}
\label{equation:decomposition-J2-paper3,2}
\begin{aligned}
&J_{2,2} \\
& \leq 
\underbrace{
{\color{black}
\sum_{n=0}^{N-N_{1}-2} 
}
\left|
\mathbb{E}
\left[
\int_{t_{n}}^{t_{n+1}} 
\int_{0}^{1}
D^{2}u\big(T-s, \upsilon_{2}^{n}(\bar{r})
\big)
\Big(
\mathbb{Y}^{n}(s)-\mathbb{Y}^{n}(t_{n}),
f\big(\mathbb{Y}^{n}(t_{n}) \big) 
-  f\big(\mathbb{Y}^{n}(s) \big)
\Big) 
\dd \bar{r}
\dd s
\right] 
\right| 
}_{=:J_{2,2,1}}\\
& \quad + 
\underbrace{
\sum_{n=N-N_{1}}^{N-2} 
\left| 
\mathbb{E}
\left[
\int_{t_{n}}^{t_{n+1}}  
\int_{0}^{1}
D^{2}u\big(T-s, \upsilon_{2}^{n}(\bar{r})
\big)
\Big(
\mathbb{Y}^{n}(s)-\mathbb{Y}^{n}(t_{n}),
f\big(\mathbb{Y}^{n}(t_{n}) \big) 
-  f\big(\mathbb{Y}^{n}(s) \big)
\Big) 
\dd \bar{r}
\dd s
\right] 
\right| 
}_{=: J_{2,2,2}}\\
& \quad 
+ 
\color{black}
\underbrace{
\left|  
\mathbb{E}
\left[
\int_{T-h}^{T} 
\Big(
Du\big(T-s, \mathbb{Y}^{N-1}(s)\big) 
- Du\big(T-s, \mathbb{Y}^{N-1}(t_{N-1}) \big)
\Big)
\Big(
f\big(\mathbb{Y}^{N-1}(t_{N-1}) \big) -  f\big(\mathbb{Y}^{N-1}(s) \big)
\Big) \dd s
\right] \right|
}_{=:J_{2,2,3}}\\
& \quad 
+ 
\bigg|  
\mathbb{E}
\bigg[
\int_{(N - N_{1} - 1)h}^{(N - N_{1} )h} 
\int_{0}^{1}
D^{2}u\big(T-s, \upsilon_{2}^{N-N_{1}-1}(\bar{r})
\big) \\
&\quad \quad
\underbrace{ \hspace{5em}
\Big(
\mathbb{Y}^{N-N_{1}-1}(s)-\mathbb{Y}^{N-N_{1}-1}(t_{N - N_{1} - 1}),
f\big(\mathbb{Y}^{N-N_{1}-1}(t_{N - N_{1} - 1}) \big) 
-  f\big(\mathbb{Y}^{N-N_{1}-1}(s) \big)
\Big) 
\dd \bar{r}
\dd s
\bigg] 
\bigg|.
}_{=:J_{2,2,4}}
\\ 
\end{aligned}
\end{equation}
\end{small}
By 
{\color{black}
\eqref{equation:D2u-for-t>1-paper3} in 
}
Theorem \ref{theorem:regular-estimate-of-u-and-its-derivatives-paper3}, we have
{\color{black}
\begin{small}
\begin{equation}
\begin{aligned}
&J_{2,2,1} \\
&\leq 
C_{\star}
\|\phi\|_{0}
\sum_{n=0}^{N-N_{1}-2} 
\mathbb{E}
\bigg[
\int_{t_{n}}^{t_{n+1}} 
\int^{1}_{0}
e^{-c_{\star}(T-s-1)} 
(d^{1/2} + \|\upsilon^{n}_{2}(\bar{r})\|)
\left[
\textbf{1}_{\gamma \in [1,2]}
+
\textbf{1}_{\gamma \in (2, \infty)} 
(d^{\gamma/2-1}+\|\upsilon_{2}^{n}(\bar{r})\|^{\gamma-2})
\right]  \\
& \hspace{21em}
\times
\left\| 
\mathbb{Y}^{n}(s)-\mathbb{Y}^{n}(t_{n}) 
\right\| \cdot
\left\| 
f\big(\mathbb{Y}^{n}(t_{n}) \big) -  f\big(\mathbb{Y}^{n}(s) \big)
\right\| \dd \bar{r} \dd s
\bigg].
\end{aligned}
\end{equation}
\end{small}
}
Following the same arguments as used in the estimate of $J_{2,2,1}$, one derives from 
{\color{black}
\eqref{equation:D2u-in-theorem-paper3} in 
}
Theorem \ref{theorem:regular-estimate-of-u-and-its-derivatives-paper3} that
{\color{black}
\begin{equation}
\begin{aligned}
J_{2,2,2} 
&\leq 
C_{\star}
\|\phi\|_{0}
\sum_{n=N-N_{1}}^{N-2} 
\mathbb{E}
\bigg[
\int_{t_{n}}^{t_{n+1}} 
\int^{1}_{0}
\frac{1}{T-s}
\left[
\textbf{1}_{\gamma \in [1,2]}
+
\textbf{1}_{\gamma \in (2, \infty)} (d^{\gamma/2-1}+\|\upsilon^{n}_{2}(\bar{r})\|^{\gamma-2})
\right]  \\
& \hspace{11em}
\times
\left\| 
\mathbb{Y}^{n}(s)-\mathbb{Y}^{n}(t_{n}) 
\right\| \cdot
\left\| 
f\big(\mathbb{Y}^{n}(t_{n}) \big) -  f\big(\mathbb{Y}^{n}(s) \big)
\right\| \dd \bar{r} \dd s
\bigg].
\end{aligned}
\end{equation}
}
In the following, we come to the estimate of $J_{2,2,1}$ and $J_{2,2,2}$, which would be divided into two cases depending on different ranges of $\gamma$.

\noindent
\textbf{Case I: $1\leq \gamma \leq 2$:}\\
{\color{black}
In the case $1\leq \gamma \leq 2$, 
we recall 
$\upsilon^{n}_{2}(\bar{r}) :=\mathbb{Y}^{n}(t_{n}) + \bar{r}
(\mathbb{Y}^{n}(s)-\mathbb{Y}^{n}(t_{n}))$, $\bar{r}\in [0,1]$,
and use Lemma \ref{lemma:Uniform-moment-bounds-of-the-Langevin-SDE-paper3},
Lemma \ref{lemma:Uniform-moment-bounds-of-the-PLMC-paper3},
Lemma \ref{lemma:holder-continuity-of-PLMC algorithm-paper3} and the H\"older inequality to obtain
}
{\color{black}
\begin{equation}\label{equation:estimate-of-J2,2-paper3,1-part-1}
\begin{aligned}
J_{2,2,1}  
&\leq 
C_{\star}
\|\phi\|_{0}
\sum_{n=0}^{N-N_{1}-2} 
\int_{t_{n}}^{t_{n+1}} 
e^{-c_{\star}(T-s-1)} 
\left(
d^{1/2} 
+ 
\sup_{0\leq r \leq N} 
\left\|
Y_{r} 
\right\|_{L^{3\gamma }(\Omega, \mathbb{R}^{d})} 
\right)
 \\
& \hspace{9em}
\times
\left\| 
\mathbb{Y}^{n}(s)-\mathbb{Y}^{n}(t_{n}) 
\right\|_{L^{3}(\Omega, \mathbb{R}^{d})}
\left\| 
f\big(\mathbb{Y}^{n}(t_{n}) \big) -  f\big(\mathbb{Y}^{n}(s) \big)
\right\|_{L^{3\gamma/ (2\gamma-1) }(\Omega, \mathbb{R}^{d})}
\dd s
 \\
& \leq
C 
\left(
d^{3\gamma/2}
+
 \|x_{0} \|^{3\gamma}_{L^{3\gamma }(\Omega, \mathbb{R}^{d})}
\right)
h.
\end{aligned}
\end{equation}
}
Concerning $J_{2,2,2}$, one deduces
{\color{black}
\begin{equation}
\begin{aligned}
J_{2,2,2}  
&\leq 
C_{\star}
\|\phi\|_{0}
\sum_{n=N-N_{1}}^{N-2} 
\int_{t_{n}}^{t_{n+1}} 
\dfrac{1}{T-s} 
\left\| 
\mathbb{Y}^{n}(s)-\mathbb{Y}^{n}(t_{n}) 
\right\|_{L^{(3\gamma-1)/\gamma}(\Omega, \mathbb{R}^{d})}
 \\
&\hspace{16em}
\times
\left\| 
f\big(\mathbb{Y}^{n}(t_{n}) \big) -  f\big(\mathbb{Y}^{n}(s) \big)
\right\|_{L^{(3\gamma-1)/ (2\gamma-1) }(\Omega, \mathbb{R}^{d})}
\dd s \\
& \leq
C \left(
d^{(3\gamma-1)/2} 
+ \|x_{0} \|^{3\gamma-1}_{L^{3\gamma-1 }(\Omega, \mathbb{R}^{d})}
\right)
h |\ln h|,
\end{aligned}
\end{equation}
}
where
\begin{equation}
\begin{aligned}
\sum_{n=N-N_{1}}^{N-2} 
\int_{t_{n}}^{t_{n+1}}
\frac{1}{T-s} \dd s \
{\color{black}
\leq
}
 \int_{h}^{1} \frac{1}{s} \dd s 
= \left| \ln{h} \right|.
\end{aligned}
\end{equation}
\textbf{Case II: $\gamma>2$:}\\
In the case $\gamma >2$, 
{\color{black}
using the 
the H\"older inequality, Lemma \ref{lemma:Uniform-moment-bounds-of-the-PLMC-paper3},
{\color{black}
\eqref{equation:D2u-for-t>1-paper3} in 
}
Theorem \ref{theorem:regular-estimate-of-u-and-its-derivatives-paper3} and 
Lemma \ref{lemma:holder-continuity-of-PLMC algorithm-paper3} yields
}
{\color{black}
\begin{small}
\begin{equation} \label{equation:estimate-of-J2,2-paper3,1-part-2}
\begin{aligned}
J_{2,2,1} 
&\leq 
C_{\star}
\|\phi\|_{0}
\sum_{n=0}^{N-N_{1}-2} 
\mathbb{E}
\bigg[
\int_{t_{n}}^{t_{n+1}} 
e^{-c_{\star}(T-s-1)} 
\sup_{0\leq r \leq N} 
\left(
d^{1/2}
+ 
\left\|
Y_{r} 
\right\|_{L^{4\gamma-2 }(\Omega, \mathbb{R}^{d})} 
\right)
\left(
d^{\gamma/2-1}
+ 
\left\|
Y_{r} 
\right\|^{\gamma-2}_{L^{4\gamma-2 }(\Omega, \mathbb{R}^{d})} 
\right)
 \\
& \hspace{6em}
\times
\left\| 
\mathbb{Y}^{n}(s)-\mathbb{Y}^{n}(t_{n}) 
\right\|_{L^{(4\gamma-2)/\gamma}(\Omega, \mathbb{R}^{d})}
\cdot
\left\| 
f\big(\mathbb{Y}^{n}(t_{n}) \big) -  f\big(\mathbb{Y}^{n}(s) \big)
\right\|_{L^{(4\gamma-2)/ (2\gamma-1) }(\Omega, \mathbb{R}^{d})}
\dd s
\bigg] \\
& \leq
C
\left(
d^{2\gamma-1}
+
\|x_{0}\|^{4\gamma-2}_{L^{4\gamma-2}(\Omega, \mathbb{R}^{d})}
\right)
h.
\end{aligned}
\end{equation}
\end{small}
}
For $J_{2,2,2}$,
{\color{black}
by the H\"older inequality, Lemma \ref{lemma:Uniform-moment-bounds-of-the-PLMC-paper3}, 
\eqref{equation:D2u-in-theorem-paper3} and Lemma \ref{lemma:holder-continuity-of-PLMC algorithm-paper3},
}
one 
{\color{black}
obtains
}
{\color{black}
\begin{equation}
\begin{aligned}
J_{2,2,2}  
&\leq 
C_{\star}
\|\phi\|_{0}
\sum_{n=N-N_{1}}^{N-2} 
\mathbb{E}
\bigg[
\int_{t_{n}}^{t_{n+1}} 
\dfrac{1}{T-s} 
\sup_{0\leq r \leq N}
\left(
d^{\gamma/2-1}
+ 
\left\|
Y_{r} 
\right\|^{\gamma-2}_{L^{4\gamma-3 }(\Omega, \mathbb{R}^{d})} 
\right) 
\\
&\hspace{6em}
\times
\left\| 
\mathbb{Y}^{n}(s)-\mathbb{Y}^{n}(t_{n}) 
\right\|_{L^{(4\gamma-3)/\gamma}(\Omega, \mathbb{R}^{d})}
\cdot
\left\| 
f\big(\mathbb{Y}^{n}(t_{n}) \big) -  f\big(\mathbb{Y}^{n}(s) \big)
\right\|_{L^{(4\gamma-3)/ (2\gamma-1) }(\Omega, \mathbb{R}^{d})}
\dd s
\bigg] \\
& \leq
C \left(
d^{2\gamma-3/2} 
+ \|x_{0} \|^{4\gamma-3}_{L^{4\gamma-3 }(\Omega, \mathbb{R}^{d})}
\right)
h |\ln h|.
\end{aligned}
\end{equation}
}
We are now in a position to present the estimate of $J_{2,2,3}$ as follows:
\begin{equation} 
\begin{aligned}
J_{2,2,3} 
&\leq  
\left|
\mathbb{E}
\left[
\int_{T-h}^{T} 
Du\big(T-s, \mathbb{Y}^{N-1}(s)\big) 
\Big(
f\big(\mathbb{Y}^{N-1}(t_{N-1}) \big) 
-  f\big(\mathbb{Y}^{N-1}(s) \big)
\Big) \dd s
\right] 
\right| \\
& \quad 
{\color{black}
+ 
\left|
\mathbb{E}
\left[
\int_{T-h}^{T}  
Du\big(T-s, \mathbb{Y}^{N-1}(t_{N-1}) \big)
\Big(
f\big(\mathbb{Y}^{N-1}(t_{N-1}) \big) 
-  f\big(\mathbb{Y}^{N-1}(s) \big)
\Big) \dd s
\right] \right|.
}
\end{aligned}
\end{equation}
{\color{black}
We can deduce from Lemma \ref{lemma:Uniform-moment-bounds-of-the-PLMC-paper3}, 
{\color{black}
\eqref{equation:Du-in-theorem-paper3} in }
Theorem \ref{theorem:regular-estimate-of-u-and-its-derivatives-paper3}, Lemma \ref{lemma:holder-continuity-of-PLMC algorithm-paper3} and the H\"older inequality that
\begin{equation} 
\begin{aligned}
J_{2,2,3} 
&\leq C 
\int_{T-h}^{T} 
\dfrac{\|\phi\|_{0}}{\sqrt{T-s}} 
\left\|
f\big(\mathbb{Y}^{N-1}(t_{N-1}) \big) 
-  f\big(\mathbb{Y}^{N-1}(s) \big)
\right\|_{L^{2 }(\Omega, \mathbb{R}^{d})}
\dd s  \\
& \leq
C
\int_{T-h}^{T} 
\dfrac{\|\phi\|_{0}}{\sqrt{T-s}} \dd s 
\left(
d^{\gamma-1/2} 
+
\|x_{0} \|^{2\gamma-1}_{L^{4\gamma-2 }(\Omega, \mathbb{R}^{d})}
\right)
h^{1/2}
\\
&\leq 
C
\left(
d^{\gamma-1/2} 
+
\|x_{0} \|^{2\gamma-1}_{L^{4\gamma-2 }(\Omega, \mathbb{R}^{d})}
\right)
h,
\end{aligned}
\end{equation}
where it is straightforward to show $\int_{T-h}^{T} \frac{1}{T-s} \dd s = 2h^{1/2}$.
}
{\color{black}
For $J_{2,2,4}$,
we 
split the interval $[(N - N_{1} - 1)h, (N - N_{1})h)$ into
$[(N - N_{1} - 1)h, Nh-1)$, $[Nh-1, (N - N_{1})h)$ and
follow the same argument as used in treating $J_{2,2,1}$,
$J_{2,2,2}$ above
to obtain
\begin{equation}
\begin{aligned}
J_{2,2,4}
&\leq
\bigg|  
\mathbb{E}
\bigg[
\int_{(N - N_{1} - 1)h}^{Nh-1} 
\int_{0}^{1}
D^{2}u\big(T-s, \upsilon^{N-N_{1}-1}_{2}(\bar{r})
\big) \\
&\hspace{4em}
\Big(
\mathbb{Y}^{N-N_{1}-1}(s)-\mathbb{Y}^{N-N_{1}-1}(t_{N-N_{1}-1}),
f\big(\mathbb{Y}^{N-N_{1}-1}(t_{N-N_{1}-1}) \big) 
-  f\big(\mathbb{Y}^{N-N_{1}-1}(s) \big)
\Big) 
\dd \bar{r}
\dd s
\bigg] 
\bigg| \\
& \quad
+
\bigg|  
\mathbb{E}
\bigg[
\int_{Nh-1}^{(N - N_{1})h} 
\int_{0}^{1}
D^{2}u\big(T-s, \upsilon^{N-N_{1}-1}_{2}(\bar{r})
\big) \\
&\hspace{4em}
\Big(
\mathbb{Y}^{N-N_{1}-1}(s)-\mathbb{Y}^{N-N_{1}-1}(t_{N-N_{1}-1}),
f\big(\mathbb{Y}^{N-N_{1}-1}(t_{N-N_{1}-1}) \big) 
-  f\big(\mathbb{Y}^{N-N_{1}-1}(s) \big)
\Big) 
\dd \bar{r}
\dd s
\bigg] 
\bigg| \\
& \leq
Cd^{\max\{3\gamma/2 , 2\gamma-1 \} } h^{2}
\left(
1+\left| \ln{h} \right|
\right).
\end{aligned}
\end{equation}
}
Combining all the estimates of $J_{2,2,1}$ to $J_{2,2,3}$ yields
\begin{equation} \label{equation:estimate-of-J2,2-paper3}
\begin{aligned}
J_{2,2} 
\leq 
Cd^{\max\{3\gamma/2 , 2\gamma-1 \} } h\left| \ln{h} \right|.
\end{aligned}
\end{equation}
Gathering \eqref{equation:estimate-of-J1-paper3}, \eqref{equation: estimate of J2,1-paper3} and \eqref{equation:estimate-of-J2,2-paper3} together
completes the proof of \eqref{equation:weak-error-for-super-linear-drift-paper3}.

\noindent
\textbf{Step II: $\gamma=1$:}

In the case $\gamma=1$,
the LMC algorithm, i.e. the Euler-Maruyama scheme \eqref{equation:numerical-scheme-euler-maruyama-paper3}, is equivalent to the PLMC algorithm \eqref{equation:numerical-scheme-PLMC-algorithm-paper3}. 
Recalling Lemma \ref{lemma:error-estimate-between-x-and-projected-x-paper3}, one will arrive at
\begin{equation}
    J_{1}=0
\end{equation}
in \eqref{eq:decomposition-of-the-weak-error-paper3}.
In conjunction with \eqref{equation: estimate of J2,1-paper3} and \eqref{equation:estimate-of-J2,2-paper3} with $\gamma=1$, 
{\color{black}
the proof 
is completed.
}
\end{proof}
To conclude, 
{\color{black}
the proof of Theorem \ref{theorem:main-result-paper3} is obtained 
due to \eqref{equation:error-decomposition-tv-paper3}.
}

{\color{black}
\section{
Optimal convergence rate for the Lipschitz case $\gamma = 1$
}
\label{section:optimal-error-analysis}
\subsection{Introduction to Malliavin calculus}
In this part, we give a brief introduction to Malliavin calculus, which is a key tool for us to obtain the optimal convergence rate. 
For details, one can refer to the classical monograph \cite{nualart2006malliavin}.
Denote $H = L^2([0,T], \mathbb{R}^d)$ 
as the space of all measurable functions 
$\varphi: [0,T] \rightarrow \mathbb{R}^{d}$ such that
$\int_{0}^{T} \|\varphi(t)\|^2  \dd t < \infty$.
In the following, we use 
$W(v)$, 
$v = ( [v(\cdot)]_{1}, \cdots, [v(\cdot)]_{d} )^{T} \in H$,
to denote the isonormal Gaussian process 
as
\begin{equation}
W(v)
:=
\sum_{i=1}^d 
\int_0^T [v(t)]_{i} \mathrm{d} W_{i, t} .
\end{equation}
Moreover, denote $C^{\infty}_{p}(\mathbb{R}^{M})$ as
the set of all infinitely continuously differentiable
functions
$g: \mathbb{R}^{M} \rightarrow \mathbb{R}$
such that $g$ and all of its partial derivatives have
polynomial growth.
Let $\mathcal{S}$ denote the class of smooth random variables such that a random
variable $F \in \mathcal{S}$ has the form
\begin{equation}
\label{equation:form-of-rv-paper3}
F
=g
\left(
W\left(v_1\right), \ldots,
W\left(v_M\right)\right),
\end{equation}
where
$g \in C^{\infty}_{p}(\mathbb{R}^{M})$ and
$v_1, \ldots, v_M \in H$, $M \geq 1$. 
We are in a position to give the definition of the derivative operator $\mathscr{D}$.
\begin{definition}
\label{def:Derivative operator-D-mc-paper3}
(Derivative operator $\mathscr{D}$)
The derivative of a smooth random variable $F$ of the form \eqref{equation:form-of-rv-paper3} is the $H-$valued random variable given by
\begin{equation}
\mathscr{D}_t F
=
\sum_{i=1}^M 
\partial_i 
g\left(
W\left(v_1\right), \ldots, W\left(v_M\right)
\right) v_i(t), \quad 0 \leq t \leq T.
\end{equation}
Further, given $v\in H$,
denote $\mathscr{D}^{v}F$ by the Malliavin derivative with respect to $F$ along the direction $v$ as follows:
\begin{equation}
\mathscr{D}^{v}F:=
\int_{0}^{T}
\left\langle
v(s), \mathscr{D}_{s}F
\right\rangle
\dd s.
\end{equation}
\end{definition} 
For any $p \geq 1$,
denote 
by $\mathbb{D}^{1,2}$ as
the closure of the class of smooth random variables in $\mathcal{S}$ with
respect to the norm
\begin{equation}
\|F\|_{\mathbb{D}^{1,2}}
:=
\left(
\mathbb{E}
\left[|F|^2\right]
+\mathbb{E}
\left[
\int_{0}^{T}
\|\mathscr{D}_{s} F\|^{2}
\dd s
\right]
\right)^{1/2} .
\end{equation}
We then present the definition of the divergence operator $\delta$
as well as 
 the useful integration by parts formula.
\begin{definition}
\label{def:Divergence operator-paper3}
(Divergence operator $\delta$)
Denote by $\delta$ the adjoint of the operator $\mathscr{D}$.
That is,
$\delta$
is an unbounded operator on $L^{2}(\Omega, H)$ with values in $L^{2}(\Omega, \mathbb{R})$ such that:
\begin{enumerate}
\item[(i)]
The domain of $\delta$, denoted by $\operatorname{Dom} \delta$, is the set of $\mathbb{R}^{d}$-valued square integrable random variables $u \in L^2(\Omega, H)$ such that
\begin{equation}
\left|
\mathbb{E}
\left[
\mathscr{D}^{u} F
\right]
\right| 
\leq C
\|F\|_{ L^2(\Omega, \mathbb{R})},
\end{equation}
for all $F \in \mathbb{D}^{1,2}$, where $C$ is some constant depending on $u$.
\item[(ii)]
If $u$ belongs to $\operatorname{Dom} \delta$, then $\delta(u)$ is the element of $L^2(\Omega, \mathbb{R})$ characterized by
\begin{equation}
\label{equation:integration-ny-part-in-def-paper3}
\mathbb{E}
\left[
F \delta(u)
\right]
=
\mathbb{E}
\left[
\mathscr{D}^{u} F
\right],
\end{equation}
for any $F \in \mathbb{D}^{1,2}$.
\end{enumerate}
\end{definition}
The next lemma
gives the important calculus rule (see \cite[Proposition 1.3.3]{nualart2006malliavin}), which allows us to factor out a scalar random variable in a divergence.
\begin{lemma}
\label{lemma:calculus-rules-paper3}
Let $F \in \mathbb{D}^{1,2}$ and $u \in \operatorname{Dom} \delta$ such that $Fu \in L^2(\Omega, H)$. Then $Fu \in \operatorname{Dom} \delta$ and the following equality is true
\begin{equation}
\label{equation:calculus-rules-in-lemma-paper3}
\delta (Fu) = F \delta(u) -   \mathscr{D}^{u} F,
\end{equation}
provided the right-hand side of \eqref{equation:calculus-rules-in-lemma-paper3} is square integrable.
\end{lemma}

\subsection{Proof of Theorem \ref{theorem:main-result-paper3-optimal}: optimal weak convergence rate}
In this subsection, we start to show an optimal weak convergence rate for the time discretization.
Revisiting the error analysis in the proof of Theorem \ref{theorem: Time-independent weak error analysis-paper3},
one can straightforwardly find that the logarithmic factor $\ln h $ only appears in $J_{2,2,2}$ from \eqref{equation:decomposition-J2-paper3,2}.
Thus 
a refinement of the estimate $J_{2,2,2}$,
which lies in the interval $[T-N_{1}h, T-h)$,
is required in 
order to 
remove the logarithmic term $\ln h $  and
achieve the optimal convergence rate.

In the following,
we set $T \geq 3/2$  without any loss of generality.
Then,
for the grid points $t_{n}$, $t_{n+1}$
satisfying
$T-N_{1}h \leq t_n < t_{n+1} \leq T-h <T$,
where
$N_1 h\leq 1 <(N_1+1)h$,
one is able to find a grid point $t_{m}$ such that,
\begin{equation}
\label{equation:time-interval-in-mc-paper3}
0 \leq t_{m} < t_{n} < t_{n+1} <T, \quad \text{with} \quad
\tfrac{1}{2} \leq t_{n+1} - t_{m} \leq \tfrac{2}{3}.
\end{equation}
Since the 
LMC algorithm \eqref{equation:numerical-scheme-euler-maruyama-paper3}, denoted by $\{\widetilde{Y}_{n}\}_{n \geq 0}$, is equivalent to the PLMC algorithm \eqref{equation:numerical-scheme-PLMC-algorithm-paper3} with $\gamma=1$,
we still use \eqref{introduction:continuous-time-version-of-PLMC algorithm-paper3} as the continuous extension of the LMC algorithm in which $\mathbb{Y}^{n}(t_{n}) = \mathscr{P}(Y_n) = \widetilde{Y}_{n}$.
Also, we introduce the following process:
\begin{equation}
\label{equation:def-el-paper3}
b_{\ell}(t):=
\frac{\sqrt{2}}{2}
\sum_{i=m}^{n}
\textbf{1}_{[t_{i}, t_{i+1})}(t)
\left( 
\tfrac{
e_{\ell}
}{t_{n+1} - t_{m}} 
- 
\tfrac{t_i - t_{m} }{t_{n+1} - t_{m}} 
D f
(\widetilde{Y}_i) 
e_{\ell}
\right),
\quad
t \in [t_m, t_{n+1}),
\end{equation}
where  $\{ e_{\ell}\}_{\ell\in \{1, \cdots,d \} }$ denotes the orthonormal basis of $\mathbb{R}^{d}$.
We directly observe that the process \eqref{equation:def-el-paper3}  is adapted with respect to $\mathcal{F}_t$, $t\in [t_{m}, t_{n+1})$, and $b_{\ell}(\cdot) \in L^{2}(\Omega, H)$ 
due to
Assumption \ref{assumption:globally-polynomial-growth-condition-paper3} and Lemma \ref{lemma:Uniform-moment-bounds-of-the-PLMC-paper3} with $\gamma=1$.
Thanks to \cite[Proposition 1.3.11]{nualart2006malliavin},
we deduce that $b_{\ell} \in \operatorname{Dom} \delta$ and
\begin{equation}
\label{equation:def-of-b-paper3}
\delta(b_{\ell}) =
\int_{t_{m}}^{t_{n+1}} 
\left\langle
b_{\ell}, \dd W_{s}
\right\rangle
.
\end{equation}
Now we are in the position to show the  Malliavin regularity estimates of the process \eqref{introduction:continuous-time-version-of-PLMC algorithm-paper3} with respect to the direction $b_{\ell}$.
\begin{lemma}
\label{lemma:Malliavin regularity-paper3}
Let Assumptions 
\ref{assumption:globally-polynomial-growth-condition-paper3},  \ref{assumption:contractivity-at-infinity-condition-paper3} and  \ref{assumption:coercivity-condition-of-the-drift-paper3} 
hold with $\gamma=1$. 
Given the direction $b_{\ell}$ as defined by \eqref{equation:def-el-paper3},
for any $p \in [1, \infty)$,
there exists a positive constant $C$ depending  on the drift $f$ and $p$ such that,
\begin{equation}
\label{eqaution:estimate-of-divergence-operator-paper3}
\left\|
\delta(b_{\ell}) 
\right\|_{L^{2p}(\Omega, \mathbb{R})} 
\leq
C,
\end{equation}
and, for 
$t_m \leq t_{i} \leq s \leq t_{i+1} \leq t_{n+1}$,
\begin{equation}
\label{equation:MC-Dy-paper3}
\mathscr{D}^{b_{\ell}} \mathbb{Y}^{i}(s) 
= 
\tfrac{s - t_{m}}{t_{n+1} - t_{m}} 
e_{\ell}.
\end{equation}
Moreover, 
there exists some positive constant $C$ depending on $f$ and $p$ such that
\begin{equation}
\begin{aligned}
\left\|
\mathscr{D}^{b_{\ell}} 
\mathbb{Y}^{n}(s) 
-
\mathscr{D}^{b_{\ell}} 
\mathbb{Y}^{n}(t_{n})
\right\|
 &\leq
2h, \\
\left\|
\mathscr{D}^{b_{\ell}}
f\big(
\mathbb{Y}^{n}(s)
\big)
-
\mathscr{D}^{b_{\ell}} 
f
\big(
\mathbb{Y}^{n}(t_{n}) 
\big)
\right\|_{L^{2p}(\Omega, \mathbb{R}^{d})} 
&\leq
Ch^{1/2}
\left(
d^{1/2} + \|x_{0}\|_{L^{2p}(\Omega, \mathbb{R}^{d})}
\right).
\end{aligned}
\end{equation}
\end{lemma}
\begin{proof}
[Proof of Lemma \ref{lemma:Malliavin regularity-paper3}.]
In light of
Assumption \ref{assumption:globally-polynomial-growth-condition-paper3} with $\gamma=1$,
\eqref{equation:time-interval-in-mc-paper3}, 
\eqref{equation:def-el-paper3}, \eqref{equation:def-of-b-paper3} and the Burkholder-Davis-Gundy inequality,
one derives,
for $p \in [1, \infty)$,
\begin{equation}
\left\|
\delta(b_{\ell}) 
\right\|_{L^{2p}(\Omega, \mathbb{R})}
\leq
\left(
\int_{t_{m}}^{t_{n+1}}
\left\|
b_{\ell}(s)
\right\|^{2}_{L^{2p}(\Omega, \mathbb{R}^{d})}
\dd s
\right)^{1/2}
\leq
C.
\end{equation}
Thus the proof of estimate \eqref{eqaution:estimate-of-divergence-operator-paper3} is finished.
Before proceeding further,
we first show that the estimate \eqref{equation:MC-Dy-paper3} holds true 
at any grid point in time interval $[t_{m}, t_{n+1}]$ as
\begin{equation}
\label{equation:DY-mc-grid-paper3}
\mathscr{D}^{b_{\ell}} \mathbb{Y}^{i}(t_{i})
=
\mathscr{D}^{b_{\ell}} 
\widetilde{Y}_i
=
\tfrac{t_{i} - t_{m}}{t_{n+1} - t_{m}} 
e_{\ell}.
\end{equation}
Thanks to the chain rule, \eqref{equation:numerical-scheme-euler-maruyama-paper3} and \eqref{equation:def-of-b-paper3}, we have
\begin{equation}
\label{equation:expression-LMC-MC-paper3}
\begin{aligned} 
\mathscr{D}^{b_{\ell}}
\widetilde{Y}_{i+1} 
&=
\mathscr{D}^{b_{\ell}}
\widetilde{Y}_i
+
h 
\mathscr{D}^{b_{\ell}} 
f(
\widetilde{Y}_i 
) 
+
\sqrt{2}
\left(
\int_{t_{i}}^{t_{i+1}}
b_{\ell}(s)
\dd s
\right) \\
& = 
\mathscr{D}^{b_{\ell}}
\widetilde{Y}_i 
+
h D 
f(
\widetilde{Y}_i
)  
\mathscr{D}^{b_{\ell}}
\widetilde{Y}_i
+
h
\left( 
\tfrac{
e_{\ell}
}{
t_{n+1} - t_{m}
} 
- 
\tfrac{t_i - t_{m} }{t_{n+1} - t_{m}}
D 
f
(\widetilde{Y}_i)
e_{\ell} \right).
\end{aligned}
\end{equation}
It is straightforward that the estimate \eqref{equation:DY-mc-grid-paper3} is trivial for $i=m$.
We then
suppose that 
\eqref{equation:DY-mc-grid-paper3} holds for $m < i< n+1 $ 
and
obtain the  estimate of $\mathscr{D}^{b_{\ell}}
\widetilde{Y}_{i+1}$  by \eqref{equation:expression-LMC-MC-paper3} as follows,
\begin{equation}
\begin{aligned}
\mathscr{D}^{b_{\ell}}
\widetilde{Y}_{i+1}
&=
\tfrac{t_i - t_{m}}{t_{n+1} - t_{m}} 
e_{\ell}
+
h \tfrac{t_{i} - t_{m}}{t_{n+1} - t_{m}} 
D f
(\widetilde{Y}_{i})
e_{\ell} 
+
h
\left( \tfrac{e_l}{t_{n+1} - t_{m}} 
- 
\tfrac{t_i - t_{m} }{t_{n+1} - t_{m}} 
D 
f
(\widetilde{Y}_{i})
e_{\ell} \right)
\\
&=
\tfrac{t_{i+1} - t_{m}}{t_{n+1} - t_{m}} 
e_{\ell},
\end{aligned}
\end{equation}
where we have used the fact that $h = t_{i+1} - t_{i}$.
Then it follows from the induction that the estimate \eqref{equation:DY-mc-grid-paper3} holds for $m \leq i \leq n+1$.
In this way,
we use \eqref{equation:DY-mc-grid-paper3}
and the continuous version \eqref{introduction:continuous-time-version-of-PLMC algorithm-paper3} of the LMC with $\mathbb{Y}^{i}(t_i) = \widetilde{Y}_{i}$
and the chain rule
to get, 
for $t_m \leq t_{i} \leq s \leq t_{i+1} \leq t_{n+1}$,
\begin{equation}
\label{equation:MC-Dy-paper3-in-proof}
\begin{aligned}
\mathscr{D}^{b_{\ell}} \mathbb{Y}^{i}(s) 
&=
\mathscr{D}^{b_{\ell}}
\mathbb{Y}^{i}(t_i)
+
(s-t_i) 
\mathscr{D}^{b_{\ell}} 
f(
\mathbb{Y}^{i}(t_i)
) 
+
\sqrt{2}
\left(
\int_{t_{i}}^{s}
b_{\ell}(s)
\dd s
\right) \\
& =
\tfrac{t_i - t_{m}}{t_{n+1} - t_{m}} 
e_{\ell}
+
(s-t_i)
\tfrac{t_{i} - t_{m}}{t_{n+1} - t_{m}} 
D f
(\widetilde{Y}_{i})
e_{\ell} 
+
(s-t_i)
\left( \tfrac{e_l}{t_{n+1} - t_{m}} 
- 
\tfrac{t_i - t_{m} }{t_{n+1} - t_{m}} 
D 
f
(\widetilde{Y}_{i})
e_{\ell} \right) \\
&=
\tfrac{s - t_{m}}{t_{n+1} - t_{m}} 
e_{\ell}.
\end{aligned}
\end{equation}
The estimate of \eqref{equation:MC-Dy-paper3} is thus finished.

Recalling \eqref{introduction:continuous-time-version-of-PLMC algorithm-paper3},
\eqref{equation:time-interval-in-mc-paper3},
\eqref{equation:DY-mc-grid-paper3} and \eqref{equation:MC-Dy-paper3-in-proof},
one has,
for $t_{m} < t_{n} \leq s \leq t_{n+1}$,
\begin{equation}
\label{equation:MC-for-delta-y-paper3}
\begin{aligned}
\left\|
\mathscr{D}^{b_{\ell}}
\mathbb{Y}^{n}(s) 
-
\mathscr{D}^{b_{\ell}}
\mathbb{Y}^{n}(t_{n})
\right\|
&=
\tfrac{s - t_{n}}{t_{n+1} - t_{m}}
\left\|
e_{\ell}
\right\|
\leq
2h.
\end{aligned}
\end{equation}
Equipped with \eqref{equation:MC-for-delta-y-paper3} above,
we then use Assumption \ref{assumption:globally-polynomial-growth-condition-paper3}, Lemma \ref{lemma:holder-continuity-of-PLMC algorithm-paper3} with $\gamma=1$, the chain rule, \eqref{equation:time-interval-in-mc-paper3}  and \eqref{equation:DY-mc-grid-paper3} to derive,
for $p \in [1, \infty)$,
\begin{equation}
\begin{aligned}
&\left\|
\mathscr{D}^{b_{\ell}} 
f\big(
\mathbb{Y}^{n}(s) 
\big) 
-
\mathscr{D}^{b_{\ell}} 
f
\big(
\mathbb{Y}^{n}(t_{n})
\big) 
\right\|_{L^{2p}(\Omega, \mathbb{R}^{d})} \\
&=
\left\|
D
f\big(
\mathbb{Y}^{n}(s) 
\big)
\mathscr{D}^{b_{\ell}}
\mathbb{Y}^{n}(s)
-
D
f
\big(
\mathbb{Y}^{n}(t_{n}) 
\big)
\mathscr{D}^{b_{\ell}} 
\mathbb{Y}^{n}(t_{n}) 
\right\|_{L^{2p}(\Omega, \mathbb{R}^{d})} \\
&
\leq
\left\|
D
f\big(
\mathbb{Y}^{n}(s) 
\big)
\mathscr{D}^{b_{\ell}}
\mathbb{Y}^{n}(s)
-
D
f
\big(
\mathbb{Y}^{n}(s) 
\big)
\mathscr{D}^{b_{\ell}} 
\mathbb{Y}^{n}(t_{n}) 
\right\|_{L^{2p}(\Omega, \mathbb{R}^{d})}
\\
& \quad
+
\left\|
D
f\big(
\mathbb{Y}^{n}(s) 
\big)
\mathscr{D}^{b_{\ell}}
\mathbb{Y}^{n}(t_{n})
-
D
f
\big(
\mathbb{Y}^{n}(t_{n}) 
\big)
\mathscr{D}^{b_{\ell}} 
\mathbb{Y}^{n}(t_{n})
\right\|_{L^{2p}(\Omega, \mathbb{R}^{d})}
\\
& \leq
C
\left(
\left\|
\mathscr{D}^{b_{\ell}} 
\mathbb{Y}^{n}(s) 
-
\mathscr{D}^{b_{\ell}}
\mathbb{Y}^{n}(t_{n}) 
\right\|
+
\left\|
\mathbb{Y}^{n}(s) - \mathbb{Y}^{n}(t_n)
\right\|_{L^{2p}(\Omega, \mathbb{R}^{d})}
\left\|
\mathscr{D}^{b_{\ell}}
\mathbb{Y}^{n}(t_{n}) 
\right\|
\right)
\\
&\leq
Ch^{1/2}
\left(
d^{1/2} + \|x_{0}\|_{L^{2p}(\Omega, \mathbb{R}^{d})}
\right),
\end{aligned}
\end{equation}
where $C$ is some positive constant depending on the drift $f$.
The proof is completed.
\end{proof}
At this stage,
we are well-prepared to derive the optimal convergence rate of the LMC \eqref{equation:numerical-scheme-euler-maruyama-paper3}.
\begin{theorem}
\label{theorem:optimal-convergence-rate-lipschitz-setting-paper3}
Let Assumptions \ref{assumption:globally-polynomial-growth-condition-paper3}, \ref{assumption:contractivity-at-infinity-condition-paper3} and \ref{assumption:coercivity-condition-of-the-drift-paper3} hold with $\gamma=1$. Let $\{X^{x_{0}}_{t}\}_{t\geq 0}$ and $\{\widetilde{Y}^{x_{0}}_{n} \}_{n \geq 0}$ be the solutions of SDE \eqref{eq:langevin-SODE} and the LMC algorithm \eqref{equation:numerical-scheme-euler-maruyama-paper3} with the same initial state $X^{x_{0}}_{0} = Y^{x_{0}}_{0} = x_{0}$, respectively. Also, let 
$h\in (0, \min\{1/2a_{1}, 2a_{1}/(a_{1}+2C^{2}_{f}), 1/4\} )$, 
where $a_{1}$ and $C_f$ are given in Assumption \ref{assumption:coercivity-condition-of-the-drift-paper3} and Lemma \ref{lemma: useful estimate of PLMC algorithm-paper3}, respectively,
be the uniform timestep. Then for 
 any
test function
$\phi \in C_{b}(\mathbb{R}^{d})$, 
\begin{equation} \label{equation:weak-error-for-Lipshitz-drift-paper3}
\Big| 
\mathbb{E}
\big[
\phi\big(
\widetilde{Y}^{x_{0}}_{N}
\big) 
\big] 
- 
\mathbb{E}\left[
\phi\big(
X^{x_{0}}_{T}
\big) 
\right]
\Big| 
\leq 
Cd^{3/2} h.
\end{equation}
\end{theorem}

We mention that some techniques in the proof of Theorem \ref{theorem:optimal-convergence-rate-lipschitz-setting-paper3} are borrowed from \cite{li2023unadjusted}, where the non-asymptotic analysis of the LMC \eqref{equation:numerical-scheme-euler-maruyama-paper3} with decreasing step sizes was carried out. It is worthwhile to point out that the error constants in \cite{li2023unadjusted} essentially depend on the maximum step size in the decreasing sequence. This fact makes the error analysis of LMC \eqref{equation:numerical-scheme-euler-maruyama-paper3} with uniform step sizes here isnon-trival and not a direct consequence of that in \cite{li2023unadjusted}.

\begin{proof}
[Proof of Theorem \ref{theorem:optimal-convergence-rate-lipschitz-setting-paper3}.]
According to the proof of Theorem \ref{theorem: Time-independent weak error analysis-paper3},
the weak error can be decomposed based on the Kolmogorov equation as
\eqref{eq:decomposition-of-the-weak-error-paper3}.
Recalling
$x-\mathscr{P}(x) = 0, \forall x \in \mathbb{R}^{d}$
for $\gamma=1$ in
Lemma \ref{lemma:error-estimate-between-x-and-projected-x-paper3}, one will arrive at
\begin{equation}
\label{equation:J1-lipschitz-paper3}
    J_{1}=0.
\end{equation}
Regarding $J_2$,
we only need to improve the estimate of $J_{2,2,2}$ in \eqref{equation:decomposition-J2-paper3}:
\begin{equation}
J_{2,2,2}:=
\sum_{n=N-N_{1}}^{N-2} 
\left| 
\mathbb{E}
\left[
\int_{t_{n}}^{t_{n+1}}  
\int_{0}^{1}
D^{2}u\big(T-s, \upsilon^{n}_{2}(\bar{r})
\big)
\Big(
\mathbb{Y}^{n}(s)-\mathbb{Y}^{n}(t_{n}),
f\big(\mathbb{Y}^{n}(t_{n}) \big) 
-  f\big(\mathbb{Y}^{n}(s) \big)
\Big) 
\dd \bar{r}
\dd s
\right] 
\right|,
\end{equation}
where
$\upsilon_{2}^{n}(\bar{r}) :=\mathbb{Y}^{n}(t_{n}) + \bar{r}
(\mathbb{Y}^{n}(s)-\mathbb{Y}^{n}(t_{n}))$, $\bar{r}\in [0,1]$.
Indeed, by setting $\gamma=1$ in the proof of Theorem \ref{theorem: Time-independent weak error analysis-paper3}, the remaining terms in \eqref{equation:decomposition-J2-paper3} and \eqref{equation:decomposition-J2-paper3,2} can be bounded to be optimal:
\begin{equation}
\label{equation:J2-remaining-paper3}
J_{2,1} \vee J_{2,2,1} \vee J_{2,2,3}
\vee J_{2,2,4}
\leq Cd^{3/2}h.
\end{equation}
To improve the estimate $J_{2,2,2}$,
we first use the chain rule and Lemma \ref{lemma:Malliavin regularity-paper3} to derive,
for $i, \ell \in \{1, \cdots, d\}$, 
$\bar{r}\in [0,1]$,
\begin{equation}
\label{equation:D2u-to-malliavin-derivatives-paper3}
\begin{aligned}
\mathscr{D}^{b_{\ell}}
\Big(
Du\big(
T-s, \upsilon_{2}^{n}(\bar{r})
\big)
e_{i}
\Big)
&=
D^{2}
u
\big(
T-s, \upsilon_{2}^{n}(\bar{r})
\big)
\big(
e_{i}, \mathscr{D}
^{b_{\ell}}
\upsilon_{2}^{n}(\bar{r})
\big) \\
&=
\tfrac{t_n - t_m + \bar{r}(s-t_n)} {t_{n+1} - t_{m}}
D^{2}u\big(T-s, \upsilon_{2}^{n}(\bar{r})
\big)
\big(
e_{i}, e_{\ell}
\big),
\end{aligned}
\end{equation}
where we recall $\{ e_{j}\}_{j\in \{1, \cdots,d \} }$ is the orthonormal basis of $\mathbb{R}^{d}$.
Further, it follows from \eqref{equation:time-interval-in-mc-paper3} that for $0 < h \leq 1/4$
\begin{equation}
\label{equation:malliavin-derivative-of-v-paper3}
\tfrac{t_n - t_m + \bar{r}(s-t_n)} {t_{n+1} - t_{m}}
\geq
1- \tfrac{h} {t_{n+1} - t_{m}} \geq
\tfrac{1}{2}.
\end{equation}
By Definition \ref{def:Derivative operator-D-mc-paper3},
the integration by parts formula for
Malliavin derivative, i.e., \eqref{equation:integration-ny-part-in-def-paper3} in Definition \ref{def:Divergence operator-paper3}, \eqref{equation:D2u-to-malliavin-derivatives-paper3} and \eqref{equation:malliavin-derivative-of-v-paper3},
we rewrite $J_{2,2,2}$ as,
for 
$\mathbb{Y}^{n}(\cdot) 
= (
[\mathbb{Y}^{n}(\cdot)]_{1}, \cdots, 
[\mathbb{Y}^{n}(\cdot)]_{d}
)^{T}$,
$f(\cdot) = ([f(\cdot)]_{1}, \cdots, [f(\cdot)]_{d})^{T}$,
\begin{equation}
\label{equation:decomposition-J2-paper3,2,2-mc}
\begin{aligned}
J_{2,2,2}
&=
\sum_{n=N-N_{1}}^{N-2} 
\Bigg| 
\sum_{i, \ell =1}^{d}
\mathbb{E}
\bigg[
\int_{t_{n}}^{t_{n+1}}  
\int_{0}^{1}
D^{2}u\big(T-s, \upsilon_{2}^{n}(\bar{r})
\big)
\big(
e_{i}, e_{\ell}
\big)   \\
&
\hspace{8em}
\times
\big(
[\mathbb{Y}^{n}(s)]_{i}
-
[\mathbb{Y}^{n}(t_{n})]_{i}
\big)
\Big(
\big[
f\big(\mathbb{Y}^{n}(t_{n}) \big) 
\big]_{\ell}
-
\big[
f\big(\mathbb{Y}^{n}(s) \big)
\big]_{\ell}
\Big) 
\mathrm{d} \bar{r}
\mathrm{d} s
\bigg] 
\Bigg|, \\
& \leq
2
\sum_{n=N-N_{1}}^{N-2} 
\Bigg| 
\sum_{i, \ell =1}^{d}
\mathbb{E}
\bigg[
\int_{t_{n}}^{t_{n+1}}  
\int_{0}^{1}
\mathscr{D}^{b_{\ell}}
\Big(
Du
\big(T-s, \upsilon_{2}^{n}(\bar{r})
\big)
e_{i}
\Big)  \\
&
\hspace{8em}
\times
\big(
[
\mathbb{Y}^{n}(s)
]_{i}
-
[
\mathbb{Y}^{n}(t_{n})
]_{i}
\big)
\Big(
\big[
f\big(\mathbb{Y}^{n}(t_{n}) \big) 
\big]_{\ell}
- 
\big[
f\big(\mathbb{Y}^{n}(s) \big)
\big]_{\ell}
\Big) 
\mathrm{d} \bar{r}
\mathrm{d} s
\bigg] 
\Bigg|, \\
& = 
2
\sum_{n=N-N_{1}}^{N-2} 
\Bigg| 
\sum_{i, \ell =1}^{d}
\mathbb{E}
\bigg[
\int_{t_{n}}^{t_{n+1}}  
\int_{0}^{1}
\int_{t_m}^{t_{n+1}}
\Big\langle
b_{\ell}(\widetilde{r})
,
\mathscr{D}_{\widetilde{r}}
\left(
Du\big(
T-s, \upsilon_{2}^{n}(\bar{r}) 
\big)
e_{i}
\right)
\Big\rangle
 \\
&
\hspace{8em}
\times
\big(
[\mathbb{Y}^{n}(s)]_{i}
-
[\mathbb{Y}^{n}(t_{n})]_{i}
\big)
\Big(
\big[
f\big(\mathbb{Y}^{n}(t_{n}) \big) 
\big]_{\ell}
-  
\big[
f\big(\mathbb{Y}^{n}(s) \big)
\big]_{\ell}
\Big) 
\mathrm{d} \widetilde{r}
\mathrm{d} \bar{r}
\mathrm{d} s
\bigg] 
\Bigg|, \\
& = 
2
\sum_{n=N-N_{1}}^{N-2} 
\Bigg| 
\sum_{i, \ell =1}^{d}
\mathbb{E}
\bigg[
\int_{t_{n}}^{t_{n+1}}  
\int_{0}^{1}
Du\big(
T-s, \upsilon_{2}^{n}(\bar{r}) 
\big)
e_{i}
 \\
&
\hspace{8em}
\times
\delta
\bigg(
b_{\ell}
\big(
[\mathbb{Y}^{n}(s)]_{i}
-
[\mathbb{Y}^{n}(t_{n})]_{i}
\big)
\Big(
\big[
f\big(\mathbb{Y}^{n}(t_{n}) \big) 
\big]_{\ell}
-  
\big[
f\big(\mathbb{Y}^{n}(s) \big)
\big]_{\ell}
\Big) 
\bigg)
\mathrm{d} \bar{r}
\mathrm{d} s
\bigg] 
\Bigg|, \\
\end{aligned}
\end{equation}
where it is straightforward that $D(t, \cdot)e_{i} \in \mathbb{D}^{1,2}$ and 
$b_{\ell}
\big(
[
\mathbb{Y}^{n}(s)
]_{i}
-
[\mathbb{Y}^{n}(t_{n})]_{i}
\big)
\big(
[
f\big(\mathbb{Y}^{n}(t_{n}) \big) 
]_{\ell}
- 
[
f\big(\mathbb{Y}^{n}(s) \big)
]_{\ell}
\big) \in \operatorname{Dom} \delta$ due to Assumption \ref{assumption:globally-polynomial-growth-condition-paper3} with $\gamma=1$, Lemma \ref{lemma:Uniform-moment-bounds-of-the-PLMC-paper3}, Theorem \ref{theorem:regular-estimate-of-u-and-its-derivatives-paper3} and Lemma \ref{lemma:Malliavin regularity-paper3}.
Using Lemma \ref{lemma:calculus-rules-paper3} and the chain rule yields
\begin{equation}
\begin{aligned}
&\delta
\bigg(
b_{\ell}
\big(
[
\mathbb{Y}^{n}(s)
]_{i}
-
[
\mathbb{Y}^{n}(t_{n})
]_{i}
\big)
\Big(
\big[
f\big(\mathbb{Y}^{n}(t_{n}) \big) 
\big]_{\ell}
-  
\big[
f\big(\mathbb{Y}^{n}(s) \big)
\big]_{\ell}
\Big) 
\bigg) \\
&=
\big(
[
\mathbb{Y}^{n}(s)
]_{i}
-
[
\mathbb{Y}^{n}(t_{n})
]_{i}
\big)
\Big(
\big[
f\big(\mathbb{Y}^{n}(t_{n}) \big) 
\big]_{\ell}
-  
\big[
f\big(\mathbb{Y}^{n}(s) \big)
\big]_{\ell}
\Big)
\delta(b_{\ell})  \\
&\quad
-
\Big(
\big[
f\big(\mathbb{Y}^{n}(t_{n}) \big) 
\big]_{\ell}
-  
\big[
f\big(\mathbb{Y}^{n}(s) \big)
\big]_{\ell}
\Big) \
\Big(
\mathscr{D}^{b_{\ell}}
[
\mathbb{Y}^{n}(s)
]_{i}
-
\mathscr{D}^{b_{\ell}}
[
\mathbb{Y}^{n}(t_{n})
]_{i}
\Big)
\\
& \quad
-
\big(
[
\mathbb{Y}^{n}(s)
]_{i}
-
[
\mathbb{Y}^{n}(t_{n})
]_{i}
\big) \
\Big(
\mathscr{D}^{b_{\ell}}
\big[
f\big(\mathbb{Y}^{n}(t_{n}) \big) 
\big]_{\ell}
-  
\mathscr{D}^{b_{\ell}}
\big[
f\big(\mathbb{Y}^{n}(s) \big) 
\big]_{\ell}
\Big).
\end{aligned}
\end{equation}
Plugging this estimate into \eqref{equation:decomposition-J2-paper3,2,2-mc} shows,
\begin{equation}
\begin{aligned}
J_{2,2,2} 
&\leq
2
\sum_{n=N-N_{1}}^{N-2} 
\Bigg| 
\sum_{i, \ell =1}^{d}
\mathbb{E}
\bigg[
\int_{t_{n}}^{t_{n+1}}  
\int_{0}^{1}
Du\big(T-s, \upsilon_{2}^{n}(\bar{r}) 
\big)
e_{i}
 \\
&
\hspace{6em}
\times
\big(
[
\mathbb{Y}^{n}(s)
]_{i}
-
[
\mathbb{Y}^{n}(t_{n})
]_{i}
\big)
\Big(
\big[
f\big(\mathbb{Y}^{n}(t_{n}) \big) 
\big]_{\ell}
-  
\big[
f\big(\mathbb{Y}^{n}(s) \big)
\big]_{\ell}
\Big)
\delta(b_{\ell})
\mathrm{d} \bar{r}
\mathrm{d} s
\bigg] 
\Bigg| \\
&\quad +
2
\sum_{n=N-N_{1}}^{N-2} 
\Bigg| 
\sum_{i, \ell =1}^{d}
\mathbb{E}
\bigg[
\int_{t_{n}}^{t_{n+1}}  
\int_{0}^{1}
Du\big(T-s, \upsilon_{2}^{n}(\bar{r})
\big)
e_{i}
 \\
&
\hspace{6em}
\times
\Big(
\big[
f\big(\mathbb{Y}^{n}(t_{n}) \big) 
\big]_{\ell}
-  
\big[
f\big(\mathbb{Y}^{n}(s) \big)
\big]_{\ell}
\Big) \
\Big(
\mathscr{D}^{b_{\ell}}
[
\mathbb{Y}^{n}(s)
]_{i}
-
\mathscr{D}^{b_{\ell}}
[
\mathbb{Y}^{n}(t_{n})
]_{i}
\Big)
\mathrm{d} \bar{r}
\mathrm{d} s
\bigg] 
\Bigg| \\
&\quad +
2
\sum_{n=N-N_{1}}^{N-2} 
\Bigg| 
\sum_{i, \ell =1}^{d}
\mathbb{E}
\bigg[
\int_{t_{n}}^{t_{n+1}}  
\int_{0}^{1}
Du\big(T-s, \upsilon_{2}^{n}(\bar{r})
\big)
e_{i}
 \\
&
\hspace{6em}
\times
\big(
[
\mathbb{Y}^{n}(s)
]_{i}
-
[
\mathbb{Y}^{n}(t_{n})
]_{i}
\big) \
\Big(
\mathscr{D}^{b_{\ell}}
\big[
f\big(\mathbb{Y}^{n}(t_{n}) \big) 
\big]_{\ell}
-  
\mathscr{D}^{b_{\ell}}
\big[
f\big(\mathbb{Y}^{n}(s) \big) 
\big]_{\ell}
\Big)
\mathrm{d} \bar{r}
\mathrm{d} s
\bigg] 
\Bigg| \\
&=:
J^{(1)}_{2,2,2}
+
J^{(2)}_{2,2,2}
+
J^{(3)}_{2,2,2}.
\end{aligned}
\end{equation}
Taking the H\"older inequality, Theorem \ref{theorem:regular-estimate-of-u-and-its-derivatives-paper3}, Lemma \ref{lemma:holder-continuity-of-PLMC algorithm-paper3},
\eqref{equation:uniform-estimate-for-sum-paper3}
and Lemma \ref{lemma:Malliavin regularity-paper3} into account, we get
\begin{equation}
\label{equation:J2221-paper3}
\begin{aligned}
J^{(1)}_{2,2,2} 
&\leq
C
\sum_{n=N-N_{1}}^{N-2} 
\int_{t_{n}}^{t_{n+1}} 
\frac{1}{\sqrt{T-s}}
\left\|
\mathbb{Y}^{n}(s)-\mathbb{Y}^{n}(t_{n})
\right\|_{L^{4}(\Omega, \mathbb{R}^{d})} 
 \\
& \hspace{7em}
\times
\left\|
f\big(\mathbb{Y}^{n}(t_{n}) \big) 
-  
f\big(\mathbb{Y}^{n}(s) \big)
\right\|_{L^{4}(\Omega, \mathbb{R}^{d})}
\left\|
\delta(b_{\ell})
\right\|_{L^{2}(\Omega, \mathbb{R})} \dd s
\\
& \leq
C
\left(
d + 
\|x_{0}\|^{2}_{L^{4}(\Omega, \mathbb{R}^{d})}
\right) h.
\end{aligned}
\end{equation}
In the same manner, we obtain
\begin{equation}
\label{equation:J2222-paper3}
\begin{aligned}
J^{(2)}_{2,2,2} 
&\leq
C
\sum_{n=N-N_{1}}^{N-2} 
\int_{t_{n}}^{t_{n+1}} 
\frac{1}{\sqrt{T-s}}
\left\|
\mathscr{D}^{b_{\ell}} 
\mathbb{Y}^{n}(s) 
-
\mathscr{D}^{b_{\ell}} 
\mathbb{Y}^{n}(t_{n}) 
\right\|_{L^{2}(\Omega, \mathbb{R}^{d})}  
\left\|
f\big(\mathbb{Y}^{n}(t_{n}) \big) 
-  
f\big(\mathbb{Y}^{n}(s) \big)
\right\|_{L^{2}(\Omega, \mathbb{R}^{d})} \dd s
\\
& \leq
C
\left(
d + 
\|x_{0}\|^{2}_{L^{2}(\Omega, \mathbb{R}^{d})}
\right) h^{3/2},
\end{aligned}
\end{equation}
and
\begin{equation}
\label{equation:J2223-paper3}
\begin{aligned}
J^{(3)}_{2,2,2} 
&\leq
C
\sum_{n=N-N_{1}}^{N-2} 
\int_{t_{n}}^{t_{n+1}} 
\frac{1}{\sqrt{T-s}}
\left\|
\mathscr{D}^{b_{\ell}}
f
\big(
\mathbb{Y}^{n}(t_{n}) 
\big)
-
\mathscr{D}^{b_{\ell}}
f\big(
\mathbb{Y}^{n}(s) 
\big)
\right\|_{L^{2}(\Omega, \mathbb{R}^{d})}  
\left\|
\mathbb{Y}^{n}(s)
-
\mathbb{Y}^{n}(t_{n})
\right\|_{L^{2}(\Omega, \mathbb{R}^{d})} \dd s
\\
& \leq
C
\left(
d + 
\|x_{0}\|^{2}_{L^{2}(\Omega, \mathbb{R}^{d})}
\right) h.
\end{aligned}
\end{equation}
By 
combining \eqref{equation:J2221-paper3}, \eqref{equation:J2222-paper3}, \eqref{equation:J2223-paper3} with \eqref{equation:J1-lipschitz-paper3}, \eqref{equation:J2-remaining-paper3},
the proof is thus completed.
\end{proof}
Now, 
the proof of Theorem \ref{theorem:main-result-paper3-optimal} is straightforward due to \eqref{equation:error-decomposition-tv-paper3}.
}

\section{Numerical experiments} \label{section:Numerical-experiment}
In this section, some numerical results are performed to verify the theoretical analysis above.
We illustrate our finding via 
{\color{black}
the high-dimensional double-well model 
with the drift function being given
} 
by
\begin{equation} \label{equation:target-drift-paper3}
    f(x) = \alpha x - \beta \| x\|^{2} x, \quad \alpha, \beta >0, \ \forall x \in \mathbb{R}^{d}. 
\end{equation}
{\color{black}
For this particular drift,
Assumption \ref{assumption:globally-polynomial-growth-condition-paper3} is fulfilled with $\gamma=3$ and Assumptions \ref{assumption:contractivity-at-infinity-condition-paper3}, \ref{assumption:coercivity-condition-of-the-drift-paper3} can be also verified, as already done by Example \ref{eqn:doublewell}.
}

\textbf{\color{black}Numerical parameters of the convergence test.}
Let $\alpha=1$, $\beta = 4$, the initial condition 
{\color{black}
$X_{0}=
\bold{0} \in \mathbb{R}^{d}$.
}
We fix a terminal time $T=6$ and five different  stepsizes $h=2^{-5}, 2^{-6}, 2^{-7}, 2^{-8}, 2^{-9}$.
Here we focus on the convergence analysis with dimension $d\in \{6,10,50,100 \}$ with the  parameter of the PLMC algorithm \eqref{equation:numerical-scheme-PLMC-algorithm-paper3} being chosen as $\vartheta = 1$.
The empirical mean of $\mathbb{E}\left[\phi(X_{T}) \right]$
is estimated by a Monte Carlo approximation, involving 3,000 independent trajectories.

\textbf{Test functions.}
We here construct the  indicator function and step function as, for any $ x \in \mathbb{R}^{d}$.
\begin{equation} \label{equation:bounded-test-function-paper3}
\begin{aligned}
{\color{black}
\phi_{1}(x):=  \textbf{1}_{
\|x \| \in (0,1/2) \cup (3/2,2) 
\cup (5/2,3)
\cup (7/2,4)
}
}
\quad \
\text{and} \quad \
{\color{black}
\begin{split}
\phi_{2}(x):= \left\{
    \begin{array}{llll}
  \ \  0 \ \ \ \ \quad \ \| x \| \in [0, 1/2),
    \\
   \ \ 1 \ \ \ \ \quad \  \| x \|  \in [ 1/2, 1), \\
   \  1/2   \ \ \ \quad \| x \|  \in [ 1, 3/2), \\
  \   -1   \ \hspace{0.9em}  \quad \| x \|  \in [ 3/2, 2), \\
   \  1/4   \ \hspace{0.7em}  \quad \| x \|  \in [ 2, 5/2), \\
  \  1/3 \ \ \ \hspace{1.1em} \| x \|  \in [ 3, 7/2), \\
      -1/3   \ \ \ \hspace{0.6em}  \| x \|  \in [ 7/2, 4), \\
  -1/2   \ \ \ \hspace{0.6em}  \| x \|  \in [ 4, \infty). \\
 \end{array}\right.
 \end{split}
 }
 \end{aligned}
\end{equation}
Along this section, we consider the following four test functions
\begin{equation}
    \phi 
    \in 
    \{ \phi_{1}(x), \ e^{-\|x\|}, \ \phi_{2}(x), \ \arctan(\| x\|), x  \in \mathbb{R}^{d}\}.
\end{equation}
{\color{black}
Note that all test functions $\phi \in B_{b}(\mathbb{R}^{d})$ with $\| \phi\|_{0}=1$.
}


\textbf{Density test.}
Before proceeding further,
we would like to make sure that the choice of terminal time $T=6$ is suitable.
{\color{black}
To do so,
here we choose a convergent modified tamed Langevin Monte Carlo (MTLMC) algorithm proposed by \cite{neufeld2022non} as a reference,
\begin{equation} \label{equation:tamed-scheme-paper3}
 \overline{Y}_{n+1} 
= 
\overline{Y}_{n}
+ \dfrac{\alpha \overline{Y}_{n} 
- \beta \| \overline{Y}_{n}\|^{2} \overline{Y}_{n}}{ 
{\color{black}(1+h\|\overline{Y}_{n} \|^4)^{1/2}}
}h
+ \sqrt{2h}\xi_{n+1}, 
\quad \overline{Y}_{0} = X_{0},
\end{equation}
}
{\color{black}
%
The empirical distributions at $T=6$ for the PLMC algorithm $Y = (Y^{(1)}, \cdots, Y^{(10)})^T$ and the MTLMC algorithm $\overline{Y} = (\overline{Y}^{(1)}, \cdots, \overline{Y}^{(10)})^T$
for $d=10$ and a stepsize $h=2^{-13}$ are chosen as an example.
In Figure \ref{fig:pdf-paper3}, we plot the density curves of the first components $Y^{(1)}$ and $\overline{Y}^{(1)}$ of both algorithms.
}
As illustrated by Figure \ref{fig:pdf-paper3}, 
the normalised histogram plots of samples as well as the  marginal probability density curves generated by the PLMC algorithm \eqref{equation:numerical-scheme-PLMC-algorithm-paper3} and the MTLMC algorithm \eqref{equation:tamed-scheme-paper3} are very close, indicating that the choice of time $T=6$ is appropriate.
\begin{figure}[h]
\centering
\subfigure{
    \begin{minipage}[t]{0.45\textwidth}
    \centering
    \includegraphics[width=\textwidth]
      {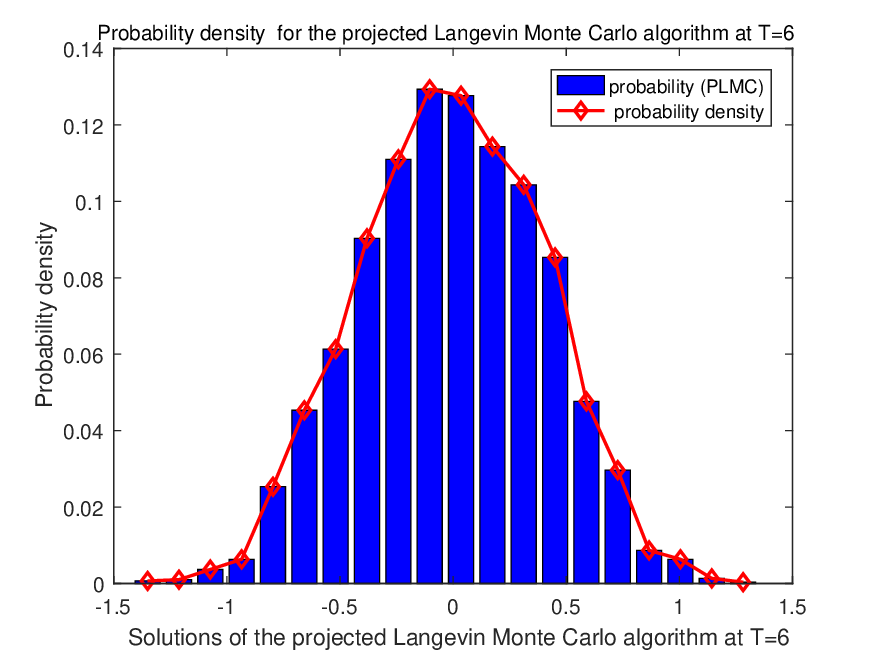}
    \end{minipage}
} 
\subfigure{
    \begin{minipage}[t]{0.45\textwidth}
    \centering
    \includegraphics[width=\textwidth]
      {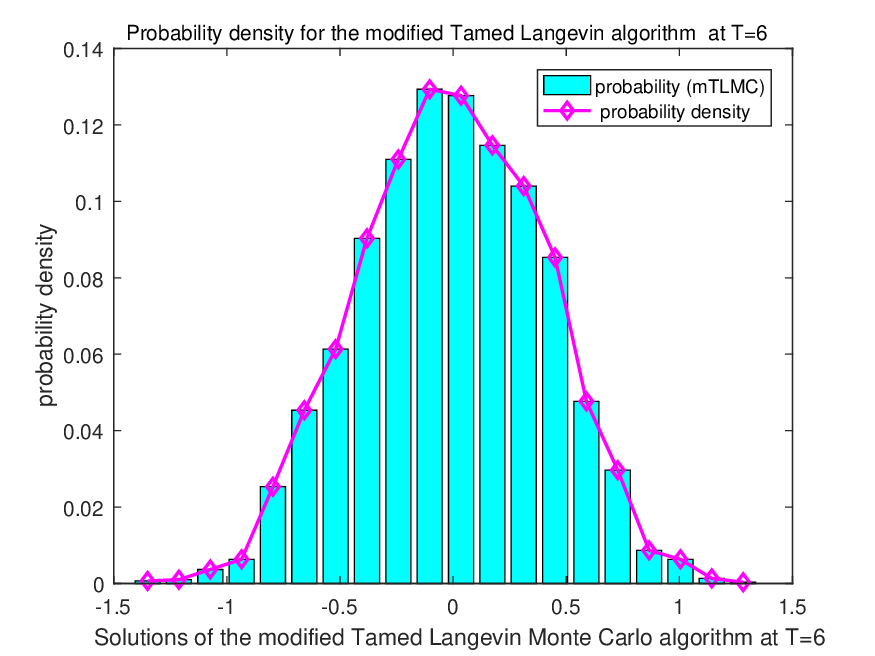}
    \end{minipage}
}
\caption{Probability density of the first component of  the double-well model.}
\label{fig:pdf-paper3}
\end{figure}

\textbf{Convergence test.}
For the convergence analysis, we run the PLMC algorithm \eqref{equation:numerical-scheme-PLMC-algorithm-paper3} for the Langevin SDEs with drift \eqref{equation:target-drift-paper3} by different stepsizes and dimensions $d\in \{6,10,50,100\}$ till $T=6$,
{
\color{black} i.e.
$T/h$ iterations for each step size $h$.
}
Further, 
the \textit{exact} solutions are identified as the  
{\color{black} numerical ones produced by the
PLMC algorithm \eqref{equation:numerical-scheme-PLMC-algorithm-paper3} 
}
using a fine stepsize $h_{ref}=2^{-13}$. 
{\color{black}
The reference lines of slope $0.5$ and $1$ are also presented,
depicted as two dashed lines, respectively.
}
From 
{\color{black}
Figures \ref{fig:errors-6d-10d-paper3}, \ref{fig:errors-50d-100d-paper3}, 
}
it is natural to obtain that the convergence rate under the total variation distance with $d \in \{6,10, 50,100 \}$ is of order $1$.
{\color{black}
In addition, from Table \ref{table:errors-6d-paper3} to Table \ref{table:errors-100d},
we present all errors with $d \in \{6,10, 50,100 \}$ between the target distribution and the law of the PLMC algorithm \eqref{equation:numerical-scheme-PLMC-algorithm-paper3} 
after $T/h$ iterations with $h \in \{ 2^{-5}, 2^{-6}, 2^{-7}, 2^{-8}, 2^{-9} \}$, $T=6$.
Approximate convergence rates for each test function are also provided.
}

{\color{black}
\begin{table}[htbp]
\centering
\setlength{\extrarowheight}{2.5pt}
\color{black}
\begin{tabular}{cccccc}
\hline
\multicolumn{6}{c}{Errors and convergence rates after $T/h$ iterations for $d=6$} \\ \hline
$h$\textbackslash $\phi$ & \multicolumn{1}{c}{$\phi_{1}(x)$} & $\exp(-\| x\|)$ & \multicolumn{1}{c}{$\phi_{2}(x)$} & $\arctan(\| x\|)$ & \multicolumn{1}{c}{Iterations}   \\ \hline
$2^{-5}$               & 3.13e-02           &   8.85e-04                & 3.83e-02        &   1.28e-03 &
192 \\ 
$2^{-6}$               & 1.33e-02           &   5.58e-04                & 1.68e-02        &   7.93e-04 
 & 384 \\ 
$2^{-7}$               & 6.00e-03           &   2.17e-04                & 8.00e-03        &   3.10e-04 
& 768 \\ 
$2^{-8}$               & 3.00e-03           &   1.36e-04                & 4.17e-03        &   1.92e-04  
& 1536\\ 
$2^{-9}$               & 1.33e-03           &   7.35e-05                & 1.83e-03        &   1.03e-04 
& 3072\\ \hline
\text{Order}               & 1.13           &   0.92               & 1.08        & \  0.93 
& \ \\ \hline
\end{tabular} 
\captionof{table}{
\color{black}Weak errors and convergence rates of the PLMC algorithm for the double well model when $d=6$.}
\label{table:errors-6d-paper3}
\end{table}
}

\begin{table}[htbp]
\centering
\setlength{\extrarowheight}{2.5pt}
\color{black}
\begin{tabular}{cccccc}
\hline
\multicolumn{6}{c}{Errors and convergence rates after $T/h$ iterations for $d=10$} \\ \hline
$h$\textbackslash $\phi$ & \multicolumn{1}{c}{$\phi_{1}(x)$} & $\exp(-\| x\|)$ & \multicolumn{1}{c}{$\phi_{2}(x)$} & $\arctan(\| x\|)$ & \multicolumn{1}{c}{Iterations}   \\ \hline
$2^{-5}$               & 6.17e-02           &   3.98e-03               & 8.58e-02        &   5.56e-03 &
192 \\ 
$2^{-6}$               & 2.70e-02           & 1.96e-03                & 3.68e-02        &  2.73e-03 
 & 384 \\ 
$2^{-7}$               & 1.53e-02           &  9.86e-04                & 2.15e-02        &   1.37e-03
& 768 \\ 
$2^{-8}$               & 6.33e-03           &   5.01e-04                & 9.00e-03        &   6.93e-04 
& 1536\\ 
$2^{-9}$               & 2.67e-03           &   2.36e-04                & 4.17e-03        &   3.27e-04 
& 3072\\ \hline
\text{Order}               & 1.12           & \ \ 1.01                & 1.08        & \  1.02 
& \ \\ \hline
\end{tabular} 
\captionof{table}{\color{black}
Weak errors and convergence rates of the PLMC algorithm for the double well model when $d=10$.}
\end{table}

\begin{figure}
    \centering
\subfigure{
    \begin{minipage}[t]{0.7\textwidth}
    \centering
    \includegraphics[width=\textwidth]
      {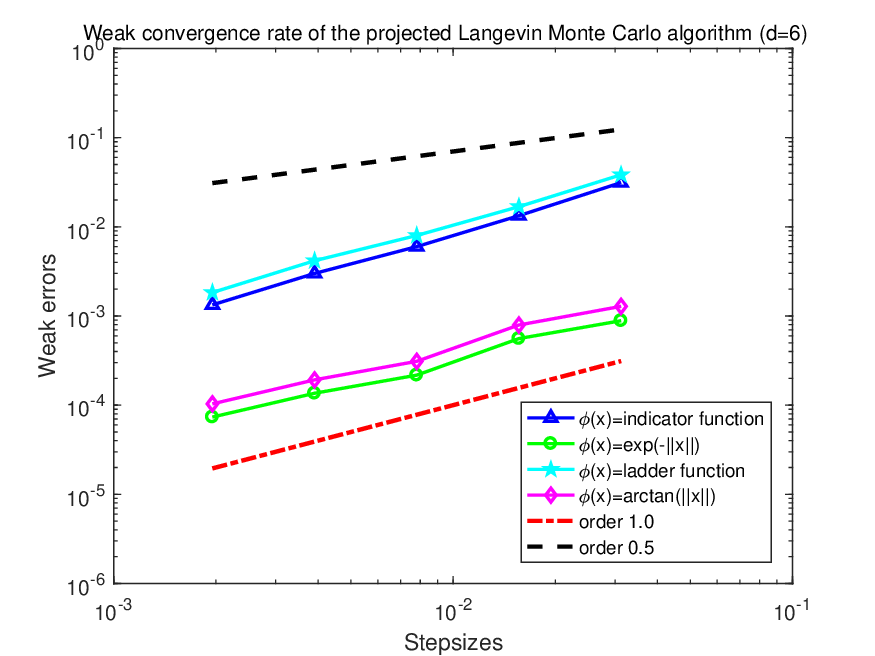}
    \end{minipage}
} 
\subfigure{
    \begin{minipage}[t]{0.7\textwidth}
    \centering
    \includegraphics[width=\textwidth]
      {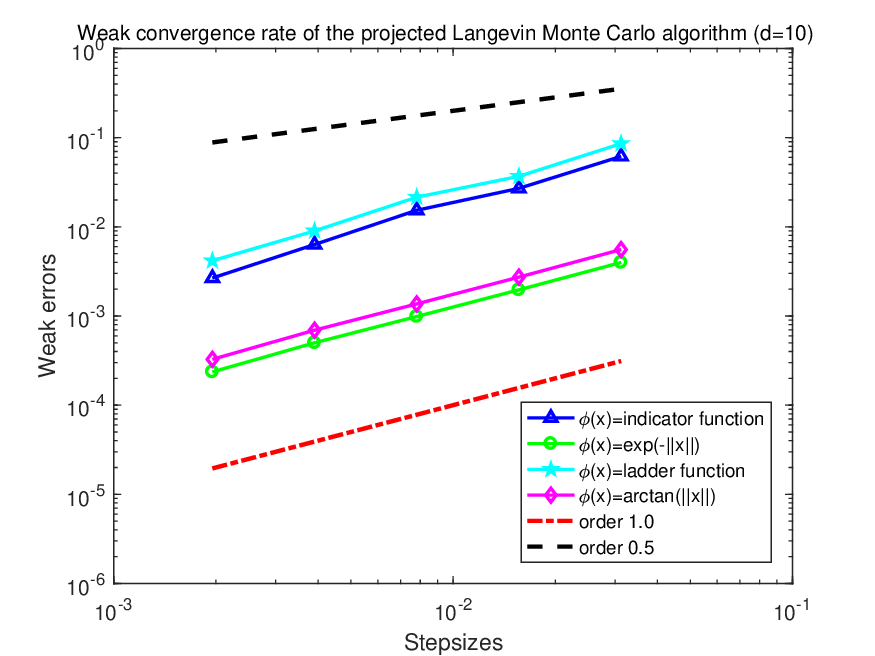}
    \end{minipage}
}
\caption{Weak convergence rates of PLMC algorithm of the double well model for $d=6$ (Top) and $d=10$ (Bottom).}
\label{fig:errors-6d-10d-paper3}
\end{figure}

\begin{table}[htbp]
\centering
\setlength{\extrarowheight}{2.5pt}
\color{black}
\begin{tabular}{cccccc}
\hline
\multicolumn{6}{c}{Errors and convergence rates after $T/h$ iterations for $d=50$} \\ \hline
$h$\textbackslash $\phi$ & \multicolumn{1}{c}{$\phi_{1}(x)$} & $\exp(-\| x\|)$ & \multicolumn{1}{c}{$\phi_{2}(x)$} & $\arctan(\| x\|)$ & \multicolumn{1}{c}{Iterations}   \\ \hline
$2^{-5}$               & 3.03e-01           &   1.66e-02                & 3.85e-01        &   2.48e-02 &
192 \\ 
$2^{-6}$               & 1.50e-01           &  7.04e-03                & 1.88e-01        &  1.04e-02 
 & 384 \\ 
$2^{-7}$               & 7.23e-02           &  3.33e-03                & 9.06e-02        &   4.91e-03
& 768 \\ 
$2^{-8}$               & 3.50e-02           &   1.61e-03                & 4.38e-02        &  2.37e-03 
& 1536\\ 
$2^{-9}$               & 1.97e-02           &   7.69e-04                & 2.47e-02        &   1.13e-03 
& 3072\\ \hline
\text{Order}               & 1.00          &   1.01                & 1.00        &   1.11
& \ \\ \hline
\end{tabular} 
\captionof{table}{\color{black}Weak errors and convergence rates of the PLMC algorithm for the double well model when $d=50$.}
\end{table}


\begin{table}[htbp]
\centering
\setlength{\extrarowheight}{2.5pt}
\color{black}
\begin{tabular}{cccccc}
\hline
\multicolumn{6}{c}{Errors and convergence rates after $T/h$ iterations for $d=100$} \\ \hline
$h$\textbackslash $\phi$ & \multicolumn{1}{c}{$\phi_{1}(x)$} & $\exp(-\| x\|)$ & \multicolumn{1}{c}{$\phi_{2}(x)$} & $\arctan(\| x\|)$ & \multicolumn{1}{c}{Iterations}   \\ \hline
$2^{-5}$               & 5.83e-01           &   2.65e-02                & 4.26e-01        &   4.36e-02 &
192 \\ 
$2^{-6}$               & 1.53e-01           &  9.37e-03                & 1.09e-02        &   1.51e-02 
 & 384 \\ 
$2^{-7}$               & 4.47e-02           &   4.29e-03                & 3.31e-02        &   6.83e-03
& 768 \\ 
$2^{-8}$               & 1.77e-02           &   2.07e-03                & 1.27e-02        &   3.29e-03 
& 1536\\ 
$2^{-9}$               & 8.33e-03           &   9.89e-04                & 6.33e-03        &  1.57e-03 
& 3072\\ \hline
\text{Order}               & 1.54           &  1.17               & 1.53        &  1.18
& \ \\ \hline
\end{tabular} 
\captionof{table}{\color{black}Weak errors and convergence rates of the PLMC algorithm for the double well model when $d=100$.}
\label{table:errors-100d}
\end{table}

\begin{figure}
    \centering
\subfigure{
    \begin{minipage}[t]{0.7\textwidth}
    \centering
    \includegraphics[width=\textwidth]
      {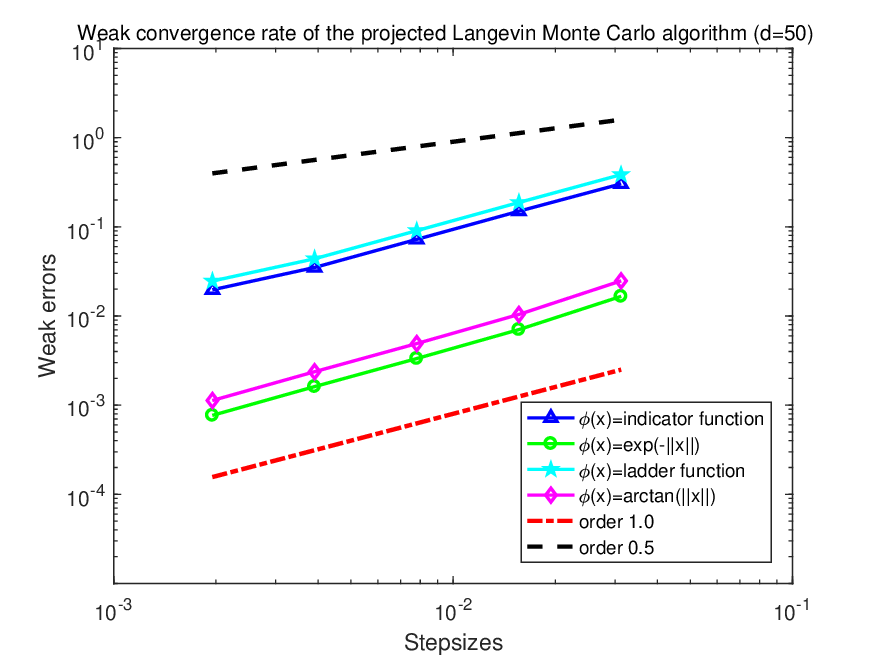}
    \end{minipage}
} 
\subfigure{
    \begin{minipage}[t]{0.7\textwidth}
    \centering
    \includegraphics[width=\textwidth]
      {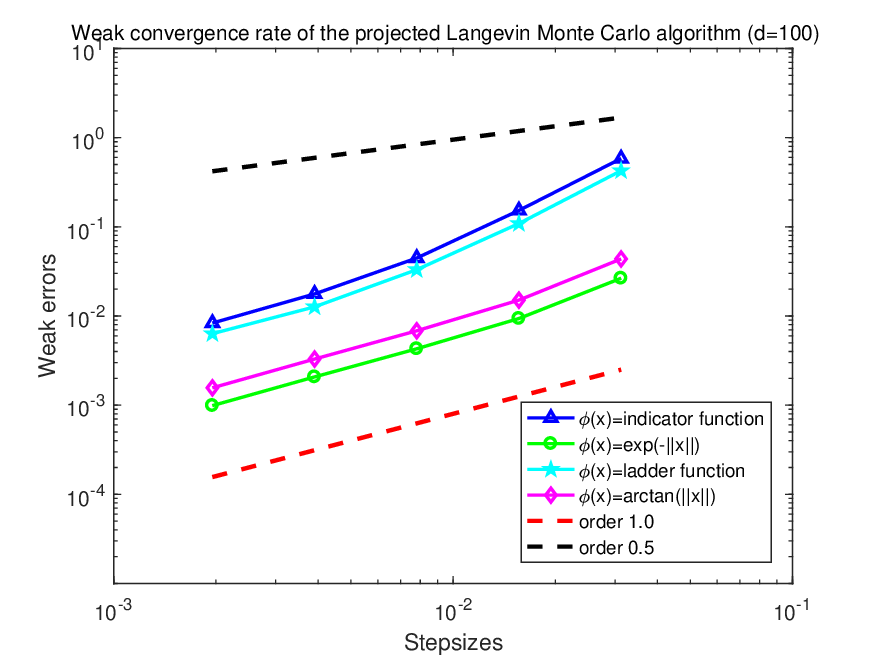}
    \end{minipage}
}
\caption{Weak convergence rates of PLMC algorithm of the double well model for $d=50$ (Top) and $d=100$ (Bottom).}
\label{fig:errors-50d-100d-paper3}
\end{figure}
{\color{black}
\textbf{Dimension Dependence.}
To study the dimension dependence of sampling errors,
we let $\alpha=1$, $\beta = 1$ in \eqref{equation:target-drift-paper3} with fixed  step size $h=2^{-4}$ 
and
run PLMC algorithm \eqref{equation:numerical-scheme-PLMC-algorithm-paper3} for 80 iterations.
All the errors and the approximate dependencies on dimensions with four test functions \eqref{equation:bounded-test-function-paper3} are shown in Table \ref{table:errors-d-dependence}.
Furthermore, we present these errors in Figure \ref{fig:d-dependence}, accompanied by  reference lines of $\mathcal{O}(d)$ and $\mathcal{O}(d^{3/2})$ for comparison.
In Figure \ref{fig:d-dependence} and Table \ref{table:errors-d-dependence},
one can observe a linear dependence of dimensions $\mathcal{O}(d)$ for the PLMC algorithm \eqref{equation:numerical-scheme-PLMC-algorithm-paper3},
which performs much better than the theoretical result $\mathcal{O}(d^{5})$ obtained in Theorem \ref{theorem:main-result-paper3} with $\gamma=3$.
It is indicated that the theoretical dimension dependence can be further improved even in the superlinear setting.
Nevertheless, this remains an open problem in theory and would be a direction for our future research.
}

\begin{table}[htbp]
\centering
\setlength{\extrarowheight}{2.5pt}
\color{black}
\begin{tabular}{ccccc}
\hline
\multicolumn{5}{c}{Errors and dimension dependence after 80 iterations } \\ \hline
$d$\textbackslash $\phi$ & $\phi_{1}(x)$ & $\exp(-\| x\|)$ & $\phi_{2}(x)$ & $\arctan(\| x\|)$  \\ \hline
$10$               & 4.20e-02           &   1.51e-03                & 3.53-02        &   1.86e-03  \\ 
$20$               & 1.20e-01           &  3.91e-03                & 8.01e-02        &   5.97e-03
 \\ 
$50$               & 2.48e-01           &   8.69e-03                & 1.44e-01        &   1.72e-02
 \\ 
$100$               & 4.79e-01           &   1.26e-02               & 3.99e-01        &   3.08e-02 
\\ \hline
\text{Order}               & 1.02           &  0.91               & 1.00       &  1.21
 \\ \hline
\end{tabular} 
\captionof{table}{\color{black}Weak errors and dimension dependence of the PLMC algorithm for the double well model.}
\label{table:errors-d-dependence}
\end{table}
\begin{figure}
    \centering
    \includegraphics[width=0.7\textwidth]
      {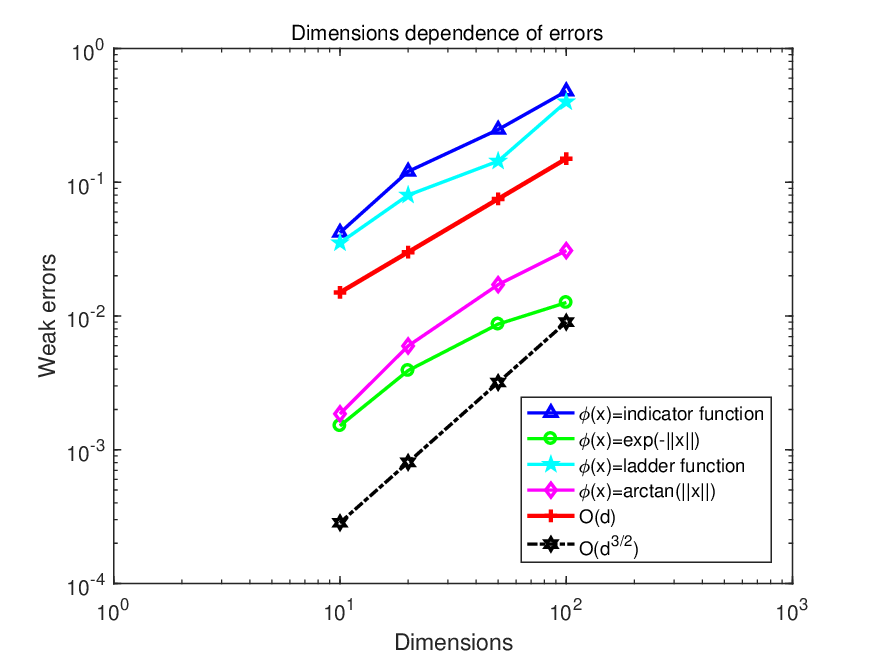}
\caption{Dimension dependence of weak errors of PLMC algorithm.}
\label{fig:d-dependence}
\end{figure}

\bibliographystyle{abbrv}
\bibliography{refer}

\appendix

\section{Proof of Lemmas in Section \ref{section:Preliminary-results-paper3}}
\subsection{Proof of Lemma \ref{lemma:Uniform-moment-bounds-of-the-Langevin-SDE-paper3}}
\label{proof-of-lemma:Uniform-moment-bounds-of-the-Langevin-SDE-paper3}
\begin{proof}[Proof of Lemma \ref{lemma:Uniform-moment-bounds-of-the-Langevin-SDE-paper3}]
{\color{black}
First we define a stopping time as 
\begin{equation} \label{equation:stopping-time}
\tau_{n} =
\inf \{ s \geq 0 : \| X_{s}  \| > n  \}.
\end{equation} 
}
{\color{black}
Applying the It\^o formula to 
$e^{cpt\land \tau_{n}}(\epsilon_{3}+\|X_{t\land \tau_{n}}\|^{2})^{p}$
}
for some constants 
{\color{black}
$c \in (0, 2a_{1})$,
}
$ \epsilon_{3}>0$ and using the Cauchy-Schwarz inequality show 
{\color{black}
\begin{equation} \label{equation: 1st step of the ito formula}
\begin{aligned}
 & e^{cpt\land \tau_{n}}
\left(
\epsilon_{3}
+\|X_{t\land \tau_{n}}\|^{2}
\right)^{p} \\
&\leq (
\epsilon_{3}
+\|x_{0}\|^{2}
)^{p} 
+ cp
\int_{0}^{t\land \tau_{n}} 
e^{cps} 
\left(
\epsilon_{3}+\|X_{s}\|^{2}
\right)^{p} 
\dd s 
+ 2p
\int_{0}^{t\land \tau_{n}} 
e^{cps} 
\left(
\epsilon_{3}
+\|X_{s}\|^{2}
\right)^{p-1} 
\left\langle 
X_{s}, f(X_{s}) 
\right\rangle \dd s\\
& \quad 
+ 2\sqrt{2} p
\int_{0}^{t\land \tau_{n}} 
e^{cps} 
\left(
\epsilon_{3}
+\|X_{s}\|^{2}
\right)^{p-1} 
\left\langle
X_{s},
\dd W_{s}  
\right\rangle
+ 2p(2p-1)d
\int_{0}^{t\land \tau_{n}} 
e^{cps} 
\left(
\epsilon_{3}
+\|X_{s}\|^{2}
\right)^{p-1} \dd s.
\end{aligned}
\end{equation}
}
{\color{black}
Taking expectations on the both sides of \eqref{equation: 1st step of the ito formula}, using Assumption \ref{assumption:coercivity-condition-of-the-drift-paper3} and letting $\epsilon_{3}\rightarrow 0^{+}$ yield
}
{\color{black}
\begin{equation} \label{eq:moment-estimate-in-sde}
\begin{aligned}
&
\mathbb{E}
\left[
e^{cp 
{
\color{black}(t\land \tau_{n})
}
}
\|
{\color{black}
X_{t\land \tau_{n}}
}
\|^{2p}
\right] \\
& \leq 
\mathbb{E}\left[\|x_{0}\|^{2p}\right] 
+ p \mathbb{E}\left[ 
\int_{0}^{\color{black}t\land \tau_{n}}
e^{cps}  
\Big[
(- 2a_{1} + c) \|X_{s}\|^{2p} 
+ (4pd-2d+2a_{2})\|X_{s}\|^{2p-2} 
\Big] \dd s 
\right],
\end{aligned}
\end{equation}
}
{
\color{black}
where the property of the It\^o integral was used that 
\begin{equation}
2\sqrt{2} p
\mathbb{E}
\left[
\int_{0}^{t\land \tau_{n}} 
e^{cps} 
\|X_{s}\|^{2(p-1)}
\left\langle
X_{s},
\dd W_{s}  
\right\rangle
\right] = 0.
\end{equation}
}
Using the Young inequality 
{\color{black}
\begin{equation}
ab^{2p-2} 
\leq 
\hat{\epsilon}b^{2p} 
+ \tfrac{1}{p}
\left(
\tfrac{\hat{\epsilon} p}{p-1}
\right)^{1-p}
a^{p},
\end{equation} 
}
for any $a,b >0$, 
{\color{black}
$\hat{\epsilon}>0$,
}
and $p \geq 1$ 
{\color{black}
yields}
{\color{black}
\begin{equation} \label{eq:young-inequality-in-moment-estimate-for-sde}
(
4pd-2d+2a_{2}
)
\|
X_{s}
\|^{2p-2} 
\leq 
\hat{\epsilon}
\|
X_{s}
\|^{2p}
+
\tfrac{1}{p}
\left(
\tfrac{\hat{\epsilon} p}{p-1}
\right)^{1-p}
(
4pd-2d+2a_{2}
)^{p}.
\end{equation}
}
{\color{black}
Choosing $c\in (0, 2a_{1})$ and 
$\hat{\epsilon} = 2a_{1}-c$, 
by \eqref{eq:young-inequality-in-moment-estimate-for-sde}
we obtain that the following estimate holds true 
\begin{equation}
\left(
-2a_{1}+c
\right)
\|X_{s}\|^{2p} 
+\big(
4pd-2d+2a_{2} 
\big)\|X_{s}\|^{2p-2}
\leq 
\tfrac{(
4p-2+2a_{2}
)^{p}}{p}
\left(
\tfrac{ p-1}{(2a_{1} - c)p}
\right)^{p-1}
d^{p}.
\end{equation}
As a result, we deduce that
\begin{equation}
\mathbb{E}
\left[
e^{cp \color{black} (t\land \tau_{n})}
\|
{ \color{black}
X_{t\land \tau_{n}}
}
\|^{2p}
\right] 
\leq 
\mathbb{E}
\left[
\|x_{0}\|^{2p}
\right]
+ 
\tfrac{(
4p-2+2a_{2}
)^{p}}{p}
\left(
\tfrac{ p-1}{(2a_{1} - c)p}
\right)^{p-1}
d^{p}
{\color{black}
\int_{0}^{t\land \tau_{n}} e^{cps} \dd s.
}
\end{equation}
{\color{black}
Due to the Fatou Lemma, letting $n \rightarrow \infty$ gives
}
\begin{equation}
\mathbb{E}\left[\|X_{t}\|^{2p}\right] 
\leq e^{-cpt}
\mathbb{E}\left[\|x_{0}\|^{2p}\right] 
+ 
\tfrac{(
4p-2+2a_{2}
)^{p}}{p}
\left(
\tfrac{ p-1}{(2a_{1} - c)p}
\right)^{p-1}
d^{p}.
\end{equation}
}
The proof is completed. 
\end{proof}

\section{Proof of lemmas in Section \ref{section:Kolmogorov-equation-and-regularization-estimates}}
\subsection{Proof of Lemma \ref{lemma:differentiability-of-solutions-paper3}}
\label{proof:lemma:differentiability-of-solutions-paper3}
\begin{proof}[Proof of Lemma \ref{lemma:differentiability-of-solutions-paper3}]
The existence and the uniqueness of the mean-square derivatives up to the second order can be derived owing to Remark \ref{remark: one-side Lipschitz condition of the drfit-paper3} 
{\color{black}
(see  \cite[Proposition 1.3.5]{cerrai2001second} or \cite[Appendix C]{zhao2024one}).
}
For simplicity, we denote
\begin{equation}\label{equation:notation-of-the-derivatives-paper3}
\eta^{v_{1}} (t,x):=\mathcal{D}X^{x}_{t} v_{1}, 
\quad 
\xi^{v_{1}, v_{2}}(t,x) :=  \mathcal{D}^{2}X^{x}_{t}  (v_{1}, v_{2}), 
\quad \forall x, v_{1}, v_{2}\in \mathbb{R}^{d}.
\end{equation}
\textbf{Part I: estimate of the first variation process}

It follows from Remark \ref{remark: one-side Lipschitz condition of the drfit-paper3} that, 
\begin{equation} \label{eq: enhanced lipschitz condition}
    \langle Df(x)y, y\rangle  \leq L \| y\|^{2}, \quad \forall x, y \in \mathbb{R}^{d}.
\end{equation}
Moreover, for the first variation process of SDE \eqref{eq:langevin-SODE}, one gets
\begin{equation} \label{eq:1st-variation-of-SODE}
\dd \eta^{v_{1}} (t,x)  
= Df( X^{x}_t)\eta^{v_{1}} (t,x) \, \dd t, 
\quad \eta^{v_{1}} (0) = v_{1}.
\end{equation}
Taking the temporal derivative of $\|\eta^{v_{1}}(t,x) \|^{2}$, we obtain that
\begin{equation}
\begin{aligned}
\dd \|\eta^{v_{1}}(t,x) \|^{2}
&= 2\left\langle 
\eta^{v_{1}}(t,x), Df(X^{x}_{t})\eta^{v_{1}}(t,x)
\right\rangle 
\leq 2L\|\eta^{v_{1}}(t,x) \|^{2},
\end{aligned}
\end{equation}
which leads to
\begin{equation} \label{equation:estimate-of-the-1st-variation-process-paper3}
\begin{aligned}
\|\eta^{v_{1}}(t,x) \|^{2} 
\leq e^{2Lt} \|v_{1}\|^{2}.
\end{aligned}
\end{equation}
\textbf{Part II: estimate of the second variation process}

Similarly, due to the variation approach, we have the second variation process with respect to SDEs \eqref{eq:langevin-SODE} as follows,
\begin{equation} \label{eq: 2nd variation of SODE}
\begin{aligned}
\dd \xi^{v_{1}, v_{2}} (t,x)  
= 
\Big(
Df( X^{x}_t)\xi^{
v_{1}, v_{2}
} (t,x) 
+ D^{2}f( X^{x}_t)\big(
\eta^{v_{1}} (t,x), \eta^{v_{2}} (t,x) 
\big)  \Big) \dd t, 
\quad \xi^{v_{1}, v_{2}} (0) = 0.
\end{aligned}
\end{equation}
Following the same argument as before and the Young inequality, we deduce that
\begin{equation} \label{equation:estimate-of-the-squared-second-variation-process-paper3}
\begin{aligned}
&\dd 
\|
\xi^{v_{1}, v_{2}} (t,x)  
\|^{2} \\
&= \Big(
2\left\langle
\xi^{
v_{1}, v_{2}
} (t,x) , 
Df( X^{x}_t) 
\xi^{
v_{1}, v_{2}
} (t,x) 
\right\rangle 
+  
2\left\langle 
\xi^{
v_{1}, 
v_{2}
} (t,x) ,  
D^{2}f( X^{x}_t) 
\big(
\eta^{v_{1}}(t,x), \eta^{v_{2}}(t,x)
\big) 
\right\rangle 
\Big) \dd t\\
& \leq 
(2L+1) 
\|
\xi^{
v_{1}, v_{2}
} (t,x)  
\|^{2} \dd t 
+ 
\left\|
D^{2}f( X^{x}_t) 
\big(
\eta^{v_{1}}(t,x), 
\eta^{v_{2}}(t,x)
\big)
\right\|^{2} 
\dd t.
\end{aligned}
\end{equation}
The key issue is to estimate $\left\|  D^{2}f( X^{x}_t) (\eta^{v_{1}}(t,x), \eta^{v_{2}}(t,x))\right\|_{L^{2}(\Omega, \mathbb{R}^{d})}$, where, by Assumption \ref{assumption:globally-polynomial-growth-condition-paper3}, the range of $\gamma$ deserves to be discussed carefully. 
{\color{black}
In the case $1\leq \gamma \leq 2$,  we employ Assumption \ref{assumption:globally-polynomial-growth-condition-paper3} and
\eqref{equation:estimate-of-the-1st-variation-process-paper3} to arrive at
}
%
%
{\color{black}
\begin{equation} 
\label{equation:D2f-gamma<2-paper3}
\begin{aligned}
\left\| 
D^{2}f( X^{x}_t) 
\big(
\eta^{v_{1}}(t,x)
, 
\eta^{v_{2}}(t,x)
\big)
\right\|_{L^{2}(\Omega, \mathbb{R}^{d})}
&\leq 
C 
\left\|
\eta^{v_{1}}(t,x) 
\right\| 
\cdot 
\left\|
\eta^{v_{2}}(t,x) 
\right\|   \\
&\leq Ce^{2Lt} 
\left\| 
v_{1}
\right\|
\cdot
\left\| 
v_{2} 
\right\|, \quad
\forall
x, v_{1}, v_{2} 
\in 
\mathbb{R}^{d}.
\end{aligned}
\end{equation}
}
{\color{black}
In the case $\gamma>2$, using 
Assumption \ref{assumption:globally-polynomial-growth-condition-paper3}, Lemma \ref{lemma:Uniform-moment-bounds-of-the-Langevin-SDE-paper3}
and  \eqref{equation:estimate-of-the-1st-variation-process-paper3} yields, 
\begin{equation}
\label{equation:D2f-gamma>2-paper3}
\begin{aligned}
&\left\| 
D^{2}f( X^{x}_t) 
\big(
\eta^{v_{1}}(t,x), 
\eta^{v_{2}}(t,x)
\big)
\right\|_{L^{2}(\Omega, \mathbb{R}^{d})} \\
&\leq C \big\|
\left(
1+\| X^{x}_{t} \| 
\right)^{\gamma-2} 
\big\|_{L^{2}(\Omega, \mathbb{R})}
\left\|
\eta^{v_{1}}(t,x) 
\right\| 
\cdot 
\left\|
\eta^{v_{2}}(t,x) 
\right\|   \\
& \leq C 
e^{2Lt}
\left(
d^{\gamma/2 -1 } + \|x\|^{\gamma-2}
\right)
\left\|
v_1 
\right\|
\cdot
\left\|
v_2
\right\|, 
\quad
\forall x, v_{1}, v_{2} \in \mathbb{R}^{d}.
\end{aligned}
\end{equation}
}
%
%
{\color{black}
Combining \eqref{equation:D2f-gamma<2-paper3} with \eqref{equation:D2f-gamma>2-paper3} above, we obtain,
\begin{equation}
\begin{aligned}
\mathbb{E}
\left[
\|
\xi^{v_{1}, 
v_{2}} (t,x)  
\|^{2} 
\right]
&\leq (2L+1)
\int_{0}^{t} 
\mathbb{E}
\left[
\|
\xi^{v_{1}, 
v_{2}} (s,x)  
\|^{2} 
\right] \dd s \\
&\quad
+
C
e^{4Lt}
\left[
\textbf{1}_{\gamma \in [1,2]}
+
\textbf{1}_{\gamma \in (2, \infty)} 
(d^{\gamma-2}+\|x\|^{2\gamma-4})
\right]
\left\|
v_1 
\right\|^{2}
\cdot
\left\|
v_2
\right\|^{2}, 
\quad
\forall x, v_{1}, v_{2} \in \mathbb{R}^{d},
\end{aligned}
\end{equation}
leading to
\begin{equation}
\begin{aligned}
\|
\xi^{
v_{1}, v_{2}
} (t,x) 
\|_{L^{2} (\Omega, \mathbb{R}^{d})}
\leq 
C e^{(3L+1/2)t}
\left[
\textbf{1}_{\gamma \in [1,2]}
+
\textbf{1}_{\gamma \in (2, \infty)} (d^{\gamma/2 -1}+\|x\|^{\gamma-2})
\right]
\left\|
v_1 
\right\|
\cdot
\left\|
v_2
\right\|, 
\
\forall x, v_{1}, v_{2} \in \mathbb{R}^{d},
\end{aligned}
\end{equation}
with the help of the Gronwall inequality.
The proof is thus completed.
}
%
\end{proof}

\section{Proof of Lemmas in Section \ref{section:Time-independent-error-analysis}}
\subsection{Proof of Lemma \ref{lemma:error-estimate-between-x-and-projected-x-paper3}}
\label{Proof of Lemma projected error}
\begin{proof}[Proof of Lemma \ref{lemma:error-estimate-between-x-and-projected-x-paper3}]
{\color{black}
For the case $\gamma > 1$, owing to \eqref{equation:projection-operator-paper3}
we have
\begin{equation}
\|
x-\mathscr{P}(x)
\|
=\|x-\mathscr{P}(x)\| 
\textbf{1}_{
\|x\| \geq \vartheta d^{1/2 \gamma} h^{-1/2 \gamma}
}.
\end{equation}
On the one hand,
one directly derives from Lemma \ref{lemma: useful estimate of PLMC algorithm-paper3} that
\begin{equation}\label{equation:estimate-1-paper3}
\|x-\mathscr{P}(x)\|
\leq
\|x \| + \|\mathscr{P}(x) \|
\leq 
2\|x\|.
\end{equation}
On the other hand,
it is straightforward to show that
\begin{equation}\label{equation:estimate-2-paper3}
\textbf{1}_{
\|x\| \geq \vartheta d^{1/2 \gamma} h^{-1/2 \gamma}
}
\leq
\left(
\dfrac{\|x\|}
{\vartheta d^{1/2\gamma}
h^{-1/2\gamma}
}
\right)^{4 \gamma}=
\vartheta^{-4 \gamma}
 d^{-2}
 h^{2} 
\|x\|^{4 \gamma}.
\end{equation}
Combining \eqref{equation:estimate-1-paper3} with \eqref{equation:estimate-2-paper3} yields,
for $\gamma>1$,
\begin{equation}
\|x-\mathscr{P}(x)\|
\leq
2\vartheta^{-4 \gamma}
 d^{-2}
 h^{2} 
\|x\|^{4 \gamma+1}, \quad
\forall x \in \mathbb{R}^{d}.
\end{equation}
To close the proof, we mention that it is evident to obtain $x - \mathscr{P}(x) = 0$ for $\gamma = 1$ due to \eqref{equation:projection-operator-paper3}.
}
\end{proof}

\section{Proof of {\color{black} Proposition \ref{theorem-mixing-time}}}
\label{proof-of-mixing-time}

\begin{proof}[\color{black}Proof of Proposition \ref{theorem-mixing-time}.]
Given tolerance $\epsilon \in (0,1)$,
{\color{black}
it follows from Theorem \ref{theorem:main-result-paper3} and Theorem \ref{theorem:main-result-paper3-optimal} that
choosing sufficiently large $k$ and sufficiently small $h$ such that
\begin{equation}\label{equation:error-decomposition-in-tv}
C_{\star}
\|\phi\|_{0}
e^{-c_{\star} kh}
\left(
1+
\mathbb{E}
\left[
\|x_{0}\|
\right]
\right) 
\leq
\dfrac{\epsilon}{2} \quad
\text{and}
\quad
Cd^{\max\{3\gamma/2, 2\gamma-1 \}}
h 
\big(
\textbf{1}_{\gamma=1} +
\textbf{1}_{\gamma>1}| \ln h |
\big)
\leq
\dfrac{\epsilon}{2},
\end{equation}
one then has $\|\Pi(Y^{x_{0}}_{k}) - \pi\|_{\text{TV}} \leq \epsilon$.
}
{\color{black}
Rearranging the first inequality of \eqref{equation:error-decomposition-in-tv} yields
\begin{equation}
e^{-c_{\star} kh}
\leq
\dfrac{\epsilon}{
2C_{\star}\|\phi\|_{0}
\left(
1+
\mathbb{E}
\left[
\|x_{0}\|
\right]
\right) 
},
\end{equation}
resulting in 
\begin{equation} \label{equation:constraint-on-k}
k \geq 
\dfrac{1}{c_{\star}h}  
\ln 
\left(
\dfrac{2C_{\star}\|\phi\|_{0}
\left(
1+
\mathbb{E}
\left[
\|x_{0}\|
\right]
\right) }{\epsilon}
\right).
\end{equation}
}
{\color{black}
Likewise, rearranging the second inequality of \eqref{equation:error-decomposition-in-tv} gives
\begin{equation} \label{equation:second-part}
\begin{aligned}
\dfrac{\tfrac{1}{h}}{\textbf{1}_{\gamma=1} +
\textbf{1}_{\gamma>1}\ln (\tfrac{1}{h})}
\geq 
\dfrac{  2C 
d^{
\max\{3\gamma/2, 2\gamma-1 \}
}
}{\epsilon}.
\end{aligned}
\end{equation}
}
{\color{black}
If $\gamma=1$, one directly derives 
$1/h 
\geq 
  2C 
d^{
3/2
}
/\epsilon$,
leading to $k = \mathcal{O}( \frac{d^{
3/2
}}
{\epsilon} \cdot \ln (\frac{1}{\epsilon}))$.
For $\gamma>1,$
note that $x / \ln x \geq \widehat{C}$ for $x \geq 1$, $\widehat{C}>0$ holds if $ x \geq 2\widehat{C}  \ln \widehat{C}$.
Therefore, \eqref{equation:second-part} can be satisfied on the condition
\begin{equation} \label{equation:restriction-on-h}
\dfrac{1}{h} 
\geq 
4
C 
\dfrac{d^{
\max\{3\gamma/2, 2\gamma-1 \}
}}{\epsilon}
\cdot  
\ln 
\left(2C\dfrac{d^{
\max\{3\gamma/2, 2\gamma-1 \}
}}{\epsilon}
\right).
\end{equation}
Plugging \eqref{equation:restriction-on-h} into \eqref{equation:constraint-on-k} yields
\begin{equation}
\begin{aligned}
    k 
    &\geq 
    4
    C \dfrac{d^{
    \max\{3\gamma/2, 2\gamma-1 \}
    }}{c_{\star}\epsilon} 
    \cdot 
    \ln \big(2C\dfrac{d^{\max\{3\gamma/2, 2\gamma-1 \}}}{\epsilon}\big)
    \cdot
    \ln 
\left(
\dfrac{2C_{\star}\|\phi\|_{0}
\left(
1+
\mathbb{E}
\left[
\|x_{0}\|
\right]
\right) }{\epsilon}
\right) \\
&= \mathcal{O}
\left(
\dfrac{d^{\max\{3\gamma/2, 2\gamma-1 \}}}{\epsilon} 
\cdot 
\ln \big(\dfrac{d}{\epsilon}\big) 
\cdot
\ln \big(\dfrac{1}{\epsilon}\big)
\right),
\end{aligned}
\end{equation}
which completes the proof.
}

\end{proof}

{\color{black}
\section{Proof  of auxiliary results in  Example \ref{eqn:doublewell}}
\label{proof-of-example}
\begin{proof}
It is straightforward to show that the function $f(x) = x(1 - \| x\|^{2})$ satisfies Assumptions \ref{assumption:globally-polynomial-growth-condition-paper3} with $\gamma=3$ and
\begin{equation}\label{equation:growth-of-f-paper3-of-double-well}
\left\|
f(x)  - f(y) 
\right\| 
\leq 
(
1
+
\|x\|
+
\|
y
\|
)^{2}  
\|
x-y
\|,
\quad 
\forall x, y\in \mathbb{R}^{d}.
\end{equation}
{\color{black}
Furthermore, one has
\begin{equation}
\left\langle
x, f(x)
\right\rangle
=\|x \|^{2} - \|x \|^{4} 
= -\|x \|^{2} + 2\|x \|^{2} - \|x \|^{4} 
 \leq -\|x \|^{2} + 1,
\end{equation}
which ensures that Assumption \ref{assumption:coercivity-condition-of-the-drift-paper3} is satisfied with $a_{1}=1$, $a_{2}=1$.
}
In the following,
we aim to illustrate how the drift $f$  satisfies Assumption \ref{assumption:contractivity-at-infinity-condition-paper3}.

Before proceeding further, we denote the notation $B(0, R)$ as the Euclidean ball centered at $0$ with radius $R>0$. As shown by \cite[Remark 2.3]{neufeld2022non}, 
the drift $f$ has the following \textit{contractivity outside a ball} property as
\begin{equation} \label{equation:convex-outside-ball}
\left\langle 
x-y,f(x)-f(y) 
\right\rangle 
\leq 
-\|x-y \|^{2}, 
\quad 
\forall x, y \notin B(0, \sqrt{2}).
\end{equation}
One may assume $x \in B(0, \sqrt{2})$ and $y \notin B(0, \sqrt{2})$ such that $\|x-y \| \geq \mathcal{R}$ 
without loss of generality. Setting $z = \lambda x + (1-\lambda)y$, then there exists some $\lambda \in (0, 1)$ such that $\| z\| \ = \sqrt{2}$.
It is also straightforward to derive that
\begin{equation}\label{equation:some-estiamtes-of-z-paper3}
z-y = \lambda(x-y), \quad \text{and} \quad
x - z = (1-\lambda)(x-y).
\end{equation}
Moreover, using the Cauchy-Schwarz inequality, \eqref{equation:growth-of-f-paper3-of-double-well}, 
\eqref{equation:convex-outside-ball} and \eqref{equation:some-estiamtes-of-z-paper3} yields,
\begin{equation}\label{equation:prior-estimates-of-inner-paper3}
\begin{aligned}
\left\langle 
x-y,f(x)-f(y) 
\right\rangle 
&=
\left\langle 
x-y,f(x)-f(z) 
\right\rangle 
+
\left\langle 
x-y,f(z)-f(y) 
\right\rangle  \\
& = 
\left\langle 
x-y,f(x)-f(z) 
\right\rangle 
+
\dfrac{1}{\lambda}
\left\langle 
z-y,f(z)-f(y) 
\right\rangle \\
& \leq
\left\langle 
x-y,f(x)-f(z) 
\right\rangle 
-
\lambda
\|x - y \|^2 \\
& \leq 
(
1
+
\|x\|
+
\|
z
\|
)^{2}  
\|
x-y
\| \cdot
\|x-z \|
-
\lambda
\|x - y \|^2 \\
& \leq
-
\big[
\lambda 
- 
(
2\sqrt{2}+1
)^{2}(1-\lambda)
\big]
\|x - y \|^2 \\
& = 
-
\big[
1
- 
(
4\sqrt{2}+10
)(1-\lambda)
\big]
\|x - y \|^2,
\end{aligned}
\end{equation}
where $\|x\|, \|z \| \leq \sqrt{2}$.
For the case that $\|x-y \| \geq \mathcal{R}$, according to 
\eqref{equation:some-estiamtes-of-z-paper3} and the fact that $\|x\|, \|z \| \leq \sqrt{2}$, we obtain
\begin{equation}
1-\lambda = \dfrac{\|x-z\|}{\|x-y\|} \leq \dfrac{\sqrt{2}}{\mathcal{R}}.
\end{equation}
Choosing $\mathcal{R} = 16+20\sqrt{2}$ and
plugging the estimate above into \eqref{equation:prior-estimates-of-inner-paper3} shows,
\begin{equation} \label{equation:estimate-bigger-than-R-paper3}
\left\langle 
x-y,f(x)-f(y) 
\right\rangle 
\leq
-\dfrac{\|x-y\|^{2}}{2}, \quad \
\text{for} \ \quad
\|x-y\| > 16+20\sqrt{2}.
\end{equation}
If $\|x-y\| \leq \mathcal{R} = 16+20\sqrt{2}$,
one directly obtains the following estimates by \eqref{equation:prior-estimates-of-inner-paper3}
and the fact that $1-\lambda \leq 1$
as
\begin{equation}\label{equation:estimate-smaller-than-R-paper3}
\begin{aligned}
\left\langle 
x-y,f(x)-f(y) 
\right\rangle 
\leq
(4\sqrt{2}+9)
\|x-y\|^2.
\end{aligned}
\end{equation}
Therefore, we derive from \eqref{equation:estimate-bigger-than-R-paper3} and \eqref{equation:estimate-smaller-than-R-paper3} that $\widetilde{a}_{1} = 4\sqrt{2} + 19/2$, $\widetilde{a}_{2} = 1/2$ and $\mathcal{R}=16+20\sqrt{2}$.
\end{proof}
}

\end{document}